\documentclass[11pt,english]{article}
\usepackage[T1]{fontenc}
\usepackage[latin9]{inputenc}
\usepackage{geometry}
\usepackage{amsmath}
\usepackage{amsthm}
\usepackage{amssymb,mathrsfs}
\usepackage{soul}
\usepackage{multirow}
\usepackage{makecell}
\usepackage{comment}
\usepackage[affil-it]{authblk}

\usepackage{hyperref}
\hypersetup{
    colorlinks=true,
    linkcolor=red,
    filecolor=magenta,      
    urlcolor=blue,
	citecolor=red,
    pdftitle={Multigrid Monte Carlo revisited},
    }

\makeatletter
\theoremstyle{plain}
\newtheorem{thm}{\protect\theoremname}[section]
\theoremstyle{plain}
\newtheorem{definition}[thm]{\protect\definitionname}
\newtheorem{assumption}{\protect\assumptionname}
\newtheorem{prop}[thm]{\protect\propositionname}
\theoremstyle{plain}
\newtheorem{lem}[thm]{\protect\lemmaname}
\theoremstyle{plain}
\newtheorem{cor}[thm]{\protect\corollaryname}
\newtheorem{example}[thm]{\protect\examplename}
\newtheorem{rem}[thm]{\protect\remarkname}
\usepackage{algorithm,algorithmicx,algpseudocode}
\usepackage{todonotes}
\usepackage{xcolor}
\definecolor{brightblue}{rgb}{0.75,0.75,0.9}
\definecolor{brightgray}{rgb}{0.9, 0.9, 0.9}
\definecolor{darkblue}{rgb}{0,0.08,0.4}
\definecolor{darkred}{rgb}{0.8,0.2, 0.2}
\definecolor{darkgreen}{rgb}{0, 0.6, 0}
\definecolor{darkorange}{rgb}{0.93, 0.57, 0.13}
\definecolor{green(html/cssgreen)}{rgb}{0.0, 0.5, 0.0}
\numberwithin{equation}{section}

\newcommand{\figdir}{./}

\newcommand*{\id}{I}
\newcommand*{\WCov}{\mathrm{K}}

\newcommand{\rev}[1]{#1}
\newcommand{\jg}[1]{#1}

\definecolor{darkgreen}{RGB}{30,130,80}

\usepackage[normalem]{ulem}
\newcommand*{\ykcst}[1]{\relax\ifmmode\text{\textcolor{darkgreen}{\sout{\ensuremath{#1}}}}\else\textcolor{darkgreen}{\sout{#1}}\fi}

\newcommand*{\ecst}[1]{\relax\ifmmode\text{\textcolor{darkorange}{\sout{\ensuremath{#1}}}}\else\textcolor{darkorange}{\sout{#1}}\fi}

\usepackage[normalem]{ulem}
\newcommand*{\rscst}[1]{\relax\ifmmode\text{\textcolor{darkblue}{\sout{\ensuremath{#1}}}}\else\textcolor{darkblue}{\sout{#1}}\fi}

\usepackage{footnote}

\makeatother

\usepackage{babel}
\usepackage[style=numeric,sorting=none,maxnames=9,date=year,isbn=false,url=false,doi=false,eprint=false,giveninits=true]{biblatex}
\providecommand{\definitionname}{Definition}
\providecommand{\assumptionname}{Assumption}
\providecommand{\corollaryname}{Corollary}
\providecommand{\lemmaname}{Lemma}
\providecommand{\propositionname}{Proposition}
\providecommand{\theoremname}{Theorem}
\providecommand{\examplename}{Example}
\providecommand{\remarkname}{Remark}

\usepackage{subfiles}
\begin{document}
\title{Multigrid Monte Carlo Revisited: Theory and Bayesian Inference}
\author[1]{Yoshihito Kazashi}
\author[2]{Eike~M\"{u}ller}
\author[3,*]{Robert Scheichl}
\affil[1]{Department of Mathematics, University of Manchester, Manchester
M13 9PL, UK}
\affil[2]{Department of Mathematical Sciences, University of Bath, Bath BA2 7AY, UK}
\affil[3]{Institute for Mathematics, Heidelberg University, 69120  Heidelberg, Germany}
\affil[*]{corresponding author}
\maketitle
\begin{abstract}
	\noindent
	Gaussian random fields play an important role in many areas of science and engineering. In practice, they are often simulated by sampling from a high-dimensional multivariate normal distribution, which arises from the discretisation of a suitable precision operator. Existing methods such as Cholesky factorization and Gibbs sampling become prohibitively expensive on fine meshes due to their high computational cost.
	\rev{In this work, we revisit the Multigrid Monte Carlo (MGMC) algorithm developed by Goodman \& Sokal (\emph{Physical Review D} 40.6, 1989) in the quantum physics context. While the authors of this paper conclude that MGMC does not overcome critical slowing down in simulations of field theories near phase transitions, we demonstrate here that it has the potential to significantly accelerate sampling in spatial statistics. The class of Gaussian Random Fields we consider includes those with Mat\'{e}rn covariance, but is more general in that it also allows for non-stationary covariance functions. To show that MGMC can overcome the limitation of existing methods, we establish a grid-size-independent convergence theory based on the link between linear solvers and samplers for multivariate normal distributions, drawing on standard multigrid convergence arguments.}
	We then apply this theory to linear Bayesian inverse problems. This application is achieved by extending the standard multigrid theory to operators with a low-rank perturbation. 
	Moreover, we develop a novel bespoke random smoother which takes care of the low-rank updates that arise in constructing posterior moments. In particular, we prove that Multigrid Monte Carlo is algorithmically optimal in the limit of the grid-size going to zero. Numerical results support our theory, demonstrating that Multigrid Monte Carlo can be significantly more efficient than alternative methods when applied in a Bayesian setting.
\end{abstract}
\textbf{MSC classes:}
60J22, %
60G60, %
62F15, %
65C05, %
65N55  %
\begin{center}
Communicated by Tony Leli\`{e}vre
\end{center}
\section{Introduction}
The approximate simulation of Gaussian random fields plays a pivotal role in a large number of research areas, such as quantum physics \cite{goodman1989multigrid,edwards1992multi}, spatial statistics \cite{LindgrenEtAl.F_2011_ExplicitLinkGaussian}, additive manufacturing \cite{ZhangEtAl.H_2021_StochasticModelingGeometrical}, cosmology \cite{Marinucci.D_2011_book}, natural language processing \cite{inoue-etal-2022-infinite}, public health \cite{MannsethEtAl.J_2021_VariationUseCaesarean}, geosciences \cite{LiuEtAl.Y_2019_AdvancesGaussianRandom}, or uncertainty quantification in engineering applications \cite{Soize.C_2017_UncertaintyQuantificationAccelerated,VasileEtAl.M_2021_AdvancesUncertaintyQuantification}. In theory, it is a well understood and solved problem, but in practice the efficiency and performance of traditional sampling procedures degenerates quickly when the random field is discretised on a grid (lattice) with spatial resolution going to zero. Moreover, most existing algorithms do not scale well on large-scale parallel computers.  Direct approaches based on factorisations of the covariance or precision matrix, e.g., \cite{rue2001fast,rue2005gaussian}, run into memory problems and their cost per sample scales in general like $\mathcal{O}(n^{1+\zeta})$ with respect to the number of grid points $n$ where $\zeta>0$ is a positive constant; the computational complexity can only be reduced to $\mathcal{O}(n\log(n))$ if the matrix is structured and FFT-based approaches are applicable. \rev{While in exact arithmetic the Cholesky factorisation allows the generation of independent samples from the target distribution, this is not the case in finite precision floating point arithmetic: for very large problems the ill-conditioning of the precision matrix can lead to the accumulation of rounding errors and the generated samples might in fact be drawn from a slightly biased distribution.}

Additional complications arise in the case where the Gaussian distribution is conditioned on data, as in (linear) Bayesian inverse problems, or when the spectrum of the covariance operator decays slowly. On the other hand, stationary, iterative approaches, such as random-walk Metropolis-Hastings \cite{metropolis1953,hastings1970}, preconditoned Crank--Nicholson \cite{cotter2013mcmc}, Langevin-based samplers \cite{roberts1996exponential} or Hamiltonian Monte Carlo \cite{duane1987216} become extremely inefficient as the grid size goes to zero. In particular, the convergence rate and the integrated autocorrelation time of the resulting Markov chain degenerate, leading to poor effective sample sizes. A lot of work has gone into this problem, but it is still a topic of  ongoing research \cite{LindgrenEtAl.F_2011_ExplicitLinkGaussian,LiuEtAl.Y_2019_AdvancesGaussianRandom,LangEtAl.A_2011_FastSimulationGaussian,Croci.M_etal_2018_EfficientWhiteNoise,FeischlEtAl.M_2018_FastRandomField, BachmayrEtAl.M_2020_UnifiedAnalysisPeriodizationBased,BolinEtAl.D_2020_RationalSPDEApproach,KressnerEtAl.D_2020_CertifiedFastComputations,LindgrenEtAl.F_2022_SPDEApproachGaussian}.
\paragraph{Connection between samplers and solvers.}
One of the most basic stationary iterative methods is the Gibbs sampler, see e.g. \cite{RobertEtAl.CP_2004_MonteCarloStatistical,GelmanEtAl.A_2013_BayesianDataAnalysis,RubinsteinEtAl.RY_2016_SimulationMonteCarlo}. The problem of this standard method is its slow convergence to the target distribution and the strong correlation of the resulting samples, both of which become substantially worse as the resolution increases. Already in the late 1980s, Goodman, Sokal and their collaborators wrote a series of papers aimed at accelerating the Gibbs sampler using multigrid ideas, leading to extensive research activity in the quantum physics community \cite{goodman1989multigrid,edwards1992multi,edwards1991multi,mana1996dynamic,mendes1996multi}. The key observation in the seminal paper by Goodman and Sokal \cite{goodman1989multigrid} is the connection of random samplers, such as the Gibbs method, to iterative solvers for linear systems. In particular, the Gibbs sampler for generating Gaussian samples $x \sim\mathcal{N}(\overline{x},A^{-1})$ with mean $\overline{x}$ and covariance matrix $A^{-1}$ can be written as an iteration identical, except for an additive noise component, to the Gauss--Seidel iteration for solving the linear system $A\overline{x}=f$ to obtain $\overline{x}=A^{-1}f$; for more details see \cite{goodman1989multigrid}, as well as \cite{Sokal1997,Fox.C_Parker_2017_AcceleratedGibbsSampling}. Based on this observation, they proposed the so-called \emph{Multigrid Monte Carlo (MGMC) method}---a random analogue of the multigrid method for solving discretised partial differential equations (PDEs). A careful analysis of the MGMC method is the focus of this paper. We assume that we have a positive definite operator $\mathcal A: H \to H$ on some function space $H$ on $\mathbb{R}^d$, for example $\mathcal{A} = (-\Delta+\kappa^2 \id)^\alpha$ with positive integer $\alpha$ where $\Delta$ is the Laplace operator. A sufficiently accurate discretisation of $\mathcal A$ on some grid in $\mathbb{R}^d$ results in a possibly very large matrix $A \in \mathbb{R}^{n \times n}$. The goal of this paper is to show how to efficiently generate samples from $\mathcal{N}(\overline{x},A^{-1})$ for $n\rightarrow\infty$. \rev{It should be pointed out that a discussion of the (rate of) convergence of the finite dimensional distribution $\mathcal{N}(\overline{x},A^{-1})$ to the corresponding inifinite dimensional Gaussian Random Field is beyond the scope of this paper, and we refer the reader to \cite{LindgrenEtAl.F_2022_SPDEApproachGaussian} and references therein for further details.} We provide a rigorous theoretical justification for the optimal scaling of our method for large scale problems: the cost for generating an independent sample grows linearly with the problem size and the sampler is optimal in the continuum limit. Our analysis includes the important setting of a Gaussian random field conditioned on noisy data via a Bayesian approach; here, $\mathcal A$ is a finite-rank perturbation of some differential operator and the quantity of interest a functional of the Gaussian random field. Somewhat surprisingly, MGMC has not yet been considered in this context.\vspace{-1ex}
\paragraph{Multigrid Monte Carlo.}
Just like the deterministic multigrid linear solver, the MGMC method is a recursive algorithm built from (random) smoothing iterations and coarse-level updates. In the deterministic multigrid method, coarse-level updates accelerate the convergence of the iteration for solving $A \overline{x}=f$ by reducing the low-frequency components in the error; in the MGMC, coarse-level updates accelerate convergence of the moments, i.e., convergence of the mean and the covariance to $\overline{x}=A^{-1}f$ and $A^{-1}$, respectively,  again by targeting the low-frequency components in the distribution. The smoothing iteration -- in the form of a Gibbs sweep across the grid -- is realised by adding an additive noise component to the usual deterministic Gauss-Seidel iteration. This noise is zero-mean and has a covariance suitably chosen to achieve convergence and to leave the target distribution invariant. Thus, upon taking the expectation, the mean of the MGMC iterates is identical to the iterates in deterministic multigrid for solving linear systems.

\rev{While the focus of this paper is the application of MGMC in a geometric setting which naturally induces a hierarchy of nested function spaces, the method is more general and can also be applied in a purely algebraic sense. This extension to an algebraic multigrid Monte Carlo method is beyond the scope of this paper and the topic of a subsequent publication, based on preliminary work in \cite{friess2024}.}\vspace{-1ex}
\paragraph{Key achievements.}
The main objective of this paper is to provide theoretical support for the reported efficacy of the MGMC method, which has not been comprehensively addressed in the literature even though the method was invented 35 years ago. The key contributions are as follows:
\begin{enumerate}
	\item A grid-size-independent convergence theory for the MGMC \rev{sampler}. %
	The analysis shows that the first two moments \rev{of the samples}, which fully characterise the Gaussian distribution, converge exponentially at a uniform rate to the target moments.
	\item An extension of the MGMC algorithm, as well as its convergence theory,
	to the important situation of sampling Gaussian random fields conditioned on noisy data, i.e., to linear Bayesian inverse problems.
	\item A corollary on the exponential decay of the autocorrelations of the computed samples, again with grid-size independent rate.
	\item \rev{A corollary on the root mean squared convergence of a Monte Carlo estimator that uses the samples generated by the MGMC algorithm to predict some quantity of interest.}
	\item A corollary on the optimal complexity of MGMC \rev{updates to produce samples, whose mean and covariance converge to those} of the limiting infinite-dimensional Gaussian random field.
	\item A numerical investigation demonstrating the grid independent convergence and the efficiency of the MGMC sampler for a set of representative model problems.
\end{enumerate}
In contrast to stationary iterative methods on a single grid, such as the Gibbs sampler, the convergence rate of MGMC does not degenerate as the grid size tends to zero. The convergence analysis follows Fox and Parker \cite{Fox.C_Parker_2017_AcceleratedGibbsSampling}. In that paper, they demonstrate that the convergence of the standard component-sweep Gibbs sampler is (in a suitable sense) equivalent to the deterministic (lexicographical) Gauss--Seidel iterative method, and that any matrix splitting used to generate a deterministic relaxation scheme also induces a stochastic relaxation that is a generalised Gibbs sampler; see also \cite{norton2016tuning} where this analysis is extended also to non-Gaussian targets and to the popular Langevin and hybrid (or Hamiltonian) Monte Carlo samplers. Here, we extend the analysis in \cite{Fox.C_Parker_2017_AcceleratedGibbsSampling} to MGMC. In doing so, we need to go beyond the derivation in \cite{goodman1989multigrid} 
and beyond the standard multigrid theory (as presented for example in \cite{Hackbusch.W_2016_IterativeSolutionLarge}) which provides convergence of the mean. Extra work is necessary to obtain the convergence of the covariance and to extend the analysis to the conditional case in a Bayesian setting. In particular, for linear functionals of the random field that are bounded in $H$, the crucial approximation and smoothing properties for the conditional case can be reduced to the corresponding properties in the unconditional case.\vspace{-1ex}
\paragraph{Relationship to previous work.}
The link between stationary iterative methods from numercial linear algebra and generalised Gibbs samplers was first noted and exploited by Adler and Whitmer \cite{adler1981over,whitmer1984over} and then by Goodman and Sokal \cite{goodman1989multigrid} in the context of MGMC. 
\jg{In \cite{amit1991comparing} the convergence of random sampling schemes%
	\footnote{\jg{In \cite{amit1991comparing}, a ``deterministic'' update rule sweeps through the state vector in a fixed order and updates local values by drawing from a random distribution. In contrast, ``stochastic'' updates process the coordinates in random order.}}
	is analysed and bounds for the rate of convergence are derived for Gaussian probability distributions. The authors of \cite{amit1991comparing} point out that the key quantity in the analysis is the error-propagation matrix of the Gauss-Seidel iteration. The work in \cite{Amit1991,Amit1996} further extends the analysis presented in \cite{amit1991comparing},} 
      \jg{while \cite{geman1984stochastic,geman1993stochastic} discuss stochastic relaxation methods and the Gibbs sampler in the context of image processing. However, as far as we can tell these papers do not explicitly point out the connection between Gibbs samplers and  relaxation methods for the solution of linear systems. As \cite{geman1984stochastic} remarks, a random sampler that only allows transitions to states with lower energy will get stuck in local minima.
Of course, for a multivariate normal distribution this implies convergence to the most likely state, which is the solution of a linear system.}

\jg{The link between iterative methods and samplers was formalised in a mathematically abstract way in \cite{Fox.C_Parker_2017_AcceleratedGibbsSampling}. They generalised the convergence analysis to any sampler that corresponds to a general matrix splitting and applied the accelerated sampler to linear Bayesian inverse problems.}
A key observation in \cite{Fox.C_Parker_2017_AcceleratedGibbsSampling} is that stationary iterative methods, such as Gauss-Seidel, might have been investigated in the 1950s as linear solvers, but their convergence is too slow to be of any practial interest for solving discretised PDEs in current days. The often purported slow convergence of MCMC methods based on stationary iterative procedures such as the Gibbs sampler thus follows directly from said equivalence between stationary iterative solvers and generalised Gibbs samplers. Parker and Fox then go on to note that \emph{``the last fifty years has seen an explosion of theoretical results and  algorithmic development that have made linear solvers faster and more efficient, so that for large problems, stationary methods are used as preconditioners at best, while the method of preconditioned conjugate gradients, GMRES, multigrid, or fast-multipole methods are the current state-of-the-art for solving linear systems in a finite number of steps \cite{saad2000iterative}''}. In \cite{Fox.C_Parker_2017_AcceleratedGibbsSampling} and the earlier papers \cite{parker2012sampling,fox2014convergence}, they exploit this link and demonstrate how to substantially speed up the stationary iterations derived from symmetric splittings by polynomial acceleration, particularly Chebyshev acceleration or Krylov-type methods such as conjugate gradients. However, the authors of \cite{Fox.C_Parker_2017_AcceleratedGibbsSampling} do not go on to analyse the MGMC method by Goodman and Sokal or to extend it to the linear Bayesian setting that we consider here.

A range of other methods have been proposed for sampling from multivariate normal distributions that arise from the discretisation of Gaussian random fields; e.g. \cite{rue2001fast,rue2005gaussian} describe several optimisations of the Cholesky sampler for problems that are formulated on a graph. A suitable reordering of the precision matrix results in a reduced bandwidth and thus a much better computational complexity. As the authors show, the method can be further optimised by recursively sampling from conditional distributions that arise from partitioning the graph via a divide-and-conquer approach; this requires the solution of sparse linear systems for which standard methods can be employed. 
\jg{The work \cite{Atzberger2010} discusses the generation of Gaussian Random Fields for the simulation of noise in fluid-structure interactions on adaptive meshes for which sampling via the Fast Fourier Transformation is not applicable. They employ a variant of MGMC that is based on Fast Adaptive Composite Mesh Multigrid \cite{mccormick1986fast}. Related to this, the authors of \cite{Plunkett2014} apply MGMC to a finite element discretisation on an unstructured mesh for a very similar physical problem.}

The precision operator $\mathcal{A}$ and the dimension $d$ of the problem determine the roughness of the underlying Gaussian Random Field. The commonly used operator $\mathcal{A}=(-\Delta+\kappa^2\id)^\alpha$ corresponds to a class of Matern fields with $\nu=\alpha - d/2$ where $\nu$ determines the mean-square differentiability of Gaussian Random Field \cite{LindgrenEtAl.F_2011_ExplicitLinkGaussian}. Although the discretised problem is well-defined for a finite lattice spacing, for $\nu\le 0$ the field is so rough that the continuum limit of the probability distributions only exists in the weak sense, i.e. when evaluated against bounded linear operators. This includes the important case $\mathcal{A}=-\Delta+\kappa^2\id$ in $d>1$ dimensions. In contrast to other methods such as the SPDE approach \cite{LindgrenEtAl.F_2011_ExplicitLinkGaussian,LindgrenEtAl.F_2022_SPDEApproachGaussian,Croci.M_etal_2018_EfficientWhiteNoise}, this does not pose any problems for MGMC. To see this, observe that both MGMC and the SPDE approach work with the precision matrix $A$. This leads to good computational properties when $A$ is sparse, and this naturally follows if the matrix is given by a discretised differential operator of the above form. However, the SPDE approach inverts the square-root $A^{1/2}$ of the precision matrix, which corresponds to the operator $\mathcal{A}^{1/2}=(-\Delta+\kappa^2\id)^{\alpha/2}$. If $\alpha$ is an odd integer, the SPDE approach will require the inversion of a fractional power of a sparse matrix, which is not sparse itself. This problem does not arise for MGMC which only uses $A$ and which can therefore readily be used to sample fields for $\mathcal{A}=-\Delta+\kappa^2\id$. More generally, for arbitrary precision operators MGMC is therefore able to also cover the case where the square-root of the precision matrix is non-sparse even though the precision matrix itself is sparse. \rev{In other words, MGMC can sample from distributions with covariance operator $(-\Delta+\kappa^2 I)^{-1}$, while additional modifications are necessary to extend the SPDE approach to this case  (see e.g. \cite{bolin2018weak}). 
It is also not obvious how to extend the SPDE approach to the Bayesian setting, where the precision operator is a low-rank perturbation of a differential operator. As we discuss in detail below, MGMC can be readily adapted to this case. Finally, even when the SPDE approach is applicable, it requires the inversion of a large, ill-conditioned matrix. The most efficient, versatile and parallelisable class of solvers for such matrices that also work on general domains and for non-stationary distributions are multigrid methods. However, these will require several iterations to converge, each of which will incur roughly the same cost as a single MGMC update. Thus, while a detailed comparison is beyond the scope of this paper, MGMC is in general expected to be competitive with SPDE methods.}

In \cite{harbrecht2012low} a low-rank truncation of the Cholesky factorisation is used to approximately sample from Gaussian random fields. Although the method introduces an approximation, the resulting error can be controlled systematically and the method can be more efficient than the naive Cholesky approach since only a finite number of eigenmodes have to be included. However, the efficiency of the method depends on the spectrum of the covariance operator, and in particular the sufficiently rapid decay of the eigenvalues. We have found that for $\mathcal{A}=-\Delta+\kappa^2\id$ in $d>1$ dimensions the number of required eigenmodes grows too rapidly for the method to be efficient. Again, MGMC is able to cover this case without any problems.

Another advantage of the MGMC approach is that it can be readily parallelised on distributed memory machines by using well established domain decomposition techniques for the parallel implementation of multigrid solvers \cite{chow2006survey}. For example, the grid traversal in the Gibbs sampler that is used at every level of the grid hierarchy, needs to be performed in a red-black ordering and well-known strategies such as coarse grid aggregation \cite{blatt2012massively} can be used on the coarser levels to take into account the fact that the number of processing units can exceed the number of unknowns. Sampling from posterior distributions in the linear Bayesian setting that we consider here, requires parallel sparse matrix vector products and scatter operations of the form $y = a (a^\top x)$ for vectors $a,x,y$ and again these can be implemented using well established techniques in scientific computing. Although a parallel implementation of MGMC is beyond the scope of this paper, we have every reason to believe that it will show the same excellent parallel scalability on large distributed memory machines as multigrid solvers \cite{baker2012scaling,blatt2012massively,kohl2022textbook}.\vspace{-1ex}
\paragraph{Structure.}
The rest of the paper is organised as follows. In Section~\ref{sec:problem_setting} we outline the linear Bayesian inversion problem that we consider and discuss its discretisation. Our numerical methods, namely the Multigrid Monte Carlo algorithm and bespoke samplers based on matrix splittings are introduced in Section~\ref{sec:methods}. The main theoretical results regarding the invariance, convergence and computational complexity of MGMC are collected in Section~\ref{sec:theory}. Finally, numerical results for several model problems are presented in Section~\ref{sec:results}. We conclude and discuss ideas for future work in Section~\ref{sec:conclusion}.

\section{Problem setting}\label{sec:problem_setting}

We start by writing down the sampling problem in abstract form and discuss its discretisation. Although typically we are interested in priors that arise from the discretisation of a PDE, writing down the problem in general form ensures that our methods are more widely applicable, provided a set of assumptions (which are clearly specified below) are satisfied.
\subsection{Sampling Hilbert space-valued Gaussian random variables}\label{sec:sampling_hilbert_space}

Let $(\Omega,\mathcal{\mathscr{F}},\mathbb{P})$ be a probability space
and $(H,\langle\cdot,\cdot\rangle_{H})$ a separable Hilbert
space \rev{over the real numbers~$\mathbb{R}$}. Let $\mathcal{A}^{-1}\colon H\to H$ be a self-adjoint positive
definite operator.
Then we aim to generate samples following the Gaussian distribution
$\mathcal{N}(0,\mathcal{A}^{-1})$ with covariance operator $\mathcal{A}^{-1}$, where without loss of generality we assume that the mean is zero. An important example of $\mathcal{A}^{-1}$ is the Mat\'ern class of covariance operators, which is given by suitable inverse powers of shifted Laplace operators, such that  $\mathcal{A}^{-1}=(-\Delta+\kappa^2\id)^{-\alpha}$ for $\alpha>0$ \cite{whittle1954stationary,whittle1963stochastic,LindgrenEtAl.F_2011_ExplicitLinkGaussian,BolinEtAl.D_2020_NumericalSolutionFractional}.

Note that for $\mathcal{N}(0,\mathcal{A}^{-1})$ to be a distribution
of an $H$-valued random variable, $\mathcal{A}^{-1}$ must necessarily
be a trace class operator (see e.g., \cite[Proposition
2.16]{DaPrato.D_2014_book}). This implies compactness
of~$\mathcal{A}^{-1}$, which provides natural connections to abstract
partial differential equations (PDEs), i.e., unbounded operators
with compact inverse, and their discretisations. More precisely, we
consider the weak formulation of abstract PDEs on the Cameron--Martin
space and  its discrete counterparts, which gives rise to the linear
operator $\mathcal{A}$ and corresponding discretisation matrices. If
$\mathcal{A}^{-1}$ is compact but not trace-class, it is still
possible to define distributions of real-valued random variables that
describe \textit{observations} of an underlying Gaussian random field
in a bigger space than $H$; this is explained in more detail in
Remark~\ref{rem:trace_class} below \rev{and most of our numerical
experiments are in this setting}.

\subsection{\rev{Linear Gaussian Bayesian inverse problem}\label{subsec:Linear-Bayesian-inverse}}

For now, let us assume that $\mathcal{A}^{-1}$ is trace class. Generating Gaussian random variables which are conditioned on observations is important in linear Bayesian inverse problems, which we now briefly outline.
The problem is to find the distribution of a random variable $v\colon\Omega\to H$,
where the prior distribution is assumed to be Gaussian $v\sim\mathcal{N}(0,\mathcal{A}^{-1})$ for some given $\mathcal{A}$ that typically arises from the weak formulation of a PDE.
Suppose that an \emph{observation} $y\in\mathbb{R}^{\beta}$ is given
by
\begin{equation}
	y=\mathcal{B}v+\eta\quad\text{with}\quad\eta\sim\mathcal{N}(0,\Gamma),\label{eq:infinitdim-obs}
\end{equation}
where $\mathfrak{\mathcal{B}}\colon H\to\mathbb{R}^{\beta}$ is a
bounded linear operator and $\Gamma\in\mathbb{R}^{\beta\times \beta}$ is
the symmetric positive definite covariance matrix of the error $\eta$. The random variables $v$ and
$\eta$ are assumed to be independent. Then, the posterior distribution
is again Gaussian $\mathcal{N}(\mu,\widetilde{\mathcal{A}}^{-1})$
with the posterior mean
\begin{equation}
	\mu=\mathcal{A}^{-1}\mathcal{B}^{*}(\Gamma+\mathcal{B}\mathcal{A}^{-1}\mathcal{B}^{*})^{-1}y\label{eq:postm-infinite}
\end{equation}
and the posterior covariance operator
\begin{equation}
	\widetilde{\mathcal{A}}^{-1}=\mathcal{A}^{-1}-\mathcal{A}^{-1}\mathcal{B}^{*}(\Gamma+\mathcal{B}\mathcal{A}^{-1}\mathcal{B}^{*})^{-1}\mathcal{B}\mathcal{A}^{-1} = \left(\mathcal{A}+\mathcal{B}^{*}\Gamma^{-1}\mathcal{B}\right)^{-1},\label{eq:postcov-infinite}
\end{equation}
where $\mathcal{B}^{*}\colon\mathbb{R}^{\beta}\to H$ is the (Hilbert
space) adjoint operator of $\mathcal{B}$; see for example \cite[Lemma 4.3]{Hairer.M_etal_2005_AnalysisSPDEsArising}
together with \cite[Chapter
2]{Sullivan.TJ_2015_IntroductionUncertaintyQuantification}. \jg{In the
  finite-dimensional case, it is straightforward linear algebra.} Further, define $f\in H$ as
\begin{equation}
	f := \widetilde{\mathcal{A}}\mu = \mathcal{B}^* \Gamma^{-1} y.\label{eqn:f_definition}
\end{equation}
We are interested in generating samples from the posterior distribution
$\mathcal{N}(\mu,\widetilde{\mathcal{A}}^{-1})$ and this paper focuses on the multigrid Monte Carlo sampler, which was introduced in \cite{goodman1989multigrid} and is described in Section~\ref{sec:methods} below.
\jg{The generated samples can be used to sample from the probability distribution of a quantity of interest,}   
which is a functional $\mathcal{F}:H\rightarrow \mathbb{R}$.
In the following we assume that $\mathcal{F}$ is a bounded linear functional of the form
\begin{equation}
	\mathcal{F}(\phi) = \langle \phi, \chi\rangle_H
	\qquad\text{for all $\phi\in H$},
	\label{eqn:qoi_linear}
\end{equation}
where $\chi$ is an element in $H$. 
In a slight abuse of notation, we will sometimes use $\mathcal{F}$ and
its Riesz-representer $\chi$ interchangeably.

\jg{While for linear $\mathcal{F}$, considered here, the expectation
  and covariance of the Gaussian observable can be computed using linear
  algebra, this is no longer the case for non-linear functionals.} 

\subsection{Discretisation\label{sec:discretisation}}

In practice, the problem described above needs to be recast into a finite
dimensional setting which can then be implemented on a computer.

We consider the vector subspace $V:=D(\mathcal{A}^{1/2})\subset H$
with $D(\mathcal{A}^{1/2})$ denoting the domain of the operator $\mathcal{A}^{1/2}$
in $H$. Here, the square-root $\mathcal{A}^{1/2}$ of $\mathcal{A}$
is defined spectrally\rev{; see \cite[Section
  3.7]{Sell.G_You_2013_book} for details}.
Because $\mathcal{A}^{1/2}$ is positive
definite, the bilinear form $a(\cdot,\cdot)\colon V\times V\to\mathbb{R}$
defined via
\begin{equation}
	a(\zeta,\varphi):=
	\langle\mathcal{A}^{1/2}\zeta,\mathcal{A}^{1/2}\varphi\rangle_{H}\qquad\text{ for }\zeta,\varphi\in V,\label{eq:bilin-a}
\end{equation}
is an inner product on $V$ and thus, 
$(V,a(\cdot,\cdot))$ %
forms a Hilbert space \rev{over $\mathbb{R}$}; see e.g., \cite[Section 3.7]{Sell.G_You_2013_book}.
As usual, we let $\|v\|_{V} := a(v,v)$ and note that then $\|v\|_{H}\leq c\|v\|_{V}$ for $v\in V$ with some constant
$c>0$. The space $D(\mathcal{A}^{1/2})$ is identical to the image of $\mathcal{A}^{-1/2}$
on $H$, which is often referred to as the Cameron--Martin space
\cite{DaPrato.G_2006_IntroductionInfiniteDimensionalAnalysis,Lifshits.M.A_2012_Lectures_on_Gaussian}.

To discretise \eqref{eq:infinitdim-obs}--\eqref{eq:postcov-infinite} we introduce a hierarchy $V_{0}\subset\dotsb\subset V_{\ell}\subset\dotsb\subset V$ of nested finite dimensional subspaces with dimensions $n_{\ell}$ for $\ell \ge 0$. 
These subspaces are typically finite element \rev{(FE) spaces on grids with mesh size $h_\ell$} and for each $\ell$ we choose a basis $\{\phi_{j}^{\ell}\}_{j=1,\dots,n_{\ell}}$ of $V_{\ell}$. In each subspace $V_{\ell}$ we observe the quantity
\begin{equation}
	y_{\ell}=\mathcal{B}\tilde{v}_{\ell}+\eta,\label{eq:finitedim-obs}
\end{equation}
where $\tilde{v}_{\ell}\in V_{\ell}$ is an approximation of $v\in V$ and $y_{\ell} \in\mathbb{R}^\beta$ as in \eqref{eq:infinitdim-obs}.
Since the spaces $V_{\ell}$ are finite dimensional, we can expand each function $u_\ell \in V_\ell$ in terms of the basis functions $\phi_{j}^{\ell}$; the expansion coefficients form an $n_\ell$-dimensional vector which decribes the degrees of freedom. As a consequence, the sampling problem on level $\ell$ can be expressed in terms of finite dimensional prior- and posterior-distributions on $\mathbb{R}^{n_{\ell}}$ with corresponding covariance matrices. For this, let $P_{\ell}\colon\mathbb{R}^{n_{\ell}}\to V_{\ell}$ be the vector-space isomorphism defined by $P_{\ell}x=\sum_{j=1}^{n_{\ell}}x^{j}\phi_{j}^{\ell}$ for the degrees-of-freedom-vector $x=(x^{1},\dots,x^{n_{\ell}})\in\mathbb{R}^{n_\ell}$, where $P_{\ell}$ is
a bijection because $\{\phi_{j}^{\ell}\}_{j=1,\dots,n_{\ell}}$ is a basis of $V_\ell$. Moreover, it is natural to assume the following.
\begin{assumption}
	\label{assu:Pell}There is a decreasing function
        $\Phi\colon\mathbb{N}_{0}\to[0,\infty)$ such that %
        $P_{\ell}\colon\mathbb{R}^{n_{\ell}}\to V_{\ell}$ satisfies
	\begin{equation}
		c_{1}\|P_{\ell}x_{\ell}\|_{H}\leq\Phi(\ell)\|x_{\ell}\|_{2}\leq c_{2}\|P_{\ell}x_{\ell}\|_{H}\quad\text{for all }x_{\ell}\in\mathbb{R}^{n_{\ell}},\label{eq:Pl-property}
	\end{equation}
	with some constants $c_1,c_2>0$ independent of $\ell$.
\end{assumption}
For example, when $H=L^{2}(D)$ for a suitable domain $D\subset\mathbb{R}^{d}$ and $V_{\ell}$ is a FE space with shape-regular mesh, then Assumption~\ref{assu:Pell} is satisfied with $\Phi(\ell)=h_{\ell}^{d/2}$; see \cite[Theorem 8.76]{Hackbusch.W_2017_book_elliptic_2nd}. 
\jg{In the language of FE methods, Assumption~\ref{assu:Pell} requires that  the mass matrix for the basis functions $\{\phi_j^{\ell}\}$ is uniformly well-conditioned, which is a common assumption. 
	For more details on FE methods, such as the construction of FE spaces, we refer the reader to \cite{Ciarlet.P2002book}.}

The matrix representation $A_{\ell}$  of the bilinear form $a(\cdot,\cdot)$ restricted to $V_{\ell}\times V_{\ell}$ with respect to $\{\phi_{j}^{\ell}\}_{j=1,\dots,n_{\ell}}$ can be constructed with the help of $P_\ell$ as
\begin{equation}
	(A_{\ell})_{jk}:=a(\phi_j^\ell,\phi_k^\ell) = \langle\mathcal{A}^{1/2}\phi_{j}^{\ell},\mathcal{A}^{1/2}\phi_{k}^{\ell}\rangle_{H}=\langle\mathcal{A}^{1/2}P_{\ell}e_{j}^{\ell},\mathcal{A}^{1/2}P_{\ell}e_{k}^{\ell}\rangle_{H}\,,\label{eq:def-A_ell}
      \end{equation}
where $\{e_{j}^\ell\}_{j=1,2,\dots,n_\ell}$ is the canonical basis of $\mathbb{R}^{n_{\ell}}$. 
\jg{The matrix $A_\ell$ is commonly referred to as the stiffness
  matrix in the FE literature.}  
The (dual) vector-representation $f_\ell\in\mathbb{R}^{n_\ell}$ of $f$ in \eqref{eqn:f_definition} is given by
\begin{equation}
	(f_\ell)_j := \langle f,\phi_j^\ell\rangle_H = y^\top \Gamma^{-1}\mathcal{B}\phi_j^\ell.\label{eqn:f_ell_definition}
\end{equation}
Let further $v_{\ell}:=P_{\ell}^{-1}\tilde{v}_{\ell}$. Then the finite dimensional observation model in \eqref{eq:finitedim-obs} can be rewritten as
\[
	y_{\ell}=\mathcal{B}P_{\ell}v_{\ell}+\eta\,,
\]
while the prior on $\mathbb{R}^{n_{\ell}}$ is chosen to be
$v_{\ell}\sim\mathcal{N}(0,A_{\ell}^{-1})$. Finally, denote by $B_{\ell}$ 
the transpose of the matrix representation of the mapping $\mathcal{B}P_{\ell}\colon\mathbb{R}^{n_{\ell}}\to\mathbb{R}^{\beta}$
with respect to the canonical bases of $\mathbb{R}^{n_\ell}$ and $\mathbb{R}^{\beta}$, i.e.
\begin{equation}
	(B_{\ell})_{jk}:=(\mathcal{B}P_{\ell}e_{j}^\ell)_{k}=(\mathcal{B}\phi_{j}^{\ell})_{k},\qquad j=1,\dots,n_{\ell}\text{ and }k=1,\dots,\beta.\label{eqn:B_ell_definition}
\end{equation}
With this the posterior mean vector $\mu_{\ell}$ and covariance matrix $\widetilde{A}_{\ell}^{-1}$ are given by
\begin{equation}
	\mu_{\ell}=A_{\ell}^{-1}B_{\ell}(\Gamma+B_{\ell}^{\top}A_{\ell}^{-1}B_{\ell})^{-1}y_{\ell},\label{eq:post-mean}
\end{equation}
\begin{equation}
	\widetilde{A}_{\ell}^{-1}=(A_{\ell}+B_{\ell}\Gamma^{-1}B_{\ell}^{\top})^{-1}=A_{\ell}^{-1}-A_{\ell}^{-1}B_{\ell}(\Gamma+B_{\ell}^{\top}A_{\ell}^{-1}B_{\ell})^{-1}B_{\ell}^{\top}A_{\ell}^{-1},\label{eq:post-cov}
      \end{equation}
\noindent
which are the discrete versions of (\ref{eq:postm-infinite}) and (\ref{eq:postcov-infinite}) respectively. In \eqref{eq:post-cov} we used the Sherman--Morrison--Woodbury formula to express $\widetilde{A}_{\ell}^{-1}$ as a low-rank update of $A_\ell^{-1}$; see for example \cite[Theorem 6.20, (2.16)]{Stuart.A_2010_Acta_Inverse} together with \cite[Chapter 0]{Horn.R_Johnson_book_2013_2nd}.

Let $I_{\ell-1}^{\ell}\in \mathbb{R}^{n_{\ell}\times n_{\ell-1}}$ be the matrix representation of the mapping $P_{\ell}^{-1}P_{\ell-1}: \mathbb{R}^{n_{\ell-1}}\rightarrow\mathbb{R}^{n_{\ell}}$, i.e.,
\begin{equation}
	(I_{\ell-1}^{\ell})_{jk} = (e_j^\ell)^\top P_{\ell}^{-1}P_{\ell-1} e_k^{\ell-1}\qquad\text{for $j=1,\dots,n_{\ell}$ and $k=1,\dots,n_{\ell-1}$}.
	\label{eqn:definition_prolongation}
\end{equation}
The matrix $I_{\ell-1}^{\ell}$ defined in this way is referred to as the canonical prolongation in the multigrid literature, the corresponding restriction matrix is $I_{\ell}^{\ell-1}:=(I_{\ell-1}^{\ell})^{\top}\in \mathbb{R}^{n_{\ell-1}\times n_{\ell}}$. With this the matrices $A_{\ell}$ defined by \eqref{eq:def-A_ell}
and $B_{\ell}$ in \eqref{eqn:B_ell_definition} on consecutive levels
can be related as follows:
\begin{equation}
	A_{\ell-1}=I_{\ell}^{\ell-1}A_{\ell}I_{\ell-1}^{\ell} =
        (I_{\ell-1}^{\ell})^\top A_{\ell}I_{\ell-1}^{\ell}, \qquad
	B_{\ell-1}= I_{\ell}^{\ell-1}B_\ell.\label{eqn:A_ell_1_B_ell_1}
\end{equation}

The components of the vector representation $F_\ell\in \mathbb{R}^{n_{\ell}}$ of the mapping $\mathcal{F}P_\ell:\mathbb{R}^{n_\ell}\rightarrow \mathbb{R}$ related to  the linear functional $\mathcal{F}$ defined in \eqref{eqn:qoi_linear} (with respect to the canonical basis) are given by
\begin{equation}
	(F_\ell)_j := \mathcal{F}(P_\ell e_j^\ell) = \mathcal{F}(\phi_j^\ell),\qquad j=1,2,\dots,n_\ell.\label{eqn:F_matrix}
\end{equation}

Finally, we list a natural assumption on the operator $\mathcal{A}$ and its
discretisations. This assumption is satisfied in general if
$\mathcal{A}$ arises from a PDE problem, such as in Example
\ref{example-Laplace} below. To motivate this assumption, we note that the following problem
admits a unique solution: Find $u\in V$ such that
\begin{equation}
	a(u,\varphi)=\langle f,\varphi\rangle_{H}\text{ for all }\varphi\in V,\label{eq:eq-given-by-A}
\end{equation}
with $a(\cdot,\cdot)$ as in \eqref{eq:bilin-a}. This follows since $\|\zeta\|_{V}^{2}=a(\zeta,\zeta) > 0$ for all $0 \not=\zeta \in V$ and $|a(\zeta,\varphi)|\leq\|\zeta\|_{V}\|\varphi\|_{V}$
for all $\zeta,\varphi\in V$, and as a consequence the Lax--Milgram theorem guarantees existence and uniqueness of the soution $u$.
\begin{assumption}
	\label{assu:WtoH-best-Vell-approx}There exists a subspace 
	$W\subset V \subset H$ equipped with a norm $\|\cdot\|_W$ such that
        \begin{enumerate}
        \item[(a)] the solution $u$ of \eqref{eq:eq-given-by-A}
	belongs to $W$ and $\|u\|_{W}\leq C_{\mathcal{A}}\|f\|_{H}$,
	for some constant $C_{\mathcal{A}}>0$;
        \item[(b)] the best-approximation error in $V_\ell\subset V$ with respect to
        the $V$-norm satisfies
	\[
        \inf_{w_{\ell}\in  V_{\ell}}\|w-w_{\ell}\|_{V}\leq\sqrt{\Psi(\ell)}\|w\|_{W}\,,\qquad\text{for
                  any }w\in W 
        \]
        for a  decreasing function $\Psi\colon\mathbb{N}_{0}\to[0,\infty)$ \jg{with $\Psi(\ell)\rightarrow 0$ as $\ell\rightarrow \infty$}.
              \end{enumerate}
\end{assumption}
To see why this assumption is natural, consider the following example. We use the notation $\Psi(\ell)\asymp h_\ell^2$ to denote asymptotic growth in the sense that there exist constants $c_\pm>0$ and $\ell_0\in\mathbb{N}$ such that $c_- h_\ell^2 < \Psi(\ell) <c_+ h_\ell^2$ for all $\ell \ge \ell_0$.
\begin{example}
	\label{example-Laplace}Let $\mathcal{A}:=-\Delta+\kappa^{2}\id$, where
	$\Delta$ is the Laplace operator with zero-Dirichlet boundary condition
	and $\kappa^{2}>0$ is a constant. 
	Let $D\subset\mathbb{R}^{d}$ be a bounded and convex domain, and suppose that $(V_{\ell})_{\ell
        \ge 0}$ is a family of continuous, piecewise linear FE
        spaces defined on a regular family of triangulations with mesh
        sizes $h_{\ell}$. Then, it is well known that Assumption~\ref{assu:WtoH-best-Vell-approx} %
	is satisfied with $W=H^{2}(D)$ and $\Psi(\ell)\asymp
        h_{\ell}^{2}$; see, e.g., \cite[Theorem
        9.24]{Hackbusch.W_2017_book_elliptic_2nd} and \cite[Theorem
        3.4.2]{Quarteroni.A_Vali_book_1994_NAPDE}, respectively.
\end{example}
\noindent
\rev{While Assumption~\ref{assu:WtoH-best-Vell-approx} is natural for
  Mat\'{e}rn-like distributions on finite domains as in
  Example~\ref{example-Laplace}, it is more general in that it also
  covers non-stationary processes similar to the ones described in
  \cite{fuglstad2015exploring}. To see this consider, for example, a
  precision operator of the form \mbox{$\mathcal{A}=-\nabla\cdot
    (K(x)\nabla \cdot ) + \kappa(x)^2\id$} where $K(x)$ and
  $\kappa(x)$ are sufficiently smooth, spatially varying fields,
  respectively matrix- and real-valued with $K(x)$ symmetric positive
  definite at every point $x\in D$.}

Under Assumption~%
\ref{assu:WtoH-best-Vell-approx}, the $H$-norm error $\|u-u_{\ell}\|_{H}$
of approximating $u$ with $u_{\ell}$ decays at the rate $\Psi(\ell)$
as stated in the following proposition (for a proof see, e.g., \cite[Proof of Theorem 8.65]{Hackbusch.W_2017_book_elliptic_2nd}).
\begin{prop}[Aubin and Nitsche]
  \label{prop:Aubin-Nitsche}
  Let $u\in V$ satisfy \eqref{eq:eq-given-by-A} and let Assumption~%
  \ref{assu:WtoH-best-Vell-approx}
  hold. Suppose furthermore that
  $u_{\ell}\in V_{\ell}$ satisfies
  \[
    a(u_{\ell},\varphi_{\ell})=\langle
    f,\varphi_{\ell}\rangle_{H}\text{ for all }\varphi_{\ell}\in
    V_{\ell}\,.
  \]
  Then there exists a constant $\tilde{C}>0$ independent of $\ell$
  such that
	\[
		\|u-u_{\ell}\|_{H}\leq\tilde{C}\Psi(\ell)\|u\|_{W}\,.
	\]
\end{prop}
\begin{rem}\label{rem:trace_class}
	The operator in Example~\ref{example-Laplace} is important and
        it will be the main one considered in our
        numerics. However, in this case $\mathcal{A}^{-1}$ is not
        trace-class. This does not affect the discrete
        formulation of the problem in
        Section~\ref{sec:discretisation}, but the random fields $v$
        and $\widetilde v$
        described by the infinite-dimensional  probability
        distributions $\mathcal{N}(0,\mathcal{A}^{-1})$ and
        $\mathcal{N}(\mu,\widetilde{\mathcal{A}}^{-1})$ are not
        classical fields in $H$. However, provided $\mathcal{A}$ has a compact
        inverse, it is still possible to define a family of $\mathbb{R}$-valued
        random variables as follows: intuitively, we consider linear functionals
        $(v,\chi) \in \mathbb{R}$ of the random field $v$ with Riesz-representor $\chi \in H$,
        generalising the $H$-inner product in \eqref{eqn:qoi_linear}. More
        precisely we
        consider the collection $\{(v,\chi) : \chi \in H\}$
of $\mathbb{R}$-valued Gaussian random variables satisfying 
\begin{align}
\mathbb{E}[(v ,\chi)] = & \ 0 \quad &\text{for \ensuremath{\chi\in H}},\label{eq:generalised-mean}\\
\mathbb{E}\bigl[\bigl((v ,\chi)-\mathbb{E}[(v ,\chi)]\bigr)\bigl((v
  ,\psi)-\mathbb{E}[(v ,\psi)]\bigr)\bigr] = & \ \langle
                                                    \chi,\mathcal{A}^{-1}\psi\rangle_{H}\quad
                                                    &\text{for \ensuremath{\chi,\psi\in H}}\label{eq:generalised-cov}
\end{align}
such that any linear combination of elements from $\{(v,\chi) :\chi\in H\}$ is a Gaussian random variable. 

To realise such random variables, we follow \cite[Section
4.1.2]{DaPrato.D_2014_book} and let $(\varphi_{j})_{j\geq1}$
be a complete $H$-orthonormal system (ONS) consisting of eigenfunctions
of $\mathcal{A}^{-1}$ with corresponding eigenvalues $(1/\lambda_{j})_{j\geq1}$. Then, $(\varphi_{j}/\sqrt{\lambda_{j}})_{j\geq1}$ is
a complete ONS of  $V =D(\mathcal{A}^{1/2})\subset
H$. For each $\chi \in H$, we define 
$(v,\chi)$ as the $L^{2}(\Omega)$-limit of the series
\begin{equation}\label{eq:def_general_functional}
  (v,\chi):= %
  \sum_{j=1}^{\infty} \xi_{j} \frac{ \langle \varphi_{j} , \chi \rangle_{H}}{\sqrt{\lambda_{j}}} 
,
\end{equation}
where $\{\xi_{j}\}_{j\geq1}$ are independent, standard
Gaussian random variables. Since
\begin{align*}
\mathbb{E}\bigg[\Big| \sum_{j=1}^{J} \xi_{j} \frac{ \langle
  \varphi_{j} , \chi \rangle_{H}}{\sqrt{\lambda_{j}}} 
  \Big|^{2}\bigg] & =\mathbb{E}\bigg[\Big| \sum_{j=1}^{J} \xi_{j}\langle\mathcal{A}^{-1/2}\varphi_{j},\chi\rangle_{H}\Big|^{2}\bigg]
=\sum_{j=1}^{J}\big|\langle\mathcal{A}^{-1/2}
                    \varphi_{j},\chi\rangle_{H}\big|^{2} \leq \|\mathcal{A}^{-1/2}\chi\|_{H}^{2}<\infty,
\end{align*}
the series in \eqref{eq:def_general_functional} converges the
resulting random variable $(v,\chi)$
has the desired mean and covariance in (\ref{eq:generalised-mean})--(\ref{eq:generalised-cov}). 
Moreover, following again \cite{DaPrato.D_2014_book} the random field $v$ can be directly constructed on a suitably large Hilbert space without taking the coupling with $\chi\in H$. 

The above derivation can easily be extended to nonzero mean, to
vector-valued functionals or to
the Bayesian setting.
For the latter we generalise the
observation operator $\mathfrak{\mathcal{B}}\colon
H\to\mathbb{R}^{\beta}$ in a similar way to $\mathcal{F}$. And then as above,
we can consider the collection $\{(\widetilde v,\chi) : \chi \in H\}$
of $\mathbb{R}$-valued Gaussian random variables satisfying 
\begin{align}
\mathbb{E}[(\widetilde v ,\chi)] = & \ \langle \chi,\mu\rangle_H \quad &\text{for \ensuremath{\chi\in H}},\label{eq:generalised-mean-post}\\
\mathbb{E}\bigl[\bigl((\widetilde v ,\chi)-\mathbb{E}[(\widetilde v ,\chi)]\bigr)\bigl((\widetilde v
  ,\psi)-\mathbb{E}[(\widetilde v ,\psi)]\bigr)\bigr] = & \ \langle
                                                    \chi, \widetilde{\mathcal{A}}^{-1}\psi\rangle_{H}\quad
                                                    &\text{for \ensuremath{\chi,\psi\in H}}\label{eq:generalised-cov-post}
\end{align}
with $\mu$ in \eqref{eq:postm-infinite} and $\widetilde{\mathcal{A}}$ in \eqref{eq:postcov-infinite}.
\end{rem}
\noindent
For the remainder of this paper, in particular to show \rev{that convergence rates are independent of the mesh-size $h_\ell$}, we will only rely on the fact that $\mathcal{A}^{-1}$ is compact not that it is trace-class.

\section{Methods}\label{sec:methods}
Having described the general context for the sampling problem, we now discuss the numerical methods employed in this work.
\subsection{Multigrid Monte Carlo}
Given some sufficiently large value $L \in \mathbb{N}$, our goal is to sample from the multivariate normal distribution $\mathcal{N}(\mu_L,\widetilde{A}_L^{-1})$ with mean $\mu_L:=\widetilde{A}_L^{-1}f_L$ defined in \eqref{eq:post-mean} and covariance $\widetilde{A}_{L}^{-1}$ given in \eqref{eq:post-cov}. Recall that $f_L\in\mathbb{R}^{n_L}$ and $\widetilde{A}_L$ is a symmetric $n_L \times n_L$ matrix. Starting from some initial state $\theta_{L}^{(0)} \in \mathbb{R}^{n_L}$, the Multigrid Monte Carlo update in Alg.~\ref{alg:mgmc} below will generate a Markov chain $\theta_{L}^{(0)},\theta_{L}^{(1)},\theta_{L}^{(2)},\dots$ with $\theta_{L}^{(k+1)}=\mathrm{MGMC}_L(\widetilde{A}_L,f_L,\theta_{L}^{(k)},\jg{\xi^{(k)}})$ which converges to the target distribution $\mathcal{N}(\mu_L,\widetilde{A}_L^{-1})$ in a sense that will be made precise in Section \ref{sec:theory_convergence}. \jg{Here,  $\xi^{(k)}$ is a collection of random variables independent of  everything before iteration $k$. To simplify notation, we will omit the dependence of $\mathrm{MGMC}_L$ on the random variable $\xi^{(k)}$ in the following.}

Multigrid Monte Carlo introduces a hierarchy of coarser levels $L-1,L-2,\dots,0$, associated with vector spaces of dimension $n_L > n_{L-1} > \dotsb > n_1 > n_0$. This corresponds to the hierarchy $V_0\subset V_1\subset \dots\subset V_{L-1}\subset V_L$ of nested function spaces introduced in Section \ref{sec:discretisation} in the sense that the vector space $\mathbb{R}^{n_{\ell}}$ contains the degrees-of-freedom vectors of the function space $V_\ell$. In \cite{goodman1989multigrid}, the method was introduced for more general distributions. \rev{While the authors conclude that MGMC does not overcome critical slowing down for theories that undergo a second order phase transition, such as the Ising model in higher dimensions,  it is successful for some applications in quantum field theory. In particular, here we restrict ourselves to the Gaussian setting, which corresponds to the Gaussian Free Field model for which MGMC has been shown to be an efficient sampler in \cite{goodman1989multigrid}.}
\jg{On the other hand, the way our Alg.~\ref{alg:mgmc} is written down makes it explicit that the linear operators for transferring samples between levels of the hierarchy are not unique. This allows a generalisation of the MGMC algorithm to algebraic multigrid (AMG) settings or more sophisticated FE discretisations. While arbitrary intergrid operators are also implicitly included in the algorithms presented in \cite{goodman1989multigrid}, the authors of that paper focus on the case of piecewise linear interpolation for finite difference discretisations on regular meshes. }%
As in the standard multigrid algorithm for solving linear systems, we introduce a \textit{prolongation matrix} $I_{\ell-1}^{\ell}$ defined in \eqref{eqn:definition_prolongation}; this matrix has full column rank for all levels $\ell$. The corresponding \textit{restriction matrix} $I_{\ell}^{\ell-1} = (I_{\ell-1}^{\ell})^\top$ is the transpose of the prolongation matrix.  Coarse level matrices are then recursively constructed via the so-called Galerkin triple-product
\begin{equation}
    \widetilde{A}_{\ell-1}  =I_{\ell}^{\ell-1} \widetilde{A}_\ell I_{\ell-1}^{\ell}                                                           = A_{\ell-1} + B_{\ell-1} \Gamma^{-1} B_{\ell-1}^{\top}
    \label{eqn:galerkin_triple_product}
\end{equation}
with $A_{\ell-1}$ and $B_{\ell-1}$ defined in \eqref{eqn:A_ell_1_B_ell_1}. We will show in Proposition \ref{prop:gaussian_coarse_level_correction} below that the relationship between $\widetilde{A}_{\ell-1}$ and $\widetilde{A}_{\ell}$ written down in \eqref{eqn:galerkin_triple_product} is crucial to ensure that the MGMC coarse grid correction leaves the target distribution invariant. On the finest level the covariance matrix $A_\ell$ is symmetric; \eqref{eqn:galerkin_triple_product} then implies that all coarse level matrices are symmetric as well.

\subsubsection{\jg{Algorithm description}}
To understand the idea behind Multigrid Monte Carlo as written down in Alg.~\ref{alg:mgmc}, it is instructive to first consider the two-level case ($L=1$). In this case, the update $\theta_L\mapsto\theta_L'$ is given as follows:
\begin{enumerate}
    \item[(i)] Starting with a given sample $\theta_{L,0}=\theta_{L}\in\mathbb{R}^{n_L}$ on the fine level $L=1$, we apply a number $\nu_1$ of random smoothing steps to obtain $\theta_{L,\nu_1}$. All smoothers that we consider in this work can be written in the form of a splitting method shown in Alg.~\ref{alg:random_smoother}. For example a standard Gibbs-sweep over all unknowns would correspond to $\widetilde{M}_L=\widetilde{D}_L+\widetilde{L}_L$, where $\widetilde{D}_L$ is the diagonal and $\widetilde{L}_L$ is the lower triangular part of $\widetilde{A}_L$, see Section \ref{subsec:Gibbs} for a more detailed discussion.
    \item[(ii)] We then use the restriction matrix $I_L^{L-1}$ to construct the state dependent coarse level mean $f_{L-1} = I_L^{L-1}(f_L -\widetilde{A}_L \theta_{L,\nu_1})\in\mathbb{R}^{n_{L-1}}$. Next, we draw a coarse level sample $\psi_{L-1} \in\mathbb{R}^{n_{L-1}}$ from the multivariate normal distribution with mean $\widetilde{A}_{L-1}^{-1}f_{L-1}$ and covariance $\widetilde{A}_{L-1}^{-1}$, where $\widetilde{A}_{L-1} $ is computed via \eqref{eqn:galerkin_triple_product}. The prolongated coarse level sample is then added to construct an updated fine-level sample $\theta_{L,\nu_1+1} = \theta_{L,\nu_1}+I_{L-1}^{L} \psi_{L-1}$.
    \item[(iii)] Finally, we apply a number of $\nu_2$ of random smoothing steps to obtain a new state $\theta'_{L}:=\theta_{L,\nu_1+\nu_2+1}$.
\end{enumerate}

Extending this idea to more than two levels by applying the above procedure recursively on all levels $\ell=L,L-1,\dots,1,0$, leads to the Multigrid Monte Carlo method in Alg.~\ref{alg:mgmc}. To achieve this the update $\mathrm{MGMC}_{\ell-1}(\widetilde{A}_{\ell-1},f_{\ell-1},\cdot)$ itself is used on the coarser level to produce the sample $\psi_{\ell-1}$. As for classical multigrid, we use $\gamma_\ell\in\mathbb{N}$ recursive calls on the coarser levels where $\gamma_\ell$ might be larger than $1$. Setting $\gamma_\ell=1$ for all levels $\ell$ corresponds to a so-called V-cycle whereas $\gamma_\ell=2$ for $\ell<L$ and $\gamma_L=1$ leads to the W-cycle.
\begin{algorithm}[t!]
    \caption{Multigrid Monte Carlo update $\theta_\ell \mapsto\theta'_\ell$\\The random pre- and post-smoothers are defined by the splitting matrices $\widetilde{M}_\ell^{\mathrm{pre}}$, $\widetilde{M}_\ell^{\mathrm{post}}$ and the number of smoothing steps $\nu_1$, $\nu_2$; the cycle parameters $\gamma_\ell$ control the number of recursive calls}\label{alg:mgmc}
    \begin{algorithmic}[1]
        \Procedure{$\mathrm{MGMC}_{\ell}$}{$\widetilde{A}_{\ell},f_{\ell},\theta_{\ell}$}
        \If{$\ell=0$ }
        \State\Return $\theta'_{0}:=\mathrm{CoarseSampler}(\widetilde{A}_{0},f_{0},\theta_{0})$\Comment{Alg.~\ref{alg:coarse_sampler} or Cholesky sampler}
        \Else
        \State  Let $\theta_{\ell,0}:=\theta_{\ell}$
        \For{$j=0,\dots,\nu_{1}-1$}
        \State $\theta_{\ell,j+1}:=\mathrm{RandomSmoother}(\widetilde{A}_{\ell},\widetilde{M}_{\ell}^{\mathrm{pre}},f_{\ell},\theta_{\ell,j})$\Comment{Random pre-smoothing (Alg.~\ref{alg:random_smoother})}
        \EndFor
        \State Define $f_{\ell-1}:=I_{\ell}^{\ell-1}(f_{\ell}-\widetilde{A}_{\ell}\theta_{\ell,\nu_{1}})$
        \Comment {Restriction}
        \State Let $\psi_{\ell-1}^{(0)}:=0$
        \For{$m=0,1,\dots,\gamma_\ell-1$}
        \State $\psi_{\ell-1}^{(m+1)}:=\mathrm{MGMC}_{\ell-1}(\widetilde{A}_{\ell-1},f_{\ell-1},\psi_{\ell-1}^{(m)})$
        \Comment{Recursive call to $\mathrm{MGMC}_{\ell-1}$}
        \EndFor
        \State $\theta_{\ell,\nu_{1}+1}:=\theta_{\ell,\nu_{1}}+I_{\ell-1}^{\ell}\psi_{\ell-1}^{(\gamma_\ell)}$
        \Comment {Prolongation}
        \For{$j=0,\dots,\nu_{2}-1$}
        \State $\theta_{\ell,\nu_{1}+2+j}:=\mathrm{RandomSmoother}(\widetilde{A}_{\ell},\widetilde{M}_{\ell}^{\mathrm{post}},f_{\ell},\theta_{\ell,\nu_{1}+1+j})$
        \Comment{Random post-smoothing (Alg.~\ref{alg:random_smoother})}
        \EndFor
        \State \Return $\theta'_{\ell}:=\theta_{\ell,\nu_{1}+\nu_{1}+1}$
        \EndIf
        \EndProcedure
    \end{algorithmic}
\end{algorithm}
As we will show in Section \ref{sec:theory}, Multigrid Monte Carlo has several desirable properties:
\begin{enumerate}
    \item Under suitable conditions on the problem sizes $n_\ell$, the cost of one call to Alg.~\ref{alg:mgmc} grows no more than linearly in the number of unknowns $n_L$ on the finest level (see Section \ref{subsec:Cost-analysis}).
    \item $\textrm{MGMC}_L(\widetilde{A}_L,f_L,\cdot)$ leaves the target distribution invariant: if $\theta_L \sim \mathcal{N}(\mu_L,\widetilde{A}_{L}^{-1})$ for $\mu_L=\widetilde{A}_L^{-1}f_L$ then $\theta'_L \sim \mathcal{N}(\mu_L,\widetilde{A}_{L}^{-1})$ (see Section \ref{sec:theory_invariance}).
    \item If the initial state $\theta_{L}^{0}\sim \mathcal{N}(\mu_L^0,(\widetilde{A}_{L}^0)^{-1})$ is drawn from a multivariate normal distribution with mean $\mu_L^0 = (\widetilde{A}_L^{0})^{-1}f_L^0$ and covariance $(\widetilde{A}_{L}^0)^{-1}$, the Markov chain $\theta_{L}^{(0)},\theta_{L}^{(1)},\theta_{L}^{(2)},\dots$ generated by $\textrm{MGMC}_L(\widetilde{A}_L,f_L,\cdot)$ converges to the target distribution $\mathcal{N}(\mu_L,\widetilde{A}_L^{-1})$ in the sense that $\theta_{L}^{(k)}\sim \mathcal{N}(\mu_L^k,(\widetilde{A}_{L}^k)^{-1})$ with $\mu_L^k\rightarrow \mu_L$ and $\widetilde{A}_L^k\rightarrow \widetilde{A}_L$. The rate of convergence is grid-independent for both the mean and the covariance (see Section \ref{sec:theory_convergence}).
    \item The autocorrelation between samples in the Markov chain is small, and the integrated autocorrelation time for random variables that depend linearly on the sample state (as in \eqref{eqn:qoi_linear} and \eqref{eqn:F_matrix}) can be bounded by a grid independent constant (see Section \ref{sec:theory_convergence}).
\end{enumerate}
This implies that MGMC is an efficient sampler in the following sense:
\begin{enumerate}
    \setcounter{enumi}{4}
    \item The cost for generating an (approximately) independent sample in the Markov chain is optimal in the sense that it is proportional to the number of unknowns.
\end{enumerate}
The intuitive explanation for the success of MGMC is the same as for multigrid solvers: by using a hierarchy of levels, the samples are updated on all length scales simultaneously and this incurs only a small, grid independent overhead. This is in stark contrast to a Gibbs sampler, which only updates the samples \textit{locally}.
\rev{The same idea can also be understood spectrally: a sample $\theta_L^{(k)}$ in the Markov chain can be expanded as
\begin{equation}
    \theta_L^{(k)} = \mu_L + \sum_{j=1}^{n_L} \widetilde{v}_{L,j}\widetilde{\lambda}_{L,j}^{-1/2} \xi_{L,j}^{(k)}
\end{equation}
where $\widetilde{\lambda}_{L,j}$ and $\widetilde{v}_{L,j}$ are the eigenvalues and normalised eigenvectors of $\widetilde{A}_L$ respectively and for samples that are drawn from the target distribution $\xi_{L,j}^{(k)}\sim\mathcal{N}(0,1)$ are i.i.d. Gaussian random variables. Small eigenvalues correspond to smooth eigenvectors that vary slowly in space. While one step of the Gibbs-sampler only significantly changes the spectral expansion coefficients $\xi_{L,j}^{(k)}$ that correspond to high-frequency eigenvectors, MGMG will modify all expansion coefficients simultaneously for suitable operators $\mathcal{A}$. This should be compared to the action of a deterministic multigrid update which reduces the error on all length scales, i.e. for all eigenvectors.
}
\subsection{Random smoothers}\label{sec:random_smoothers}
A central ingredient of Alg.~\ref{alg:mgmc} is the $\mathrm{RandomSmoother}()$ procedure in lines 7 and 16, which can be seen as a generalisation of the Gibbs smoother.
All random smoothers considered here are based on splitting a general symmetric matrix $A$ into two parts as $A =: M-N$, where the matrix $M$ defines the splitting. The generic update procedure is written down in Alg.~\ref{alg:random_smoother} where here and in the following we assume that the input $\theta$ is drawn from a multivariate normal distribution. As the following lemma shows, it is always possible to symmetrise a given random smoother defined by the splitting $A=M-N$.
\begin{algorithm}[t!]
    \caption{Random Smoother for updating $\theta\mapsto \theta'$ based on the matrix splitting $A=M-N$}\label{alg:random_smoother}
    \begin{algorithmic}[1]
        \Procedure{$\mathrm{RandomSmoother}$}{$A$, $M$, $f$, $\theta$}
        \State {Set
            \begin{equation}
                \theta':=\theta+M^{-1}(f+\xi-A \theta)\qquad\text{with $\xi\sim\mathcal{N}(0,M+M^{\top}-A)$}
                \label{eqn:random_smoother_update}
            \end{equation}
        }
        \State\Return {$\theta'$}
        \EndProcedure
    \end{algorithmic}
\end{algorithm}
\begin{lem}[Symmetrised Random Smoother]\label{lem:symmetric_random_smoother}
    Let $A=A^\top$ be a symmetric matrix and let $M$ be an invertible square matrix of the same size as $A$ with $M+M^\top\ne A$.
    Then the smoother obtained by applying the update in Alg.~\ref{alg:random_smoother} with splitting matrix $M$ followed by the same algorithm with $M^\top$ is equivalent to one application of Alg.~\ref{alg:random_smoother} with the symmetric splitting matrix
    \begin{equation}
        M^{\mathrm{sym}} = M(M+M^\top-A)^{-1}M^\top,\label{eqn:symmetric_M}
    \end{equation}
    in the sense that these two alternative update procedures yield samples from the same distribution.
\end{lem}
\begin{proof}
    The result can be shown by computing the mean and covariance of the multivariate normal distribution that is obtained by combining the two multivariate normal distributions which define the individual updates with splitting matrices $M$ and $M^\top$ respectively; see Appendix \ref{sec:proof_lem:symmetric_random_smoother} for details.
\end{proof}
\subsection{Relationship to deterministic smoothers}\label{sec:random_vs_deterministic_smoothers}
Observe that replacing \eqref{eqn:random_smoother_update} in the \textit{random smoother} Alg.~\ref{alg:random_smoother} by the deterministic update
\begin{equation}
    u':=\theta+M^{-1}(f-A u)\label{eqn:smoother}
\end{equation}
would result in a \textit{deterministic smoother} defined by the splitting $A = M-N$ (see e.g. \cite[section 4.2.2]{saad2003iterative}). Under certain conditions on the splitting matrix $M$, repeated applications of the update in \eqref{eqn:smoother} can be used to iteratively solve the linear system $Au = f$. For example, when splitting the symmetric matrix $A$ into its diagonal $D$, strict lower triangular part $L$ and strict upper diagonal part $L^\top$
\begin{equation}
    A = D + L + L^\top,
\end{equation}
successive over-relaxation (SOR) would correspond to the splitting
\begin{xalignat}{2}
    M &= M^{\mathrm{SOR}} = \frac{1}{\omega} D + L, &
    N &= N^{\mathrm{SOR}} = \frac{\omega-1}{\omega} D - L^\top\rev{,}
    \label{eqn:SOR_smoother}
\end{xalignat}
\rev{where $\omega>0$ is referred to as \textit{relaxation parameter}  in the context of iterative solvers.} 
As discussed extensively in \cite{goodman1989multigrid,Fox.C_Parker_2017_AcceleratedGibbsSampling}, there is a close relationship between classical smoothers for the solution of the linear system $Au =f$ and what we call random smoothers for sampling from the \rev{multivariate} normal distribution $\mathcal{N}(\mu,A^{-1})$ with $\mu=A^{-1}f$ in Section \ref{sec:random_smoothers}. For example, the random pendant of the deterministic SOR smoother in \eqref{eqn:SOR_smoother} with $\omega=1$ is the Gibbs-sampler (see \cite[Table 2]{Fox.C_Parker_2017_AcceleratedGibbsSampling}), which we will refer to as the Gibbs-\textit{smoother} from now on to make the connection more explicit.
\subsection{Random smoother with low-rank update for Bayesian inference}\label{subsec:Gibbs}
Having discussed general smoothers based on matrix splittings, we now design bespoke splitting methods for sampling from the posterior distribution defined in Section\ref{subsec:Linear-Bayesian-inverse}. To achieve this, we use the relationship to deterministic smoothers outlined in the previous section and construct efficient random smoothers for the (symmetric) matrix
\begin{equation}
    A = \widetilde{A}_\ell = A_\ell + B_{\ell} \Gamma^{-1} B_{\ell}^{\top}
\end{equation}
defined in \eqref{eq:post-cov}. More specifically, splitting $A_\ell$ into its diagonal $D_\ell$, strict lower triangular part $L_\ell$ and strict upper triangular part $L_\ell^\top$ as
\begin{equation}
    A_\ell = D_\ell + L_\ell + L_\ell^\top,
\end{equation}
we set
\begin{equation}
    M = \widetilde{M}_\ell :=
    \frac{1}{\omega} D_\ell + B_{\ell} \Gamma^{-1}B_{\ell}^{\top} + L_\ell,\qquad \widetilde{N}_\ell = \widetilde{M}_\ell-\widetilde{A}_\ell\label{eqn:M_lowrank_gibbs}
\end{equation}
in the random smoother update in \eqref{eqn:random_smoother_update} to obtain what we call the ``(forward) Gibbs smoother with low rank correction''. As will be discussed below, the seamingly simpler approach of setting $M = \widetilde{M}_\ell :=\frac{1}{\omega} D_\ell + L_\ell$ does not work in general. The construction in \eqref{eqn:M_lowrank_gibbs} is motivated by the observation that a classical smoother is usually more effective if we include more terms in the matrix $M$ that defines the splitting in \eqref{eqn:smoother} while making sure that $M$ can still be inverted efficiently. Since $B_{\ell} \Gamma^{-1} B_{\ell}^{\top}$ is a low-rank correction to the lower triangular matrix $\omega^{-1}D_\ell+L_\ell$, we can still invert $\omega^{-1}D_\ell+L_\ell + B_{\ell} \Gamma^{-1} B_{\ell}^{\top}$ efficiently with the Woodbury matrix identity.

Our choice of splitting also needs to allow for fast sampling of the random variable $\xi$ in \eqref{eqn:random_smoother_update}. To draw a sample from the multivariate normal distribution with mean zero and covariance $M+M^\top-A$, note that for the choice of $M$ in \eqref{eqn:M_lowrank_gibbs} we have that
\begin{equation}
        M+M^\top-A = \widetilde{M}_\ell + \widetilde{M}_\ell^\top - \widetilde{A}_\ell  =\frac{2-\omega}{\omega} D_\ell + B_{\ell} \Gamma^{-1} B_{\ell}^{\top}.
        \label{eqn:low_rank_smoother_covariance}
\end{equation}
which is the sum of two positive definite matrices. Hence, the problem can be reduced to drawing a sample $\xi^{\mathrm{diag}}_\ell\sim \mathcal{N}(0,D_\ell)$ from the $n_\ell$ dimensional multivariate normal distribution with mean zero and diagonal covariance matrix $D_\ell$ and drawing a sample $\xi^\mathrm{LR}_\ell\sim \mathcal{N}(0,\Gamma^{-1})$ from the $\beta$-dimensional multivariate normal distribution with mean zero and covariance matrix $\Gamma^{-1}$. The linear combination
\begin{equation}
    \xi_\ell = \sqrt{\frac{2-\omega}{\omega}} \xi_\ell^{\mathrm{diag}} + B_{\ell} \xi_\ell^{\mathrm{LR}}
\end{equation}
will then be a sample from a multivariate normal distribution with mean zero and the desired covariance given in \eqref{eqn:low_rank_smoother_covariance}.

Furthermore, the update $\theta_\ell' =\theta_\ell + \widetilde{M}_\ell^{-1}(f_\ell+\xi_\ell-\widetilde{A}_\ell\theta_\ell)$ in \eqref{eqn:random_smoother_update} can be split into two steps as follows
\begin{equation}
    \begin{aligned}
        \theta_\ell^* & = \theta_\ell + (\omega^{-1}D_\ell+L_\ell)^{-1} (f_\ell+\xi_\ell-A_\ell \theta_\ell) \\
        \theta'_\ell  & = \theta_\ell^* - G_{\ell} (B_{\ell}^{\top} \theta_\ell^*)
    \end{aligned}\label{eqn:lowrank_gibbs_two_steps}
\end{equation}
with the $n_\ell\times\beta$ matrix
\begin{equation}
    G_{\ell} =  (\omega^{-1}D_\ell+L_\ell)^{-1}B_{\ell}\left(\Gamma + B_{\ell}^{\top} \left(\omega^{-1}D_\ell+L_\ell\right)^{-1}B_{\ell}\right)^{-1},
    \label{eqn:Bellstar_definition}
\end{equation}
which can be precomputed once at the beginning of the simulation. An equivalent \textit{backward} Gibbs sampler with low rank correction can be obtained by swapping $L_\ell\leftrightarrow L_\ell^\top$ in \eqref{eqn:lowrank_gibbs_two_steps} and \eqref{eqn:Bellstar_definition}.
Putting everything together we arrive at Alg.~\ref{alg:low_rank_gibbs}. Obviously, $C_\ell$ and $G_{\ell}$ can be precomputed once and used in all subsequent calls, so that the setup costs are amortised for large numbers of samples.
\begin{algorithm}[t!]
    \caption{Gibbs smoother with low rank correction for Bayesian inference.\\Computes the update $\theta_\ell\mapsto\theta'_\ell$ based on the matrix splitting $\widetilde{A}_\ell=\widetilde{M}_\ell-\widetilde{N}_\ell$ in \eqref{eqn:M_lowrank_gibbs}}\label{alg:low_rank_gibbs}
    \begin{algorithmic}[1]
        \Procedure{$\mathrm{GibbsSmoother}$}{$A_\ell$, $B_{\ell}$, $\Gamma$, $f_\ell$, $\omega$, direction, $\theta_\ell$}
        \If{$\text{direction} = \text{forward}$}
        \State{Set $M_\ell = \frac{1}{\omega}D_\ell + L_\ell$}
        \Else
        \State{Set $M_\ell = \frac{1}{\omega}D_\ell + L_\ell^\top$}
        \EndIf
        \State{Solve the $\beta$ triangular systems $ M_\ell C_\ell = B_{\ell}$ for $C_\ell$.}
        \State{Compute $G_{\ell}=C_\ell\left(\Gamma + B_{\ell}^{\top}C_{\ell} \right)^{-1}$}
        \State {Draw $\xi^{\mathrm{diag}}_\ell\sim \mathcal{N}(0,D_\ell)$, $\xi^\mathrm{LR}_\ell\sim \mathcal{N}(0,\Gamma^{-1})$}
        \State{Compute the residual $r_\ell = f_\ell + \sqrt{\frac{2-\omega}{\omega}} \xi_\ell^{\mathrm{diag}} + B_{\ell} \xi_\ell^{\mathrm{LR}} - A_\ell \theta_\ell$}
        \State{Solve the triangular system $M_\ell \theta^*_\ell = r_\ell$ for $\theta_\ell^*$}
        \State{Set
            \begin{equation}
                \theta'_\ell   = \theta_\ell^* - G_{\ell} (B_{\ell}^{\top} \theta_\ell^*)\label{eqn:low_rank_update}
            \end{equation}}
        \State\Return {$\theta'_\ell$}
        \EndProcedure
    \end{algorithmic}
\end{algorithm}
\begin{algorithm}[t!]
    \caption{Symmetric Gibbs smoother with low rank correction for Bayesian inference.\\
        Computes the update $\theta_\ell\mapsto\theta'_\ell$ by using a symmetric combination of two calls to Alg.~\ref{alg:low_rank_gibbs}.}\label{alg:SGS}
    \begin{algorithmic}[1]
        \Procedure{$\mathrm{SymmetricGibbsSmoother}$}{$A_\ell$, $B_{\ell}$, $\Gamma$, $f_\ell$, $\theta_\ell$}
        \State{Compute $\theta_\ell^* = \mathrm{GibbsSmoother}(A_\ell,B_{\ell},\Gamma,f_\ell, \mathrm{forward}, \theta_\ell)$}
        \State{Compute $\theta_\ell' = \mathrm{GibbsSmoother}(A_\ell,B_{\ell},\Gamma,f_\ell, \mathrm{backward}, \theta_\ell^*)$}
        \State \Return $\theta_\ell'$
        \EndProcedure
    \end{algorithmic}
\end{algorithm}

According to Lemma \ref{lem:symmetric_random_smoother} we can construct a \textit{symmetric} Gibbs smoother with low rank correction by combining a forward Gibbs sweep and a backward Gibbs sweep; for future reference this is written down in Alg.~\ref{alg:SGS}. For simplicity, we only consider the case $\omega=1$ in the following. In this case \eqref{eqn:symmetric_M} and some straightforward algebra that uses the definition of $\widetilde{M}_\ell$ in \eqref{eqn:M_lowrank_gibbs} shows that the matrix of the resulting symmetric splitting method is
\begin{equation}
    \begin{aligned}
        \widetilde{M}_\ell^{\text{(SGS)}} & = \widetilde{M}_\ell \left(\widetilde{M}_\ell+\widetilde{M}_\ell^\top-\widetilde{A}_\ell\right)^{-1} \widetilde{M}_\ell^\top \\
                                          & = \widetilde{M}_\ell \left(D_\ell + B_{\ell} \Gamma^{-1}B_{\ell}^{\top}\right)^{-1} \widetilde{M}_\ell^\top                        %
        = \widetilde{A}_\ell + L_\ell \left(D_\ell + B_{\ell} \Gamma^{-1}B_{\ell}^{\top}\right)^{-1}L_\ell^\top.\label{eqn:SGS_splitting}
    \end{aligned}
\end{equation}
We conclude this section by showing that the seemingly simpler approach of setting $M = \widetilde{M}_\ell :=\frac{1}{\omega} D_\ell + L_\ell$ (instead of \eqref{eqn:M_lowrank_gibbs}) does not work. While this splitting avoids the computation of the low-rank correction in \eqref{eqn:low_rank_update}, the problem is that the resulting matrix $M+M^\top-A = \widetilde{M}_\ell+\widetilde{M}_\ell^\top - \widetilde{A}_\ell = \frac{2-\omega}{\omega}D_\ell - B_{\ell} \Gamma^{-1}B_{\ell}^{\top}$ is not in general positive definite. As a result it is not possible to sample from $\mathcal{N}(0,M+M^\top-A)$ which is required in \eqref{eqn:random_smoother_update}.
\subsection{Coarse level sampler}\label{sec:coarse_level_sampler}
On the coarsest level we have to make a choice for the $\mathrm{CoarseSampler()}$ procedure that generates a new sample $\theta'_0$ in line 3 of Alg.~\ref{alg:mgmc}. Using Alg.~\ref{alg:coarse_sampler}, which consists of repeated applications of the random smoother in Alg.~\ref{alg:random_smoother}, corresponds to the approach in \cite{goodman1989multigrid}. It turns out that this choice is sufficient to guarantee that Alg.~\ref{alg:mgmc} leaves the target distribution invariant (see Section \ref{sec:theory_invariance}).
\begin{algorithm}[t!]
    \caption{Coarse level sampler $\theta_0 \mapsto\theta'_0$\\
    Apply $\nu_0$ smoothing steps of Alg.~\ref{alg:low_rank_gibbs} with the splitting matrix $\widetilde{M}_0^{\mathrm{coarse}}$.}\label{alg:coarse_sampler}
    \begin{algorithmic}[1]
        \Procedure{$\mathrm{RandomCoarseSmoother}$}{$\widetilde{A}_{0},f_{0},\theta_{0}$}
        \State  Let $\theta_{0,0}:=\theta_{0}$
        \For{$j=0,\dots,\nu_0-1$}
        \State $\theta_{0,j+1}:=\mathrm{RandomSmoother}(\widetilde{A}_{0},\widetilde{M}_{0}^{\mathrm{coarse}},f_{0},\theta_{0,j})$\Comment{Coarse-smoothing (Alg.~\ref{alg:random_smoother})}
        \EndFor
        \State\Return $\theta'_0:=\theta_{0,\nu_0}$
        \EndProcedure
    \end{algorithmic}
\end{algorithm}
Alternatively, we could sample $\theta'_0$ directly from the multivariate normal distribution $\mathcal{N}(\mu_0,\widetilde{A}_0^{-1})$ with $\mu_0=\widetilde{A}_0^{-1}$ by constructing the Cholesky factorisation of $\widetilde{A}_0$; this is explained in more detail in Section\ref{sec:samplers}. Since $n_0\ll n_\ell$ this is much cheaper than using the Cholesky sampler on the fine level $L$.

The convergence theory in Section\ref{sec:theory_convergence} assumes that the coarse sampler is exact, i.e. $\theta'_0\sim\mathcal{N}(\mu_0,\widetilde{A}_0^{-1})$, but in practice this is not necessary. As for the standard multigrid solver, the analysis could also be extended to an inexact coarse sampler.

\subsection{Cost analysis\label{subsec:Cost-analysis}}
In this section, we derive an upper bound on the cost of one MGMC update in Alg.~\ref{alg:mgmc} as a function of the problem size $n_L$ and the rank $\beta\ll n_L$ of the measurement operator $B_L$. \rev{We assume that this rank is fixed and independent of $L$. The value of $\beta$ depends on the application; in our numerical experiments in Section~\ref{sec:results} we set $\beta=8$ in $d=2$ dimensions and $\beta=32$ in $d=3$.} We analyse this cost under the assumption that the matrices $A_\ell$, $I_{\ell-1}^{\ell}$ and $I_{\ell}^{\ell-1}$ are sparse, with some upper bound on the number of entries per row, independent of the problem size. For the local basis functions in finite element spaces considered here this is naturally the case. Consequently, any matrix-vector products with $A_\ell$
or $I_{\ell-1}^{\ell}$ incur a cost of $\mathcal{O}(n_\ell)$, in the same way as any tridiagonal solve $(\omega^{-1}D_\ell+L_\ell)\theta_\ell=r_\ell$ or $(\omega^{-1}D_\ell+L_\ell^\top)\theta_\ell=r_\ell$ for a single right-hand side $r_\ell\in\mathbb{R}^{n_\ell}$ in $\mathcal{O}(n_\ell)$.
Under these assumptions the computational cost of one Multigrid Monte Carlo update in Alg.~\ref{alg:mgmc} grows in proportion to the problem size $n_L$, as we will see in the following.

We start by analysing the cost of the random smoother in Alg.~\ref{alg:low_rank_gibbs}.
Generating the random vector $\xi_\ell^{\mathrm{diag}}\in\mathbb{R}^{n_\ell}$ incurs a cost of $\mathcal{O}(n_\ell)$ and generating $\xi_\ell^{\mathrm{LR}}\in\mathbb{R}^{\beta}$ a cost of $\mathcal{O}(\beta^{p_\Gamma})$ where $p_\Gamma=1$ if $\Gamma$ is diagonal and $p_\Gamma=2$ otherwise. Computation of $B_{\ell}\xi_\ell^{\text{LR}}$ in line 10 is $\mathcal{O}(\beta n_\ell)$, while all other steps in the computation of $r_\ell$ and the tridiagonal solve for $\theta_\ell^*$ in line 11 are $\mathcal{O}(n_\ell)$. Since the matrices $B_{\ell}$ and $G_{\ell}$ are both of size $n_\ell\times\beta$, the cost of the low-rank update in \eqref{eqn:low_rank_update} is $\mathcal{O}(\beta n_\ell)$.

Thus, the cost of one application of Alg.~\ref{alg:low_rank_gibbs} can be bounded by
\begin{equation}
    \mathrm{Cost}_{\mathrm{Gibbs}}(n_\ell,\beta) \le C_{\mathrm{Gibbs}}^{(1)} \beta^{p_\Gamma} + C_{\mathrm{Gibbs}}^{(2)} (1+\beta) n_\ell\label{eqn:smoother_cost}
\end{equation}
for some constants $C_{\mathrm{Gibbs}}^{(1)}$ and $C_{\mathrm{Gibbs}}^{(2)}$ independent of $n_\ell$ and $\beta$.

Prior to restriction we also need to compute the residual $f_\ell - \widetilde{A}_\ell \theta_{\ell,\nu_1}$, which again incurs a cost of $\mathcal{O}((1+\beta) n_\ell)$. To see this note that $\widetilde{A}_\ell$ defined in \eqref{eq:post-cov} also contains the correction term $B_{\ell} (\Gamma^{-1}B_{\ell}^{\top})\theta_{\ell,\nu_1}$, multiplication with which has a cost of $\mathcal{O}(\beta n_\ell)$ since the $\beta\times n_\ell$ matrix $\Gamma^{-1}B_{\ell}^{\top}$ can be precomputed.
Overall, the additional costs from prolongation, restriction and residual calculation are
\begin{equation}
    \mathrm{Cost}_{\mathrm{other}}(\beta,n_\ell) \le C_{\mathrm{other}}(1+\rev{\beta})n_\ell\label{eqn:coarse_other}
\end{equation}
for some constant $C_{\mathrm{other}}$ again independent of $n_\ell$ and $\beta$.

Depending on whether Alg.~\ref{alg:coarse_sampler} or a direct Cholesky sampler is used in line 3 of Alg.~\ref{alg:mgmc}, drawing a sample on the coarsest level will incur a cost of
\begin{equation}
    \mathrm{Cost}_{\mathrm{coarse}}(\beta,\nu_0,n_0) \le
    \begin{cases}
        \Big(C_{\mathrm{Gibbs}}^{(1)}\beta^{p_\Gamma} + C_{\mathrm{Gibbs}}^{(2)}(1+\beta)n_0\Big)\nu_0 & \text{(Alg.~\ref{alg:coarse_sampler})} \\
        \mathrm{Cost}_{\mathrm{Cholesky}}(n_0)                & \text{(Cholesky sampler)}.
    \end{cases}
    \label{eqn:coarse_cost}
\end{equation}
The cost of generating a single sample with the Cholesky method depends on the implementation. Naively, if $\widetilde{A}_0$ is treated as a dense matrix, $\mathrm{Cost}_{\mathrm{Cholesky}}(n_0)=\mathcal{O}(n_0^2)$ since the Cholesky sampler requires the inversion of triangular $n_0\times n_0$ matrices. However, for a sparse matrix $\widetilde{A}_0$ reordering can reduce the computational complexity significantly.

Putting everything together we arrive at the following result:
\begin{thm}[Cost of Multigrid Monte Carlo update]
    \label{thm:mgmc_cost}
    Assume that the problem size $n_{\ell}$ decreases geometrically on the coarser levels, i.e. there is a constant $\rho_G$ such that
    \begin{equation}
        n_{\ell-1}\, \le\, \rho_G\, n_\ell\quad\text{and}\quad \rho_G\, \gamma \,<\, 1
        \qquad\text{for\;\;$\gamma:=\max_{\ell=0,\dots,L}\{\gamma_\ell\}$}
        \label{eqn:dim-decrease}.
    \end{equation}
    Then there exists a constant $C_{\mathrm{MG}}$, which only depends on $\nu_0$, $\nu_1$, $\nu_2$, $\gamma\rho_G$ and $n_0$, such that the cost of one application of Alg.~\ref{alg:mgmc} can be bounded by
    \begin{equation}
        \mathrm{Cost}_{\mathrm{MGMC}}(L) \le 
        C_{\mathrm{MG}}  n_L\cdot \begin{cases}
        1+\beta & \text{if $\Gamma$ is diagonal}\\
        \rev{\beta^{p_\Gamma}} & \text{otherwise}.
        \end{cases}
        \label{eqn:cost_thm}
    \end{equation}
\end{thm}
\begin{proof}
    The cost of one MGMC update in Alg.~\ref{alg:mgmc} on level $\ell$ can be bounded recursively as
    \begin{equation*}
        \mathrm{Cost}_{\mathrm{MGMC}}(\ell) %
        \le \begin{cases} C_{\mathrm{other}} (1+\rev{\beta}) n_\ell + \nu \mathrm{Cost}_{\mathrm{Gibbs}}(n_\ell,\beta) +\gamma \mathrm{Cost}_{\mathrm{MGMC}}(\ell-1) & \text{for $\ell>0$} \\
              \mathrm{Cost}_{\mathrm{coarse}}(\beta,\nu_0,n_0)                                                                                            & \text{for $\ell=0$}\end{cases}
    \end{equation*}
    where $\nu:=\nu_1+\nu_2$ is the total number of random smoothing steps for  $\ell \ge 1$.
    Inserting the bounds from \eqref{eqn:smoother_cost}, \eqref{eqn:coarse_other} and \eqref{eqn:coarse_cost} this leads to
    \begin{equation}
        \mathrm{Cost}_{\mathrm{MGMC}}(L) %
        \le \sum_{\ell=1}^{L} \gamma^{L-\ell} \Big(  C_{\mathrm{Gibbs}}^{(1)} \nu \beta^{p_\Gamma} + 
        \rev{(1+\beta)C_{\mathrm{MG}}^{(1)}(\nu)} \cdot n_{\ell}\Big)
        + C_{\mathrm{MG}}^{(2)}(\beta,\nu_0,n_0)\gamma^L n_0\label{eqn:cost_sum_bound}
    \end{equation}
    where\vspace{-1ex}
    \begin{equation}
        \begin{aligned}
            C_{\mathrm{MG}}^{(1)} (\nu)       & := \rev{C_{\mathrm{other}}  + C_{\mathrm{Gibbs}}^{(2)}}\nu                                                                      \\
            C_{\mathrm{MG}}^{(2)} (\beta,\nu_0,n_0) & := \begin{cases}
                                                           \left(C_{\mathrm{Gibbs}}^{(1)}\beta^{p_\Gamma}/n_0 + C_{\mathrm{Gibbs}}^{(2)}(1+\beta)\right)\nu_0 & \text{(Alg.~\ref{alg:coarse_sampler})} \\
                                                           \mathrm{Cost}_{\mathrm{Cholesky}}(n_0)/n_0                                                                    & \text{(Cholesky)}
                                                       \end{cases}
        \end{aligned}
    \end{equation}
    From \eqref{eqn:dim-decrease} it follows that $n_{\ell}\le \rho_G^{L-\ell}n_L$ and $1\le \rho_G^{L-\ell}n_{L}/n_0$ for all $0\le\ell\le L$. Thus, we can bound the total cost in \eqref{eqn:cost_sum_bound} as
    \begin{equation}
        \mathrm{Cost}_{\mathrm{MGMC}}(L) %
        \ \le \ C'_{\mathrm{MG}}(\beta,\nu_0,\nu,n_0) \sum_{\ell=0}^{L-1} (\gamma \rho_G)^{\ell} \; n_L \ \le \
        \frac{C'_{\mathrm{MG}}(\beta,\nu_0,\nu,n_0)}{1 - \gamma \rho_G} \; n_L
        \label{eqn:mgmc_cost_bound}
    \end{equation}
    with
    \begin{equation}
        C'_{\mathrm{MG}}(\beta,\nu_0,\nu,n_0) :=  C_{\mathrm{Gibbs}}^{(1)} \nu \beta^{p_\Gamma} n_0^{-1}   + \rev{(1+\beta)C_{\mathrm{MG}}^{(1)} (\nu)} + C_{\mathrm{MG}}^{(2)} (\beta,\nu_0,n_0)  .
        \label{eqn:mgmc_cost_constant}
    \end{equation}
    The final bound in \eqref{eqn:cost_thm} then follows directly from \eqref{eqn:mgmc_cost_bound} and \eqref{eqn:mgmc_cost_constant}.
\end{proof}
\begin{rem}
    The only two practically relevant recursion schemes are the V-cycle ($\gamma=1$) and the W-cycle ($\gamma=2$). Thus, when $A_L$ arises from the discretisation of a differential operator, the assumption in \eqref{eqn:dim-decrease} is naturally satisfied. Consider for example a lowest order FE discretisation on a uniform rectangular grid for which $\rho_G = 2^{-d}\le \frac{1}{4}$ in $d\ge2$ dimensions. Then, $\gamma\rho_G\le\frac{1}{2}<1$ for $\gamma\le 2$.
\end{rem}
\begin{rem}
    \rev{%
    As will be discussed in Section~\ref{sec:theory}, the parameters $\nu_0$, $\nu_1$, $\nu_2$ and $\gamma$ are independent of the number of levels $L$. For fixed values of these parameters, the convergence rate of the MGMC sampler is independent of $L$. Hence, the constant $C_{\mathrm{MG}}$ in \eqref{eqn:cost_thm} is independent of $L$ as well.}
\end{rem}
\rev{\subsubsection{Setup costs}\label{sec:mgmc_setup_costs}
  In addition to the above online costs that are incurred in every MGMC iteration, there are also offline costs for precomputing required quantities such as the matrices $C_\ell$ and $G_\ell$ in Alg.~\ref{alg:low_rank_gibbs}, as well as the Cholesky factorisation for the coarsest level. These setup costs are in general amortised when generating a large number of samples, but we quantify them here for completeness. In particular, the following theorem shows that the setup costs of Alg.~\ref{alg:mgmc} grow linearly in the problem size $n_L$.}
\rev{\begin{thm}[MGMC setup cost]
    \label{thm:mgmc_setup_cost}
    Assume that (as in Theorem~\ref{thm:mgmc_cost}) the problem size $n_{\ell}$ decreases geometrically on the coarser levels, i.e. there is a constant $\rho_G<1$ such that
    \begin{equation}\label{eq:geometric}
        n_{\ell-1}\, \le\, \rho_G\, n_\ell \quad \text{for all $\ell=1,2,\dots,L$}.
    \end{equation}
    Then there exist constants $C_{\mathrm{MG}}^{(\mathrm{setup},1)}$ and $C_{\mathrm{MG}}^{(\mathrm{setup},2)}$, which only depend on $\rho_G$, such that the setup cost of Alg.~\ref{alg:mgmc} can be bounded by
    \begin{equation}
    \begin{aligned}
    \mathrm{Cost}^{(\mathrm{setup})}_{\mathrm{MGMC}}(L) &\le  C_{\mathrm{MG}}^{(\mathrm{setup},1)} \beta^2 n_L +C_{\mathrm{MG}}^{(\mathrm{setup},2)} \beta^3 L  + \mathrm{Cost}_{\mathrm{coarse}}^{(\mathrm{setup})}(\beta,n_0).
    \end{aligned}
    \end{equation}
    Depending on the choice in line 3 of Alg.~\ref{alg:mgmc}, the setup cost on the coarsest level is bounded by
\begin{equation}
    \mathrm{Cost}^{(\mathrm{setup})}_{\mathrm{coarse}}(\beta,n_0) \le
    \begin{cases}
        C_{\mathrm{Gibbs}}^{(\mathrm{setup},1)} \beta^{3} + C_{\mathrm{Gibbs}}^{(\mathrm{setup},2)} \beta^2 n_0 & \text{(Alg.~\ref{alg:coarse_sampler})} \\
        \mathrm{Cost}^{(\mathrm{setup})}_{\mathrm{Cholesky}}(n_0)                & \text{(Cholesky sampler)}.
    \end{cases}
    \label{eqn:coarse_setup_cost}
\end{equation}
for some constants $C_{\mathrm{Gibbs}}^{(\mathrm{setup},1)}$, $C_{\mathrm{Gibbs}}^{(\mathrm{setup},2)}$ and where $\mathrm{Cost}^{(\mathrm{setup})}_{\mathrm{Cholesky}}(n_0)$ is the cost for computing the Cholesky factorisation of the $n_0\times n_0$ matrix $\widetilde{A}_0$.
\end{thm}}
\noindent
\rev{Proof. See Appendix~\ref{sec:proof_mgmc_setup_costs}.}

\rev{Naively the setup cost of the coarse level Cholesky factorisation is $\mathcal{O}(n_0^3)$, but the complexity will be reduced if the sparsity of the $n_0\times n_0$ matrix $\widetilde{A}_0$ can be exploited. However, since $n_0 \ll n_L$, the setup cost on the coarsest level
  $\mathrm{Cost}_{\mathrm{coarse}}^{(\mathrm{setup})}(\beta,n_0)$ is in general neligible.}
\rev{\subsubsection{Memory requirements}\label{sec:mgmc_memory_requirements}
As the following theorem shows, the memory requirements of Alg.~\ref{alg:mgmc} also grow linearly in $n_L$.}
\rev{\begin{thm}[MCMC memory requirements]
    \label{thm:mgmc_memory}
    Assume that (as in Theorem~\ref{thm:mgmc_cost}) the problem sizes $n_{\ell}$ satisfy \eqref{eq:geometric}.
    Then there exists a constant $C_{\mathrm{MG}}^{(\mathrm{mem})}$, which only depends on $\rho_G$, such that the total memory requirement of Alg.~\ref{alg:mgmc} can be bounded by
    \begin{equation}
    \begin{aligned}
    \mathrm{Memory}_{\mathrm{MGMC}}(L) &\le C_{\mathrm{MG}}^{(\mathrm{mem})}(1+\beta)n_L + \beta L + \mathrm{Memory}_{\mathrm{coarse}}(\beta,n_0).
    \end{aligned}
    \end{equation}
    Depending on the choice in line 3 of Alg.~\ref{alg:mgmc}, the memory requirements on the coarsest level are bounded by
\begin{equation}
    \mathrm{Memory}_{\mathrm{coarse}}(\beta,n_0) \le
    \begin{cases}
        C_{\mathrm{MG}}^{(\mathrm{mem})} (1+\beta) n_0 + \beta & \text{(Alg.~\ref{alg:coarse_sampler})} \\
        \mathrm{Memory}_{\mathrm{Chol}}(n_0)               & \text{(Cholesky sampler)}.
    \end{cases}
    \label{eqn:coarse_storage}
\end{equation}
where $\mathrm{Memory}_{\mathrm{Chol}}(n_0)$ is the memory required to store the Cholesky factorisation of $\widetilde{A}_0$.
\end{thm}}
\noindent
\rev{Proof. See Appendix~\ref{sec:proof_mgmc_memory}.}

\rev{The memory footprint of the Cholesky factorisation is $\mathcal{O}(n_0^2)$.
It can be reduced by exploiting the sparsity of $\mathcal{A}_0$. Because of $n_0\ll n_L$ the storage requirements on the coarsest level are negligible.}
\section{Theory\label{sec:theory}}
This section presents our main theoretical results. We first put the relationship between random samplers and linear solvers already mentioned in Section~\ref{sec:random_vs_deterministic_smoothers} into a wider context: in Section~\ref{subsec:matrix-splitting} we argue that according to the pioneering work in \cite{goodman1989multigrid,Fox.C_Parker_2017_AcceleratedGibbsSampling}, the convergence of the Multigrid Monte Carlo sampler is closely related to the convergence of deterministic multigrid solvers for linear systems of equations. This is made more explicit in Section~\ref{sec:theory_invariance}, where we show that the Multigrid Monte Carlo update Alg. \ref{alg:mgmc} leaves the target distribution invariant, and Section~\ref{sec:theory_convergence}, where we prove grid-independent convergence of the algorithm. Throughout the discussion we highlight the connection with the theory for deterministic multigrid methods; to streamline the presentation we defer proofs to the appendices. Finally, in Section~\ref{subsec:Bayesian}, we apply and extend the theory to conditional distributions and Bayesian inverse problems introduced in Section~\ref{subsec:Linear-Bayesian-inverse}. We conclude with a result on the optimality of the MGMG sampler for drawing samples from the target distribution in Section~\ref{sec:optimality}.
\subsection{Random \rev{sampling based on matrix splittings}\label{subsec:matrix-splitting}}
Let $A=A^\top$ be a symmetric positive definite matrix. As pointed out in \cite{Fox.C_Parker_2017_AcceleratedGibbsSampling}, the problem of sampling from the multivariate normal distribution $\mathcal{N}(\mu,A^{-1})$ with mean $\mu=A^{-1}f$ and covariance $A^{-1}$ is closely related to the problem of finding the solution of a linear system with the precision matrix $A$. More specifically, assume that we want to generate samples $\theta\sim p(\theta)d\theta$ with the density
\begin{equation}
	p(\theta)=\frac{1}{Z}\exp\left[-H(\theta)\right]\qquad\text{where \ensuremath{H(\theta)=\frac{1}{2}\theta^{\top}A\theta-f^{\top}\theta}}\label{eqn:sampling_problem}
\end{equation}
The corresponding linear algebra problem is to find the solution $u$ of the linear equation
\begin{equation}
	Au=f.\label{eqn:theta_deterministic}
\end{equation}
Note that the solution $u=A^{-1}f$ in \eqref{eqn:theta_deterministic} can be interpreted as the most likely state in the sampling problem, i.e. $u=\text{argmax}_{\theta\in\mathbb{R}^{n}}\{p(\theta)\}=\text{argmin}_{\theta\in\mathbb{R}^{n}}\{H(\theta)\}$. As already pointed out in Section~\ref{sec:random_vs_deterministic_smoothers}, the solution of \eqref{eqn:theta_deterministic} can be obtained by splitting the matrix $A=:M-N$  and iterating
\begin{equation}
	u^{(k+1)}=M^{-1}Nu^{(k)}+M^{-1}f,\qquad k=0,1,\dots\label{eqn:splitting_iteration}
\end{equation}
This implies that the error $u^{(k+1)}-u$ satisfies
\begin{equation}
	u^{(k+1)}-u = X (u^{(k)}-u)\qquad\text{with $X:= M^{-1}N = \id - M^{-1}A$.}\label{eqn:deterministic_error_iteration}
\end{equation}
Under certain conditions on the splitting defined by the matrix $M$ (see e.g.\ \cite{saad2003iterative} for details) the iterates $u^{(k)}$ converge to the solution of \eqref{eqn:theta_deterministic}. The corresponding iteration for sampling from the distribution with density given by \eqref {eqn:sampling_problem} is
\begin{equation}
	\theta^{(k+1)}=M^{-1}N\theta^{(k)}+M^{-1}(f+\xi^{(k)}),\qquad k=0,1,\dots\label{eqn:theta_random}
\end{equation}
where $\{\xi^{(k)}\}_{k=0,1,\dots}$ with $\xi^{(k)}\sim\mathcal{N}(0,M^{\top}+N)$ is a collection of independent and identically distributed random variables. Note that, in analogy to \eqref{eqn:deterministic_error_iteration}, \eqref{eqn:theta_random} can be written as
\begin{equation}
	\theta^{(k+1)}-\mu=X(\theta^{(k)}-\mu)+M^{-1}\xi^{(k)}\qquad\text{with $\mu=A^{-1}f$}.
	\label{eqn:random_error_iteration}
\end{equation}
and the same matrix $X$ as in \eqref{eqn:deterministic_error_iteration}. As shown in \cite{Fox.C_Parker_2017_AcceleratedGibbsSampling}, the following two statements are equivalent:
\begin{enumerate}
	\item The sequence $(u^{(k)})_{k=0,1,\dots}$ defined by the deterministic iteration in \eqref{eqn:splitting_iteration} converges to the solution $u_{L}=A^{-1}f_{L}$ of \eqref{eqn:theta_deterministic}.
	\item The Markov chain $(\theta^{(k)})_{k=0,1,\dots}$ defined by the random update in \eqref{eqn:theta_random} converges to $\theta\sim\mathcal{N}(\mu,A^{-1})$ with $\mu=A^{-1}f$ in distribution; moreover the first and second moments converge to the targets $\mu$ and $A^{-1}$, respectively.
\end{enumerate}
Furthermore, the update in \eqref{eqn:theta_random} leaves the target distribution invariant:
\begin{enumerate}
	\setcounter{enumi}{2}
	\item If $\theta^{(k)}\sim \mathcal{N}(\mu,A^{-1})$ then $\theta^{(k+1)}\sim \mathcal{N}(\mu,A^{-1})$.
\end{enumerate}
The deterministic equivalent of the MGMG update in Alg.~\ref{alg:mgmc} is a standard multigrid cycle. Crucially, both algorithms can be expressed in the form in \eqref{eqn:splitting_iteration} and \eqref{eqn:theta_random} respectively with the \textit{same splitting matrix $M$}. Hence, as already pointed out in \cite{goodman1989multigrid}, the convergence analysis of the Multigrid Monte Carlo algorithm proceeds along the same lines as the well-established analysis of the corresponding multigrid method (see e.g. \cite{Hackbusch.W_1985_MultiGridMethodsApplications,Hackbusch.W_2016_IterativeSolutionLarge}). The crucial observation is that the same iteration matrix $X$ arises in the deterministic multigrid error iteration in \eqref{eqn:deterministic_error_iteration} and in the Multigrid Monte Carlo equivalent in \eqref{eqn:random_error_iteration}; the explicit form of this matrix will be discussed in Section~\ref{sec:theory_convergence}. Before doing this, it is instructive to first show that the MGMC update in Alg.~\ref{alg:mgmc} indeed leaves the target distribution invariant.
\subsection{Invariance of the target distribution under MGMC updates}\label{sec:theory_invariance}
The following theorem formalises the invariance argument that is already presented in \cite{goodman1989multigrid}.
\begin{thm}\label{thm:mgmc_invariance}
	Assume that $A_{\ell-1} = I_\ell^{\ell-1} A_\ell I_{\ell-1}^\ell$ and a random smoother of the form in Alg.~\ref{alg:random_smoother} is used in Alg.~\ref{alg:mgmc}. Let $\mu_\ell = A_\ell^{-1}f_\ell$. Then, the Multigrid Monte Carlo update in Alg.~\ref{alg:mgmc} leaves the target distribution invariant, i.e.
	\begin{equation}
		\theta_\ell \sim \mathcal{N}(\mu_\ell,A_\ell^{-1})  \quad \Rightarrow \quad
		\theta'_\ell = \mathrm{MGMC}_{\ell}(A_\ell,f_\ell,\theta_\ell) \sim \mathcal{N}(\mu_\ell,A_\ell^{-1}).
		\label{eqn:mgmc_invariance}
	\end{equation}
\end{thm}
To prove Theorem \ref{thm:mgmc_invariance} we first show that random smoothers based on matrix splittings leave the target distribution invariant.
\begin{prop}\label{prop:smoother_invariance}
	The random smoother in Alg.~\ref{alg:random_smoother} leaves the multivariate normal distribution with mean $\mu=A^{-1}f$ and covariance $A^{-1}$ invariant: $\theta \sim \mathcal{N}(\mu,A^{-1})$ implies that $\theta' \sim \mathcal{N}(\mu,A^{-1})$.
\end{prop}
\begin{proof}
	If $\theta\sim \mathcal{N}(\mu,A^{-1})$ then $\theta'$ defined by \eqref{eqn:random_smoother_update} has a multivariate normal distribution because it is the linear combination of two multivariate normal random variables $\theta$ and $\xi$. Since a multivariate normal distribution is uniquely defined by its first two moments, it is therefore sufficient to show that $\mathbb{E}[\theta']=\mu$ and $\mathbb{E}[(\theta'-\mu)(\theta'-\mu)^\top]=A^{-1}$. These two identities follow with some straightforward algebra which exploits the linearity of expectation values and the independence of $\theta$ and $\xi$ in \eqref{eqn:random_smoother_update}; see Appendix \ref{sec:proof_prop:smoother_invariance} for details.
\end{proof}
Next, we consider the coarse level correction, i.e. lines 9--14 in Alg.~\ref{alg:mgmc}. We show that given a sample $\theta_{\ell}$ on $\mathbb{R}^{n_{\ell}}$ from the target distribution, partially resampling $\psi_{\ell-1}$ on $\mathbb{R}^{n_{\ell-1}}$ from the correct distribution and adding this to $\theta_{\ell}$ as $\theta_{\ell}+I_{\ell-1}^{\ell}\psi_{\ell-1}$ (line 14 of Alg.~\ref{alg:mgmc}) does not change the distribution. This result is essentially due to Liu and Sabatti \cite[Theorem 3]{Liu.J.S_Sabatti_2000_generalised}, where the result was presented more abstractly. Since the results and proof in \cite{Liu.J.S_Sabatti_2000_generalised} may seem somewhat opaque to a broader audience in numerical analysis, we recall the main points here and replicate the proof in Appendix \ref{sec:proof_MGMC_coarse_level_invariance} for completeness.

Consider a random field $\Theta_\ell$ on $\Lambda_{\ell}\subset\mathbb{R}^{n_{\ell}}$ distributed according to
\begin{equation}
	p_{\ell}(\theta_\ell)d\theta_\ell=\frac{1}{Z_\ell}h_{\ell}(\theta_{\ell})d\theta_\ell,
	\label{eqn:theta_distribution}
\end{equation}
for some non-negative function $h_{\ell}$ on $\mathbb{R}^{n_{\ell}}$ %
and with normalising constant $Z_\ell:=\int_{\mathbb{R}^{n_{\ell}}}h_{\ell}(\theta_\ell)d\theta_\ell$.
Define the conditional coarse level density
\begin{align}
	p_{\ell-1}(\psi_{\ell-1}|\theta_\ell) & :=\frac{1}{Z^*_{\ell-1}(\theta)}h_{\ell}(\theta_\ell+I_{\ell-1}^\ell\psi_{\ell-1}),\label{eq:coarse-density}
\end{align}
with normalisation constant
\begin{equation}
	Z^*_{\ell-1}(\theta_\ell):=\int_{\mathbb{R}^{n_{\ell-1}}}h_{\ell}(\theta_\ell+I_{\ell-1}^\ell \psi_{\ell-1})\,\mathrm{d}\psi_{\ell-1}.\label{eqn:Z_theta_normalisation}
\end{equation}
Let further $\Psi_{\ell-1}$ be a random field on $\Lambda_{\ell-1}\subset\mathbb{R}^{n_{\ell-1}}$ that has the density
\[
	p_{\ell-1}(\psi_{\ell-1})=\int_{\mathbb{R}^{n_{\ell}}}p(\psi_{\ell-1},\theta_{\ell})\mathrm{d}\theta_\ell:=\int_{\mathbb{R}^{n_{\ell}}}p_{\ell-1}(\psi_{\ell-1}|\theta_{\ell})p_{\ell}(\theta_\ell)\mathrm{d}\theta_\ell,
\]
where
\begin{equation}
	p(\psi_{\ell-1},\theta_{\ell}):=p_{\ell-1}(\psi_{\ell-1}|\theta_\ell)p_{\ell}(\theta_\ell)\label{eqn:joint_density}
\end{equation}
is the joint density of $\Theta_\ell$ and $\Psi_{\ell-1}$. The following proposition states that the distribution of the corrected random variable $\Theta_\ell^{*}=\Theta_\ell+I_{\ell-1}^\ell\Psi_{\ell-1}$ is the same as that of $\Theta_\ell$.
\begin{prop}[{\cite[Theorem 3]{Liu.J.S_Sabatti_2000_generalised}}]\label{prop:coarse_level_invariance}
	Let $\Theta_\ell$ and $\Psi_{\ell-1}$ be random variables with distributions $p_\ell$ and $p_{\ell-1}(\cdot|\theta_\ell)$, as defined in \eqref{eqn:theta_distribution} and \eqref{eq:coarse-density} respectively. Let $\theta_\ell^{*}:=\theta_\ell^{*}(\theta_\ell,\psi_{\ell-1}):=\theta_\ell+I_{\ell-1}^\ell \psi_{\ell-1}$ be a realisation of the random variable $\Theta_\ell^*:=\Theta_\ell +I_{\ell-1}^\ell\Psi_{\ell-1}$ with distribution $\mu_{\theta_\ell^*}$.

	For an arbitrary $A\in\mathscr{B}(\mathbb{R}^{n_{\ell}})$, where
	$\mathscr{B}(\mathbb{R}^{n_{\ell}})$ is the Borel $\sigma$-algebra
	on $\mathbb{R}^{n_{\ell}}$, let
	\[
		\mu_{\theta_\ell^{*}}(A)=\int_{\mathbb{R}^{n_{\ell}}}\int_{\mathbb{R}^{n_{\ell-1}}}1_{A}(\theta_\ell^{*}(\theta_\ell,\psi_{\ell-1}))p(\theta_\ell,\psi_{\ell-1})\mathrm{d}\psi_{\ell-1}\mathrm{d}\theta_\ell
	\]
	with the joint density $p_\ell$ defined in \eqref{eqn:joint_density}.
	Then,
	\[
		\mu_{\theta_\ell^{*}}(A)=\mu_{\theta_\ell}(A) =\frac{1}{Z_\ell}\int_{\mathbb{R}^{n_{\ell}}}1_{A}(\theta_\ell)h_{\ell}(\theta_\ell)\mathrm{d}\theta_\ell\,.
	\]
\end{prop}
\noindent\textit{Proof.} See Appendix~\ref{sec:proof_MGMC_coarse_level_invariance}.\\[1ex]
Proposition \ref{prop:coarse_level_invariance} does not make any assumptions on the distribution $p_\ell$. If $p_\ell$ is a multivariate normal distribution with mean $\mu_\ell=A_\ell^{-1}f_\ell$ and covariance $A_\ell^{-1}$ then the coarse level distribution $p_{\ell-1}(\cdot|\theta_\ell)$ is also multivariate normal.  Moreover, if $A_{\ell-1}$ and $f_{\ell-1}$ are of the form defined in Alg.~\ref{alg:mgmc}, then the resulting multivariate normal distribution has mean $\mu_{\ell-1}=A_{\ell-1}^{-1}f_{\ell-1}$ and covariance $A_{\ell-1}^{-1}$. These results follow immediately from the following statement.
\begin{prop}\label{prop:gaussian_coarse_level_correction}
	Let $p_{\ell}(\theta_\ell)d\theta_\ell=Z_\ell^{-1}\exp(-H_{\ell}(\theta_\ell))d\theta_\ell$
	be the density on level $\ell$ with
	\begin{equation}
		H_{\ell}(\theta_\ell):=\frac{1}{2}\theta_\ell A_\ell \theta_\ell - f_\ell^\top \theta_\ell
	\end{equation}
	where $A_{\ell}=A_\ell^\top$ is a positive definite symmetric matrix. Then the conditional coarse level density $p_{\ell-1}$ defined in \eqref{eq:coarse-density} can be written as
	\begin{equation}
		p_{\ell-1}(\psi_{\ell-1}|\theta_\ell)d\psi_{\ell-1}=\frac{1}{Z_{\ell-1}(\theta_\ell)}\exp(-H_{\ell-1}(\psi_{\ell-1}|\theta_\ell))d\psi_{\ell-1}
		\label{eqn:coarse_density_gaussian}
	\end{equation}
	with
	\[
		H_{\ell-1}(\psi_{\ell-1}|\theta_\ell):= \frac{1}{2}\psi_{\ell-1}^\top A_{\ell-1} \psi_{\ell-1} - f_{\ell-1}^\top \psi_{\ell-1}
	\]
	where
	\begin{xalignat}{3}
		A_{\ell-1} &= I_\ell^{\ell-1} A_{\ell} I_{\ell-1}^\ell, &
		f_{\ell-1} &= I_\ell^{\ell-1}(f_\ell - A_\ell \theta_\ell),&
		I_\ell^{\ell-1} &= (I_{\ell-1}^\ell)^\top\label{eqn:triple_product}
	\end{xalignat}
	and $Z_{\ell-1}(\theta_\ell)$ is a normalisation constant which ensures that $\int p_{\ell-1}(\psi_{\ell-1}|\theta_\ell)\;d\psi_{\ell-1} = 1$.
\end{prop}
\begin{proof}
	According to \eqref{eqn:theta_distribution} and \eqref{eq:coarse-density} with $h_\ell(\theta_\ell)=e^{-H_\ell(\theta_\ell)}$ the conditional coarse density $p_{\ell-1}(\cdot|\theta_\ell)$ is given by
	\[
		p_{\ell-1}(\psi_{\ell-1}|\theta_\ell) = \frac{1}{Z^*_{\ell-1}(\theta_\ell)}\exp[-H_\ell(\theta_\ell+I_{\ell-1}^\ell\psi_{\ell-1})]
	\]
	with the prolongation matrix $I_{\ell-1}^\ell$ and the normalisation constant $Z^*_{\ell-1}(\theta_\ell)$ in \eqref{eqn:Z_theta_normalisation}. Defining $A_{\ell-1}$ and $f_{\ell-1}$ as in \eqref{eqn:triple_product} and using the symmetry of $A_\ell$, the exponent can be expanded out as
	\begin{align*}
		H_\ell(\theta_\ell+I_{\ell-1}^\ell\psi_{\ell-1}) & =
		\frac{1}{2}(\theta_\ell + I_{\ell-1}^\ell \psi_{\ell-1})^\top A_{\ell}(\theta_\ell + I_{\ell-1}^\ell\psi_{\ell-1}) - f_\ell^\top (\theta_\ell+I_{\ell-1}^\ell \psi_{\ell-1})                                                                                                                              \\
		                                                 & = \frac{1}{2}\psi_{\ell-1}^\top I_\ell^{\ell-1} A_\ell I_{\ell-1}^\ell \psi_{\ell-1} - (f_\ell^\top I_{\ell-1}^\ell- \theta_\ell^\top A_\ell I_{\ell-1}^\ell) \psi_{\ell-1} + \frac{1}{2}\theta_\ell^\top A_\ell \theta_\ell - f_\ell^\top \theta_\ell \\
		                                                 & = \frac{1}{2} \psi_{\ell-1}^\top A_{\ell-1} \psi_{\ell-1} - f_{\ell-1}^\top \psi_{\ell-1} + \frac{1}{2}\theta_\ell^\top A_{\ell}\theta_\ell - f_{\ell}^\top \theta_\ell.
	\end{align*}
	This gives the desired result in \eqref{eqn:coarse_density_gaussian} with $
		Z_{\ell-1}(\theta_\ell) = Z^*_{\ell-1}(\theta_\ell) h_\ell(\theta_\ell)$.
\end{proof}
\noindent
We are now ready to show that Alg.~\ref{alg:mgmc} leaves the multivariate normal
distribution $\mathcal{N}(\mu_\ell,A_{\ell}^{-1})$ invariant. The proof uses the fact that the computation of the coarse level correction $\psi_{\ell-1}^{(m_{\max}+1)}$ in lines 9--13 of Alg.~\ref{alg:mgmc} is, up to the law, equivalent to Alg.~\ref{alg:coarse_alternative}.
\begin{algorithm}[t!]
	\caption{Coarse level correction, alternative form. Given $f_\ell$ and $\theta_{\ell,\nu_1}$, compute $\psi_{\ell-1}^{(\gamma_\ell)}$.}\label{alg:coarse_alternative}
	\begin{algorithmic}[1]
		\State {Compute the modified right-hand side $\widehat{f}_{\ell-1}:= I_\ell^{\ell-1} f_\ell$}
		\State {Let $\chi_{\ell-1}^{(0)}:= A_{\ell-1}^{-1}I_{\ell}^{\ell-1}A_\ell \theta_{\ell,\nu_1}$}
		\For{$m=0,1,\dots,\gamma_\ell-1$}
		\State $\chi_{\ell-1}^{(m+1)}:=\mathrm{MGMC}_{\ell-1}(A_{\ell-1},\widehat{f}_{\ell-1},\chi_{\ell-1}^{(m)})$
		\Comment{Recursive call to $\mathrm{MGMC}_{\ell-1}$}
		\EndFor
		\State {Set $\psi_{\ell-1}^{(\gamma_\ell)} := \chi_{\ell-1}^{(\gamma_\ell)} - A_{\ell-1}^{-1}I_{\ell}^{\ell-1} A_\ell \theta_{\ell,\nu_1}$}
	\end{algorithmic}
\end{algorithm}
\begin{proof}[Proof of Theorem \ref{thm:mgmc_invariance}]
	The result is shown by induction over $\ell$. On level $\ell=0$ the statement in \eqref{eqn:mgmc_invariance} is true by definition if an exact coarse level sampler is used to draw $\theta'_0\sim \mathcal{N}(\mu_0,A_0^{-1})$ for $\mu_0=A_0^{-1}f_0$. Otherwise, if Alg.~\ref{alg:coarse_sampler} and thus multiple applications
	of Alg.~\ref{alg:random_smoother} are used, then
	Proposition~\ref{prop:smoother_invariance} guarantees that $\theta_0 \sim \mathcal{N}(\mu_0,A_0^{-1})$ implies $\theta'_0\sim \mathcal{N}(\mu_0,A_0^{-1})$.

	Next consider $\ell>0$ and assume that the statement in \eqref{eqn:mgmc_invariance} holds on level $\ell-1$. According to Proposition \ref{prop:smoother_invariance}, if $\theta_\ell \sim\mathcal{N}(\mu_\ell,A_\ell^{-1})$ with $\mu_\ell=A_\ell^{-1}f_\ell$ then we also have that $\theta_{\ell,\nu_1}\sim\mathcal{N}(\mu_\ell,A_\ell^{-1})$
	and thus $\chi_{\ell-1}^{(0)}\sim\mathcal{N}(A_{\ell-1}^{-1}\widehat{f}_{\ell-1},A_{\ell-1}^{-1})$ in Alg.~\ref{alg:mgmc} with the right-hand side $\widehat{f}_{\ell-1} = I_\ell^{\ell-1} f_\ell$ restricted to level $\ell-1$. Since we assumed that Alg.~\ref{alg:mgmc} leaves the distribution invariant on level $\ell-1$, we also have that $\chi_{\ell-1}^{(\gamma_\ell)}\sim\mathcal{N}(A_{\ell-1}^{-1}\widehat{f}_{\ell-1},A_{\ell-1}^{-1})$. The distributions of $\psi_{\ell-1}^{(\gamma_\ell)}$ and $\chi_{\ell-1}^{(\gamma_\ell)}$ have the same covariance but different means, more specifically $\psi_{\ell-1}^{(\gamma_\ell)}\sim\mathcal{N}(\mu_{\ell-1},A_{\ell-1}^{-1})$. According to Proposition \ref{prop:coarse_level_invariance} this then implies that $\theta_{\ell,\nu_1+1}=\theta_{\nu_1}+I_{\ell-1}^{\ell}\psi_{\ell-1}^{(\gamma_\ell)} \sim \mathcal{N}(\mu_\ell,A_\ell^{-1})$. Finally, another application of Proposition \ref{prop:smoother_invariance} shows that $\theta'_{\ell} =\theta_{\ell,\nu_1+\nu_2+1}\sim \mathcal{N}(\mu_\ell,A_\ell^{-1})$.
\end{proof}

\subsection{Convergence of MGMC}\label{sec:theory_convergence}
We now formally show that the convergence of MGMC and the standard multigrid solver are equivalent. As a consequence, we can use classical multigrid theory to analyse MGMC convergence later in this section. To achieve this, we will show that the Multigrid Monte Carlo update $\theta_\ell \mapsto \theta'_\ell$ defined by Alg.~\ref{alg:mgmc} can be written in the form in \eqref{eqn:theta_random} with a suitable choice of the splitting matrix $M^{\mathrm{MG}}_\ell$ such that $A_\ell=M^{\mathrm{MG}}_\ell-N^{\mathrm{MG}}_\ell$ on level $\ell$.
The key quantity which determines the convergence of the iteration in \eqref{eqn:theta_random} is the iteration matrix $X$ defined in \eqref{eqn:random_error_iteration}. For a single MGMC update on level $\ell$ we denote this matrix by $X_\ell=(M_\ell^{\mathrm{MG}})^{-1} N_\ell^{\mathrm{MG}}$, and we will argue below that it can be constructed recursively by introducing the quantities $T_{\ell}$, $Q_{\ell}$ on all levels $\ell\ge 1$ such that
\begin{subequations}
	\begin{align}
		T_{\ell} & :=\id-I_{\ell-1}^{\ell}A_{\ell-1}^{-1}I_{\ell}^{\ell-1}A_{\ell},\label{eq:def-Tell}                                                                                                                                                   \\
		X_0      & := (S_0^{\mathrm{coarse}})^{\nu_0}.\label{eq:def-X0}\\
		X_{\ell} & =(S_{\ell}^{\mathrm{post}})^{\nu_{2}}Q_{\ell}(S_{\ell}^{\mathrm{pre}})^{\nu_{1}}\qquad\text{with}\qquad Q_{\ell}:=T_{\ell}+I_{\ell-1}^{\ell}X_{\ell-1}^{\gamma_{\ell}}A_{\ell-1}^{-1}I_{\ell}^{\ell-1}A_{\ell},\label{eq:def-Xell-Qell}
	\end{align}
\end{subequations}
In these expressions
\begin{equation}
	\begin{aligned}
		S_\ell^{\rho}         & := (M^{\rho}_{\ell})^{-1}N_{\ell}^{\rho},\quad\text{with $\rho\in\{\text{pre},\text{post}\}$}\quad\text{for $\ell\ge 1$}, \\
		S_0^{\mathrm{coarse}} & := (M^{\mathrm{coarse}}_0)^{-1}N_0^{\mathrm{coarse}},
	\end{aligned}\label{eqn:S_matrices}
\end{equation}
where $M^{\mathrm{coarse}}_{0}$, $M^{\mathrm{pre}}_{\ell}$ and $M^{\mathrm{post}}_{\ell}$ are the (invertible) splitting matrices which define the random coarse-, pre- and post-smoother, resp., with $A_{\ell}=M^{\mathrm{pre}}_{\ell}-N^{\mathrm{pre}}_{\ell}=M^{\mathrm{post}}_{\ell}-N^{\mathrm{post}}_{\ell}$ and $A_{0}=M^{\mathrm{coarse}}_{0}-N^{\mathrm{coarse}}_{0}$.

The matrices $T_{\ell}$ and $X_{\ell}$ introduced in \eqref{eq:def-Tell} and \eqref{eq:def-Xell-Qell} correspond to the two-grid correction matrix and the iteration matrix in the standard multigrid theory, respectively. We reiterate that it is no coincidence that exactly the same matrix $X_\ell$ shows up in \eqref{eqn:deterministic_error_iteration} if the update in \eqref{eqn:splitting_iteration} corresponds to one application of a standard multigrid cycle with suitable smoothers and coarse grid solver.

Before discussing the convergence of the MGMC iteration recall standard multigrid theory. If $X_{L}$ is defined by \eqref{eq:def-Tell} - \eqref{eq:def-X0}, the error in the multigrid iteration $u_{L}^{(m)}$
for solving $A_{L}u_{L}=f_{L}$ can be written as (c.f. \eqref{eqn:deterministic_error_iteration})
\begin{equation}
	u_{L}^{(m+1)}-A_{L}^{-1}f_{L}=X_{L}(u_{L}^{(m)}-A_{L}^{-1}f_{L}).\label{eqn:multigrid_error_iteration}
\end{equation}
Hence, if $\|X_{L}\|=:\sigma<1$ holds for any consistent matrix norm $\|\cdot\|$, then the error converges exponentially, i.e. $\|u_{L}^{(m)}-A_{L}^{-1}f_{L}\|
	<
	\sigma^{m}\|u_{L}^{(0)}-A_{L}^{-1}f_{L}\|$.
Classical multigrid theory (see e.g. \cite{Hackbusch.W_1985_MultiGridMethodsApplications}) states that the uniform bound on $\|X_{L}\|$ can be proven if the following two properties hold:
\begin{definition}[Smoothing property]\label{def:smoothing_property}
	A symmetric splitting matrix $M_\ell$ which is used for pre- and post-smoothing on level $\ell>0$ with \eqref{eqn:splitting_iteration} (in standard multigid) or \eqref{eqn:theta_random} (in MGMC) is said to satisfy the smoothing property if
	\[
		0<A_{\ell}\leq M_{\ell}
	\]
	where $A<B$ (resp. $A\leq B$) if and only if $B-A$ is positive definite (resp. positive semidefinite).
\end{definition}

\begin{definition}[Approximation property]\label{def:approximation_property}
	The multigrid iteration defined by matrices $A_{\ell}$ with $A_{\ell-1}=I_{\ell}^{\ell-1}A_{\ell} I_{\ell-1}^{\ell}$ and splitting matrices $M_\ell$ satisfies the approximation property if
	\begin{equation}
		\|M_{\ell}^{1/2}(A_{\ell}^{-1}-I_{\ell-1}^{\ell}A_{\ell-1}^{-1}I_{\ell}^{\ell-1})M_{\ell}^{1/2}\|_{2}\leq C_{A},\label{eq:V-important-cond}
	\end{equation}
	is satisfied for all \rev{$\ell\in\mathbb{N}$} with some constant $C_{A}>0$ independent of the level $\ell$.
\end{definition}
If the two properties are satisfied, the multigrid convergence rate can then be bounded as follows.
\begin{prop}\label{prop:V-cycle-bound}
	Assume that the same symmetric splitting matrix $M_\ell := M_\ell^{\mathrm{pre}} = M_\ell^{\mathrm{post}} = M_\ell^\top$ is used for pre- and post- smoothing on level $\ell>0$ (either in the standard multigrid iteration or in MGMC) and that this matrix satisfies the Smoothing Property (Definition \ref{def:smoothing_property}). Assume further that the Approximation Property (Definition \ref{def:approximation_property}) holds on all levels.  Finally, suppose that the sampler on level $\ell=0$ in Alg.~\ref{alg:mgmc} is an exact sampler (or, resp., the coarse grid solver in the multigrid iteration is an exact solver). Consider the $V$-cycle ($\gamma=1$)
	with $\nu_{1}=\nu_{2}=\nu/2$ for some even integer $\nu>0$, and
	let
	\[
		X_{L}=X_{L}(\gamma,\nu_{1},\nu_{2})=X_{L}(1,\nu/2,\nu/2)
	\]
	be the corresponding iteration matrix in \eqref{eq:def-Xell-Qell}.
	Then
	\begin{equation}
		\|X_{L}\|_{A_{L}}:=\|A_{L}^{1/2}X_{L}A_{L}^{-1/2}\|_{2}\leq\frac{C_{A}}{C_{A}+\nu},
		\label{eqn:energy_norm}
	\end{equation}
	where the constant $C_{A}>0$ is independent of $L$ and the $V$-cycle converges monotonically with respect to the 'energy' norm $\|\cdot\|_{A_{L}}$.
\end{prop}
\noindent\textit{Proof.} See \cite[Theorem 11.59]{Hackbusch.W_2016_IterativeSolutionLarge}.\\[1ex]
\noindent
We conclude that under the conditions in Proposition \ref{prop:V-cycle-bound} the multigrid iteration in \eqref{eqn:multigrid_error_iteration} converges. Crucially, as we will argue below, this also implies that states produced by the MGMC update in Alg.~\ref{alg:mgmc} converge to the target distribution.

Since $A_{\ell}$ is symmetric positive definite, the condition $0<A_{\ell}\leq M_{\ell}$ is always met for the symmetric Gauss--Seidel iteration; see \cite[Section 6.2.4.3]{Hackbusch.W_1985_MultiGridMethodsApplications}. We will see later that  this condition also holds when MGMC is applied to linear Bayesian inverse problems described in Section~\ref{subsec:Linear-Bayesian-inverse}, provided we use the low-rank smoothers introduced in Section~\ref{subsec:Gibbs}.

\jg{The usual (deterministic) multigrid theory directly implies the convergence of the mean; compare \eqref{eqn:multigrid_error_iteration} and \eqref{eqn:error_iteration} below.}  
To discuss convergence of higher moments in the sampling context, we need extra work. Since we consider Gaussian random variables, it suffices to consider the convergence of the covariance. We start by writing down an explicit expression for the MGMC update in Alg. \ref{alg:mgmc}.
\begin{lem}
	\label{lem:MG-iteration-rep}
	For $A_{L}\in\mathbb{R}^{n_{L}\times n_{L}}$, suppose the symmetric matrices $A_{\ell}\in\mathbb{R}^{n_{\ell}\times n_{\ell}}$,
	$\ell=0,\dots,L,$ recursively defined via $A_{\ell-1}=I_{\ell}^{\ell-1}A_{\ell}I_{\ell-1}^{\ell}$ are all invertible. Let $f_{L}\in\mathbb{R}^{n_{L}}$ be given and define $f_{\ell}$ recursively for $\ell=0,\dots,L-1$ by $f_{\ell-1}:=I_{\ell}^{\ell-1}(f_{\ell}-A_{\ell}\theta_{\ell,\nu_{1}})$ as in Alg.~\ref{alg:mgmc}.

	Then, for any level $\ell=0,\dots,L$, the MGMC update which computes a new state $\theta_{\ell}^{\mathrm{new}}$ from the current state $\theta_{\ell}^{\mathrm{init}}$ can be written as
	\begin{equation}
		\theta_{\ell}^{\mathrm{new}}=\mathrm{MGMC}_{\ell}(A_{\ell},f_{\ell},\theta_{\ell}^{\mathrm{init}}) = X_{\ell}\theta_{\ell}^{\mathrm{init}}+Y_{\ell}f_{\ell}+W_\ell,\label{eq:MG-iteration}
	\end{equation}
	with
	\begin{equation}
		Y_{\ell}:=(\mathrm{id}-X_{\ell})A_{\ell}^{-1}\label{eq:def-Y-ell},
	\end{equation}
and $X_{\ell}$ as given in \eqref{eq:def-Xell-Qell}. $W_\ell$ in \eqref{eq:MG-iteration} is a multivariate normal random variable with mean zero and the covariance matrix $\WCov_{\ell}=\mathbb{E}[W_{\ell}W_{\ell}^\top]$ of $W_\ell$ satisfies
	\begin{equation}
		A_{\ell}^{-1}-\WCov_{\ell}=X_{\ell}A_{\ell}^{-1}X_{\ell}^{\top},\quad\text{for all $\ell=0,\dots,L$}.\label{eq:noise-cov-id}
	\end{equation}
\end{lem}
\noindent\textit{Proof.} See Appendix~\ref{sec:proof_MGMC_representation}.\\[1ex]
Eqn.~\eqref{eq:MG-iteration} states that $\theta_{\ell}^{\mathrm{new}}$ can be represented as a sum of (i) the iteration matrix $X_\ell$ applied to the current sample $\theta_{\ell}^{\mathrm{init}}$, (ii) a matrix applied to the RHS $f_{\ell}$, and (iii) a multivariate normal random variable $W_\ell$ of given mean and covariance. Furthermore, as the proof of Lemma \ref{lem:MG-iteration-rep} in appendix \ref{sec:proof_MGMC_representation} shows, $W_\ell$ is a linear combination of simpler i.i.d. samples which are readily generated.

On the finest level the algorithm generates states $\theta_L^{(0)},\theta_L^{(1)},\ldots$ of a Markov chain and we write
\begin{equation}
	\theta_{L}^{(m+1)}=\mathrm{MGMC}_{L}(A_{L},f_{L},\theta_{L}^{(m)}) = X_{L}\theta_{L}^{(m)}+Y_{L}f_{L}+W^{(m)}_{L}.\label{eq:MG-iteration_finest_level}
\end{equation}
where the random variables $W_L^{(m)}$ for each step $m$ are i.i.d. and distributed as in \eqref{eq:MG-iteration} for $\ell=L$.
\begin{rem}\label{rem:exact_coarse_sampler}
	The proof for Lemma~\ref{lem:MG-iteration-rep} includes, up to the law, the cases where the sampler on the coarsest level is exact.
	An exact sampler on the coarsest level would correspond to $M_0^{\mathrm{coarse}} = A_0$,
	$N_0^{\mathrm{coarse}} = 0$, and
	$\nu_0=1$. In this case $S_0^{\mathrm{coarse}}=X_0=0$.
\end{rem}
We next study the rate of convergence of the mean vectors and covariance matrices under the MGMC update in Alg. \ref{alg:mgmc}. As usual, these convergence rates also determine the integrated autocorrelation time and the integration error, see Corollary~\ref{cor:conv-dist} and Theorem~\ref{thm:integration-error}, respectively. It is important to note that $\theta^{(m)}_L$ converges in distribution, if and only if the mean and covariance converge, as detailed in the proof of Theorem~\ref{thm:equivalence} below.

Consider the covariance matrix of $\theta_{L}^{(m)}$
\begin{equation}
	\mathrm{Cov}(\theta_{L}^{(m)}) := \mathbb{E}[(\theta_{L}^{(m)}-\mathbb{E}[\theta_{L}^{(m)}])(\theta_{L}^{(m)}-\mathbb{E}[\theta_{L}^{(m)}])^{\top}]\label{eqn:covariance_iteration}
\end{equation}
and the cross covariance matrix of $\theta_{L}^{(m)}$ and $\theta_{L}^{(m+s)}$ for $s=0,1,2,\dots$
\begin{equation}
	\mathrm{Cov}(\theta_{L}^{(m+s)},\theta_{L}^{(m)}):=\mathbb{E}[(\theta_{L}^{(m+s)}-\mathbb{E}[\theta_{L}^{(m+s)}])(\theta_{L}^{(m)}-\mathbb{E}[\theta_{L}^{(m)}])^{\top}].
	\label{eqn:cross_covariance_iteration}
\end{equation}
For the samples $(\theta_{L}^{(m)})_{m\in\mathbb{N}}$ generated by MGMC according to \eqref{eq:MG-iteration_finest_level} the following holds.
\begin{lem}
	\label{lem:MGidentity-mean-cov}Let $L\in\mathbb{N}$ and $f_{L}\in\mathbb{R}^{n_{L}}$ be given. Then the mean $\mathbb{E}[\theta_{L}^{(m)}]$, the covariance matrix in \eqref{eqn:covariance_iteration} and the cross-covariance matrix in \eqref{eqn:cross_covariance_iteration} satisfy
	\begin{subequations}
		\begin{align}
			\mathbb{E}[\theta_{L}^{(m+1)}]-A_{L}^{-1}f_{L} & =X_{L}\bigl(\mathbb{E}[\theta_{L}^{(m)}]-A_{L}^{-1}f_{L}\bigr)
			\label{eqn:error_iteration}                                                                                                        \\
			\mathrm{Cov}(\theta_{L}^{(m+1)})-A_{L}^{-1}             & =X_{L}\bigl(\mathrm{Cov}(\theta_{L}^{(m)})-A_{L}^{-1}\bigr)X_{L}^{\top}.
			\label{eqn:convergence_iteration}                                                                                                  \\
			\mathrm{Cov}(\theta_{L}^{(m+s)},\theta_{L}^{(m)})       & =X_{L}^{s}\mathrm{Cov}(\theta_{L}^{(m)})
			\qquad\text{for $m,s=0,1,2,\dots$}.
			\label{eq:MG-corr}
		\end{align}
	\end{subequations}
	where $X_{L}$ is defined in \eqref{eq:def-Tell} - \eqref{eq:def-X0}. Furthermore, if $X_{L}\colon(\mathbb{R}^{n_{L}},\|\cdot\|)\to(\mathbb{R}^{n_{L}},\|\cdot\|)$ defines a contraction with respect to $\|\cdot\|$, then $\mathbb{E}[\theta_{L}^{(m)}]$ converges to the solution $u_{L}$ of $A_{L}u_{L}=f_{L}$ and $\mathrm{Cov}(\theta_{L}^{(m)})$ converges to $A_{L}^{-1}$.
\end{lem}
\noindent\textit{Proof.} See Appendix~\ref{sec:proof_recursion}.\\[1ex]
Various norm bounds for $X_{L}$ are available in the multigrid literature. In our Theorem~\ref{prop:coarse_level_invariance} we quote a result for the symmetric case from \cite[Theorem 11.59]{Hackbusch.W_2016_IterativeSolutionLarge}. See for example \cite{Hackbusch.W_1985_MultiGridMethodsApplications} for other varieties.

The following theorem formalises the equivalence of the convergence of multigrid and MGMC. It can be interpreted as a variant of \cite[Theorem 1]{Fox.C_Parker_2017_AcceleratedGibbsSampling} where the latter holds for general matrix splittings.
\begin{thm}
	\label{thm:equivalence}Let $f_{L}\in\mathbb{R}^{n_{L}}$
	be a given vector and let $u_{L}^{(m)}$ be defined by the multigrid iteration for solving $A_{L}u_{L}=f_{L}$. Let $\theta_{L}^{(m)}$ be defined by the multigrid Monte Carlo iteration $\mathrm{MGMC}_{L}(A_{L},f_{L},\cdot)$ in \eqref{eq:MG-iteration_finest_level}.
	The following statements are equivalent.
	\begin{itemize}
		\item[(i)] For any initial condition $u_{L}^{(0)}$ the sequence $u_{L}^{(0)},u_{L}^{(1)},\dots,u_{L}^{(m)}$ converges to $u_{L}\in\mathbb{R}^{n_{L}}$ with $A_L u_{L}=f_L$.
		\item[(ii)] For any initial
			state $\theta_{L}^{(0)}$ the sequence $\theta_{L}^{(0)},\theta_{L}^{(1)},\dots,\theta_{L}^{(m)}$
			converges to the random variable $\theta\sim\mathcal{N}(A_L^{-1}f_L,A_L^{-1})$ in distribution.
	\end{itemize}
\end{thm}
\noindent\textit{Proof.} See Appendix~\ref{sec:proof_equivalence}.\\[1ex]
We reiterate that Theorem~\ref{thm:equivalence} states in particular that any guarantees on the  convergence of the multigrid iteration provides a guarantee for the convergence of MGMC in distribution and vice versa. The bound $\|X_{L}\|_{A_{L}}<1$ in Proposition \ref{prop:V-cycle-bound}, together with Lemma~\ref{lem:MGidentity-mean-cov}, gives the following result.
\begin{thm}
	\label{thm:moments-conv}Let the assumptions of Proposition\ \ref{prop:V-cycle-bound} be satisfied and suppose that the sampler on the coarsest level $\ell=0$ is an exact sampler. For MGMC with $\gamma=1$, i.e. the $V$-cycle, we have the following convergence of the mean, covariance and
	auto-covariance:
	\begin{align*}
		\|\mathbb{E}[\theta_{L}^{(m+1)}]-A_{L}^{-1}f_{L}\|_{A_{L}} & \leq\frac{C_{A}}{C_{A}+\nu}\|\mathbb{E}[\theta_{L}^{(m)}]-A_{L}^{-1}f_{L}\|_{A_{L}},             \\
		\|A_L^{1/2}\mathrm{Cov}(\theta_L^{(m)})A_L^{1/2}-\id\|_2 &\leq \left(\frac{C_{A}}{C_{A}+\nu}\right)^2 \|A_L^{1/2}\mathrm{Cov}(\theta_L^{(0)})A_L^{1/2}-\id\|_2,\\
		\|\mathrm{Cov}(\theta_{L}^{(m+s)},\theta_{L}^{(m)})\|_{A_{L}}       & \leq\biggl(\frac{C_{A}}{C_{A}+\nu}\biggr)^{s}\|\mathrm{Cov}(\theta_{L}^{(m)},\theta_{L}^{(m)})\|_{A_{L}}.
	\end{align*}
	Moreover, the covariance matrix satisfies the following inequality:
	\[
		\lim_{m\to\infty}\left(\frac{\|\mathrm{Cov}(\theta_{L}^{(m)})-A_{L}^{-1}\|_{A_{L}}}{\|\mathrm{Cov}(\theta_{L}^{(0)})-A_{L}^{-1}\|_{A_{L}}}\right)^{1/m}\leq\biggl(\frac{C_{A}}{C_{A}+\nu}\biggr)^{2}.
	\]
	Here, the constant $C_{A}>0$ is again independent of $L$.
\end{thm}
As a result, we have the following convergence of distributions.
\begin{cor}
	\label{cor:conv-dist}Let the assumptions of Theorem~\ref{thm:moments-conv}
	be satisfied. Then, the Kullback--Leibler divergence of the distribution
	of $(\theta_{L}^{(m)})_{m\in\mathbb{N}}$ and the target distribution
	$\mathcal{N}(A_{L}^{-1}f_{L},A_{L}^{-1})$ converges to~$0$.
	In particular, the sequence $(\theta_{L}^{(m)})_{m\in\mathbb{N}}$
	converges  in distribution to a Gaussian random variable with mean $A_{L}^{-1}f_{L}$
	and covariance $A_{L}^{-1}$.
\end{cor}
\noindent\textit{Proof.} See Appendix~\ref{sec:proof_distribution_convergence}.\\[1ex]
The results above allow us to analyse properties of the Monte Carlo estimator based on the sample generated by MGMC: As for any Markov chain, the states $\theta_L^{(0)}, \theta_L^{(1)}, \theta_L^{(2)},\dots,\theta_L^{(m)}$ generated by the MGMC update in Alg.~\ref{alg:mgmc} are not independent. To quantify this dependence we measure the autocorrelations of the chain for a particular quantity of interest $\mathcal{F}$ which maps each state to a real number. We limit our analysis to quantities of interest that depend linearly on the sample state, \rev{which includes many practically relevant quantities of interest, such as fluxes or integrals over parts of the domain}.
To this end, we consider the linear functional $\mathcal{F}$ defined in \eqref{eqn:qoi_linear}; discretisation of $\mathcal{F}$ leads to the matrix representation $F_L$ given in \eqref{eqn:F_matrix}.

We \rev{first} analyse the integrated autocorrelation  time (IACT) of the observed quantity $F_L^\top\theta_{L}^{(m)}$ which is defined as
\begin{equation}
	\tau_{\mathrm{int},F_L}^{(m)} :=1+2\sum_{s=1}^{\infty}\frac{\text{Cov}(F_L^\top\theta_{L}^{(m+s)},F_L^\top\theta_{L}^{(m)})}{\text{Cov}(F_L^\top\theta_{L}^{(m)},F_L^\top\theta_{L}^{(m)})}.
	\label{eqn:tau_int_definition}
\end{equation}
In the statistics literature the number of generated states divided by the IACT is also known as ``effective sample size'' (ESS) which can be interpreted as a measure for the number of statistically independent realisations of the quantity of interest in the chain. 
\rev{
The IACT in \eqref{eqn:tau_int_definition} converges to 
	\begin{equation}
	\rev{\tau^{(\infty)}_{\mathrm{int},F_L}}=1+2\sum_{s=1}^{\infty}\frac{F_L^\top X_{L}^{s}A_{L}^{-1}F_L}{F_L^\top A_{L}^{-1}F_L}.
	\label{eqn:tau_int}
\end{equation}
at a rate determined by $\|X_{L}\|_{A_{L}}$, as shown in the following theorem. 
Furthermore, if the initial sample is already drawn from the target normal distribution, the IACT simplifies directly to  \eqref{eqn:tau_int}.}
\begin{thm}
	\label{thm:IACT}
	\rev{Assume} 
	that the initial sample $\theta_L^{(0)}$ is drawn from a distribution with covariance $\mathrm{Cov}(\theta_L^{(0)})$ that is bounded in the sense that $||A_L^{1/2}\mathrm{Cov}(\theta_L^{(0)})A_L^{1/2}||_2<C_0$ for some constant $C_0$ that is independent of $L$. Then under the same assumption as in Theorem~\ref{thm:moments-conv}
	\begin{equation}
		\rev{\tau^{(\infty)}_{\mathrm{int},F_L}}\leq\frac{1+||X_{L}||_{A_{L}}}{1-||X_{L}||_{A_{L}}}\qquad\text{and} \qquad\left|\tau_{\mathrm{int},F_L}^{(m)}-\rev{\tau^{(\infty)}_{\mathrm{int},F_L}}\right|\le C\frac{||X_L||^{2m+1}_{A_L}}{(1-||X_L||_{A_L})(1-||X_L||_{A_L}^{2m})}
		\label{eqn:tau_int_bounds}
	\end{equation}
	for all $m\ge 1$ and some constant $C$ which is independent of $L$.
        
	\rev{Instead, if the initial sample $\theta_L^{(0)}$ is drawn from the target multivariate normal distribution, then the corresponding $\tau_{\mathrm{int},F_L}^{(m)}$, in \eqref{eqn:tau_int_definition}, can be expressed as \eqref{eqn:tau_int}, regardless of the value of $m$.}  
\end{thm}
\noindent\textit{Proof.} See Appendix~\ref{sec:proof_IACT_bounds}.\pagebreak

\noindent
Note that $||X_{L}||_{A_{L}}$ does not have to be exceptionally small to obtain IACTs of $\sim1-10$: For $||X_{L}||_{A_{L}}=0.5$ we get
$\rev{\tau^{(\infty)}_{\mathrm{int},F_L}}=2$ and for $||X_{L}||_{A_{L}}=0.8$ we have $\rev{\tau^{(\infty)}_{\mathrm{int},F_L}}=9$. Furthermore, the condition $||A_L^{1/2}\mathrm{Cov}(\theta_L^{(0)})A_L^{1/2}||_2<C_0$ is trivially satisfied if the initial sample is drawn from an infinitely narrow delta-distribution, i.e. if we set $\theta_L^{(0)}=\theta_{L,0}$ for some fixed vector $\theta_{L,0}$.

In practice, the sample $\theta_L^{(1)}, \theta_L^{(2)},\dots,\theta_L^{(M)}$ is used to construct a Monte Carlo estimator
\begin{equation}
	\tilde{I}_{M}(F_L)(\omega):=\frac{1}{M}\sum_{m=1}^{M}F_L^\top \theta_{L}^{(m)}(\omega)\label{eqn:MGMC_estimator}
\end{equation}
for the quantity of interest
$I(F_L):=\mathbb{E}[F_L^\top \theta_{L}]$ with $\theta_{L}\sim\mathcal{N}(A_{L}^{-1}f_{L},A_{L}^{-1})$.
The root-mean-squared-error for the sum in \eqref{eqn:MGMC_estimator} can be  bounded by the standard Monte Carlo rate $\mathcal{O}(M^{-1/2})$ and the implied constant can be made \emph{independent of $L$}. More precisely, we have the following results:
\begin{thm}
	\label{thm:integration-error} Let Assumption~\ref{assu:Pell} and
	assumptions of Theorems~\ref{thm:moments-conv} and \ref{thm:IACT}
	hold. Assume further that the initial sample $\theta_{L}^{(0)}$ is
	drawn from a distribution with mean $\mathbb{E}[\theta_{L}^{(0)}]$
	that is bounded in the sense that $\|\mathbb{E}[A_{L}^{1/2}\theta_{L}^{(0)}]\|_{2}<C_{0}$
	for some constant $C_{0}$ that is independent of $L$. Suppose that
	$f_{L}$ is given by $(f_L)_j = \langle f,\phi_{j}^{L}\rangle_{H}$ for ${j=1,\dots,n_{L}}$ for some $f\in H$. Then, the root-mean-squared-error of $\tilde{I}_{M}(F_L)$ in \eqref{eqn:MGMC_estimator} can be bounded as
	\begin{equation}
		\sqrt{\mathbb{E}[|I(F_{L})-\tilde{I}_{M}(F_{L})|^{2}]}\leq\frac{C}{\sqrt{M}}.\label{eqn:integration_error}
	\end{equation}
	The constant $C>0$ depends on $F_L$ (and thus on $\mathcal{F}$ through \eqref{eqn:F_matrix}), but is independent
	of $M$ and $L$.
\end{thm}
\noindent\textit{Proof.} See Appendix~\ref{sec:proof_RMS_error}.
\paragraph{Relationship to the literature.}
The original work in \cite{goodman1989multigrid},  which introduces the Multigrid Monte Carlo method, \jg{shows} the invariance of the distribution under MGMC updates and this is formalised in our Theorem \ref{thm:mgmc_invariance}. The results in our Theorem~\ref{thm:equivalence}, Lemma~\ref{lem:MGidentity-mean-cov}, Corollary~\ref{cor:conv-dist} and Theorem~\ref{thm:IACT} are also derived in~\cite{goodman1989multigrid}. \jg{Since the original work in \cite{goodman1989multigrid} is aimed at the theoretical physics community, we rewrite the proofs here in a language more familiar to mathematicians and statisticians.}

To make the relationship to the work in \cite{Fox.C_Parker_2017_AcceleratedGibbsSampling} explicit, we now show that the MGMC update can be written as a matrix splitting method of the form in \eqref{eqn:theta_random} if the multigrid method is convergent. For this first recall that according to Lemma \ref{lem:MG-iteration-rep} the MGMC update can be written as $\theta_{L}^{(m+1)}= X_{L}\theta_{L}^{(m)}+Y_{L}f_{L}+W^{(m)}_{L}$ (c.f. \eqref{eq:MG-iteration_finest_level}) with $Y_L = (\id-X_L)A_L^{-1}$ and $W_L^{(m)}\sim\mathcal{N}(0,A_L^{-1}-X_LA_L^{-1}X_L^\top)$. Eliminating $X_L = \id - Y_LA_L$, \eqref{eq:MG-iteration_finest_level} can be written as
\begin{equation}
	\theta_L^{(m+1)} = \theta_L^{(m)} + Y_L (f_L-A_L\theta_L^{(m)}) + W_L^{(m)}\qquad\text{with $W_L^{(m)}\sim\mathcal{N}(0,Y_L+Y_L^\top-Y_LA_LY_L^\top)$}.\label{eqn:theta_update_with_YL}
\end{equation}
Under the assumptions of Proposition \ref{prop:V-cycle-bound} we have  $||X_{L}||_{A_{L}}<1$ and therefore $Y_{L}$ is invertible as Lemma \ref{lem:Y-invertible} shows.

Hence, we can set $M_L^{\mathrm{MGMC}}:=Y_L^{-1}$, $N_L^{\mathrm{MGMC}}:=M_L^{\mathrm{MGMC}}-A_L$ and straightforward algebraic manipulations show that \eqref{eqn:theta_update_with_YL} can be written as
\begin{equation}
	\theta_L^{(m+1)} = (M^{\mathrm{MGMC}}_L)^{-1}\theta_L^{(m)} + (M^{\mathrm{MGMC}}_L)^{-1}(f_L+\xi_L^{(m)})\quad\text{with $\xi_L^{(m)}\sim\mathcal{N}(0,(M^{\mathrm{MGMG}})^\top+N^{\mathrm{MGMC}})$},
\end{equation}
which is exactly the form in \eqref{eqn:theta_random}. With this, we can interpret our Theorem \ref{thm:equivalence} as a special case of \cite[Theorem 1]{Fox.C_Parker_2017_AcceleratedGibbsSampling} for the splitting matrix $M_L^{\mathrm{MGMC}}$ defined by Alg.~\ref{alg:mgmc}. Furthermore, our Lemma \ref{lem:MGidentity-mean-cov} can be seen as a special case of \cite[Corollary 3]{Fox.C_Parker_2017_AcceleratedGibbsSampling}.

\subsection{Application to \rev{linear Gaussian Bayesian inverse problems}\label{subsec:Bayesian}}
We apply the abstract MGMC convergence results from Sections~\ref{sec:theory_invariance} and \ref{sec:theory_convergence} to the linear Bayesian inverse problem introduced in Section~\ref{subsec:Linear-Bayesian-inverse}.
To do this, we show that the assumptions of Proposition~\ref{prop:V-cycle-bound} also hold for the inverse problem, when
the sampler is the symmetric Gauss--Seidel sampler.

Recall that for the linear problem, the posterior distribution is
given by $\mathcal{N}(\mu_{L},\widetilde{A}_{L}^{-1})$, where $\mu_{L}$
is as in (\ref{eq:post-mean}) and $\widetilde{A}_{L}^{-1}$ is as
in (\ref{eq:post-cov}). Hence, to integrate with respect to the posterior,
we call $\mathrm{MGMC}(\widetilde{A}_{L},f_{L})$ with $f_{L}=f_{L}(y_{L})
	=B_L \Gamma^{-1}y_L$ as in \eqref{eqn:f_ell_definition}.
The symmetric Gauss--Seidel smoother for the precision matrix corresponding
to the posterior distribution is given by \eqref{eqn:SGS_splitting}:
\[
	\widetilde{M}_{\ell}^{\text{(SGS)}}:=(D_{\ell}+B_{\ell}\Gamma^{-1}B_{\ell}^{\top}+L_{\ell})(D_{\ell}+B_{\ell}\Gamma^{-1}B_{\ell}^{\top})^{-1}(D_{\ell}+B_{\ell}\Gamma^{-1}B_{\ell}^{\top}+L_{\ell}^{\top}),
\]
where we note that the matrices $D_{\ell}+B_{\ell}\Gamma^{-1}B_{\ell}^{\top}$
and $D_{\ell}+B_{\ell}\Gamma^{-1}B_{\ell}^{\top}+L_{\ell}$ can be
inverted by applying the Woodbury matrix identity since $D_{\ell}$
and $D_{\ell}+L_{\ell}$ are assumed to be invertible. According to \eqref{eqn:SGS_splitting},
$\widetilde{M}_{\ell}^{\text{(SGS)}}$ can be rewritten as
\begin{align}
	\widetilde{M}_{\ell}^{\text{(SGS)}} & =A_{\ell}+B_{\ell}\Gamma^{-1}B_{\ell}^{\top}+L_{\ell}\left(D_{\ell}+B_{\ell}\Gamma^{-1}B_{\ell}^{\top}\right)^{-1}L_{\ell}^{\top}.\label{eq:tildeM-SGS-formula}
\end{align}
We verify two conditions for all $\text{\ensuremath{\ell\geq1}}$ to apply the V-cycle convergence result (Proposition \ref{prop:V-cycle-bound}):
\[
  0<\widetilde{A}_{\ell}\leq\widetilde{M}_{\ell}^{\text{(SGS)}}
  \quad \text{and} %
  \quad \|(\widetilde{M}_{\ell}^{\text{(SGS)}})^{1/2}(\widetilde{A}_{\ell}^{-1}-I_{\ell-1}^{\ell}\widetilde{A}_{\ell-1}^{-1}I_{\ell}^{\ell-1}))(\widetilde{M}_{\ell}^{\text{(SGS)}})^{1/2}\|_{2}\leq C_{A}\,
  .
  \]
The first condition is easy to verify. Indeed, for $x\in\mathbb{R}^{n_{\ell}}$
we have
\begin{align*}
	0 \leq x^\top A_{\ell}x+x^\top B_{\ell}\Gamma^{-1}B_{\ell}^{\top}x = x^{\top}\widetilde{A}_{\ell}x \leq x^{\top}\widetilde{A}_{\ell}x+x^{\top}L_{\ell}\left(D_{\ell}+B_{\ell}\Gamma^{-1}B_{\ell}^{\top}\right)^{-1}L_{\ell}^{\top}x  = x^{\top}\widetilde{M}_{\ell}^{\text{(SGS)}}x\,,
\end{align*}
where the first inequality follows since the matrices $A_{\ell}$ and $\Gamma$ are assumed to be positive definite.

Showing the second condition, the approximation property, is more involved.
\subsubsection{Approximation property for the perturbed problem}
To show the approximation property, we link discretisations across all levels through the infinite-dimensional problem.
Let
\begin{equation}
	b(\zeta,\varphi):=\langle\mathcal{B}^{*}\Gamma^{-1}\mathcal{B}\zeta,\varphi\rangle_{H}= (\mathcal{B}\zeta)^\top\Gamma^{-1}(\mathcal{B}\varphi)\qquad\text{ for }\zeta,\varphi\in V.
	\label{eqn:b_definition}
\end{equation}
Let $f\in H$ be given. Consider the problems: Find $u\in V$ such
that
\begin{equation}
	a(u,\varphi)+b(u,\varphi)=\langle f,\varphi\rangle_{H}=\text{ for all }\varphi\in V;\label{eq:eq-perturbed-original}
\end{equation}
find $u_{\ell}\in V_{\ell}$ such that
\begin{equation}
	a(u_{\ell},\varphi_{\ell})+b(u_{\ell},\varphi_{\ell})=\langle f,\varphi_{\ell}\rangle_{H}\text{ for all }\varphi_{\ell}\in V_{\ell}.\label{eq:eq-perturbed-Vell}
\end{equation}
From $b(\zeta,\zeta)=\|\Gamma^{-1/2}\mathcal{B}\zeta\|_{2}^{2}\geq0$,
the coercivity of $a$ implies that of $a+b$. The bilinear form $a+b$
is also bounded on $V$ because
\begin{align*}
	|b(\zeta,\varphi)|\leq & \|\Gamma^{-1}\|_{2}\|\mathcal{B}\|_{H\to\mathbb{R}^{\beta}}^{2}\|\zeta\|_{H}\|\varphi\|_{H}\leq C\|\Gamma^{-1}\|_{2}\|\mathcal{B}\|_{H\to\mathbb{R}^{\beta}}^{2}\|\zeta\|_{V}\|\varphi\|_{V}
\end{align*}
holds for $\zeta,\varphi\in V$. Hence, by the Lax--Milgram theorem
the equation (\ref{eq:eq-perturbed-original}) admits a unique solution
$u\in V$ such that $\|u\|_{V}\leq c\|f\|_{H}$. Similarly, (\ref{eq:eq-perturbed-Vell})
admits a unique solution $u_{\ell}\in V_{\ell}$.

Assumption~\ref{assu:WtoH-best-Vell-approx} assumes a regularity of the solution of the unperturbed problem  \eqref{eq:eq-given-by-A}.
It turns out that the solution of the perturbed problem \eqref{eq:eq-perturbed-original} has the same regularity.
\begin{lem}\label{lem:approximation-perturbed}
	If $\mathcal{A}$ satisfies Assumption~\ref{assu:WtoH-best-Vell-approx},
	then so does the perturbed operator $\widetilde{\mathcal{A}}=\mathcal{A}+\mathcal{B}^{*}\Gamma^{-1}\mathcal{B}$ and
	\begin{equation}
		\|u\|_{W}  \leq C_{\widetilde{\mathcal{A}}}\,\|f\|_{H}\label{eq:perturbedu-Wnorm-bound}
	\end{equation}
	holds. Moreover, the solution $u\in V$ of (\ref{eq:eq-perturbed-original})
	can be approximated by the solution $u_{\ell}\in V_{\ell}$ of (\ref{eq:eq-perturbed-Vell})
	with an error
	\begin{equation}
		\|u-u_{\ell}\|_{H}\leq C\Psi(\ell)\|u\|_{W},\label{eq:perturbed-H-W-bound}
              \end{equation}
          where $\Psi(\ell)$ is defined in Assumption~\ref{assu:WtoH-best-Vell-approx}.
\end{lem}

\begin{proof}
	The solution $u\in V$ of \eqref{eq:eq-perturbed-original} satisfies
	$a(u,\varphi)=\langle\tilde{f},\varphi\rangle_{H}$ for any $\varphi\in V$,
	where we let $\tilde{f}:=f-\mathcal{B}^{*}\Gamma^{-1}\mathcal{B}u$.
	Hence, Assumption~\ref{assu:WtoH-best-Vell-approx} for $\mathcal{A}$ together
	with $\|u\|_{H}\leq C\|u\|_{V}\leq C'\|f\|_{H}$ implies
	\begin{align*}
		\|u\|_{W} & \leq C_{\mathcal{A}}\|\tilde{f}\|_{H}                                                                                                              \leq C_{\mathcal{A}}(\|f\|_{H}+\|\mathcal{B}^{*}\Gamma^{-1}\mathcal{B}u\|_{H})                                                                    \leq C_{\mathcal{A}}(\|f\|_{H}+\|\mathcal{B}\|_{H\to\mathbb{R}^{\beta}}^{2}\|\Gamma^{-1}\|_{2}\|u\|_{H})\leq C_{\widetilde{\mathcal{A}}}\|f\|_{H}.
	\end{align*}
	The bound (\ref{eq:perturbed-H-W-bound}) is obtained by adapting Proposition~\ref{prop:Aubin-Nitsche}
	to $\widetilde{\mathcal{A}}$,  again under
	Assumption~\ref{assu:WtoH-best-Vell-approx}.
\end{proof}
To proceed further, we use a lower bound on $\|A_{\ell}\|_{2}^{-1}$ which follows under natural assumptions.

\begin{assumption}\label{ass:inverse_inequality}
	There exists a constant $C_{\Psi}>1$, such that the function
	$\Psi\colon\mathbb{N}_{0}\to[0,\infty)$ in Assumption~\ref{assu:WtoH-best-Vell-approx} satisfies
	\begin{equation}
		\Psi(\ell-1)\leq C_{\Psi}\Psi(\ell) \qquad\text{for all $\ell\geq1$}
		\label{eqn:inverse_inequality1}
	\end{equation}
	and the following inverse inequality holds:
        \begin{equation}\label{eqn:inverse_inequality2}
          \|\varphi_{\ell}\|_{V}\leq\|\varphi_{\ell}\|_{H}/\sqrt{\Psi(\ell)}, \quad \text{for all $\varphi_{\ell}\in V_{\ell}\,$}.
          \end{equation}
\end{assumption}

\begin{prop}\label{prop:assump}
	Let Assumptions~\ref{assu:Pell} and \ref{ass:inverse_inequality} hold. Then, there exists a constant $C>0$ such that
	\begin{equation}
		\frac{\Psi(\ell-1)}{(\Phi(\ell))^{2}}\leq\frac{C}{\|A_{\ell}\|_{2}}
		\qquad\text{for all $1\leq \ell\leq L$.}\label{eqn:lower_A_bound}
	\end{equation}
\end{prop}
\begin{proof}
  Let $x,y\in\mathbb{R}^{n_{\ell}}$. Then it follows from  Assumptions~\ref{assu:Pell} and \ref{ass:inverse_inequality} that
\begin{align*}
	x^{\top}A_{\ell}y
	 & =
	(I_{L-1}^L I_{L-2}^{L-1}\dotsb I_{\ell}^{\ell+1} x)^{\top}A_{L}(I_{L-1}^L I_{L-2}^{L-1}\dotsb I_{\ell}^{\ell+1}y)                 \\
	 &
	=
	\langle\mathcal{A}^{1/2}P_{\ell}x,\mathcal{A}^{1/2}P_{\ell}y\rangle_{H}  \leq\|P_{\ell}x\|_{V}\|P_{\ell}y\|_{V}  \leq \frac{\|P_{\ell}x\|_{H}\|P_{\ell}y\|_{H}}{\Psi(\ell)}
	\leq\frac{1}{c_{1}^{2}}\|x\|_{2}\|y\|_{2}\frac{\Phi^{2}(\ell)}{\Psi(\ell)}.
\end{align*}
Choosing $x=A_{\ell}y$ yields $\|A_{\ell}y\|_{2}^{2}\leq\frac{1}{c_{1}^{2}}\|A_{\ell}y\|_{2}\|y\|_{2}\frac{\Phi^{2}(\ell)}{\Psi(\ell)}$
and thus together with \eqref{eqn:inverse_inequality1}
\[
  \|A_{\ell}\|_{2} \le c_1^{-2} \frac{\Phi^{2}(\ell)}{\Psi(\ell)} \le  c_1^{-2}C_{\Psi}\frac{\Phi^{2}(\ell)}{\Psi(\ell-1)},
  \]
which implies  \eqref{eqn:lower_A_bound}. %
\end{proof}
For example, piecewise polynomial
FE spaces $V_{\ell}$ on quasi-uniform triangulations, with $V=H^{1}(D)$ and $H=L^{2}(D)$,
satisfy the inverse estimate with $\sqrt{\Psi(\ell)}\asymp h_{\ell}$;
see e.g., \cite[Prop.~6.3.2]{Quarteroni.A_Vali_book_1994_NAPDE}.

The lower bound for $A_{\ell}$ in \eqref{eqn:lower_A_bound} implies an analogous lower bound for the perturbed precision matrix $\widetilde{A}_{\ell}=A_{\ell}+B_{\ell}\Gamma^{-1}B_{\ell}^{\top}$.
\begin{prop}
	\label{prop:1/tildeAll-lb} Let Assumptions~\ref{assu:Pell} and \ref{ass:inverse_inequality} be satisfied. Then, there exists a constant $\tilde{C}>0$ such that 
	\begin{equation}
          \frac{\Psi(\ell-1)}{(\Phi(\ell))^{2}}\leq\frac{\tilde{C}}{\|\widetilde{A}_{\ell}\|_{2}}
          \qquad\text{for all $1\leq \ell\leq L$.}\label{eq:inv-norm-bound}
	\end{equation}
\end{prop}
\begin{proof}
	Assumption~\ref{assu:Pell} implies
	\begin{align*}
		\|B_{\ell}\Gamma^{-1}B_{\ell}^{\top}\|_{2} & \leq\|\Gamma^{-1}\|_{2}\|\mathcal{B}\|_{H\to\mathbb{R}^{\beta}}^{2}\|P_{\ell}\|_{\mathbb{R}^{n_{\ell}}\to H}^{2}\leq\|\Gamma^{-1}\|_{2}\|\mathcal{B}\|_{H\to\mathbb{R}^{\beta}}^{2}(\Phi(\ell))^{2},
	\end{align*}
	and thus $\|A_{\ell}+B_{\ell}\Gamma^{-1}B_{\ell}^{\top}\|_{2}\leq C_{\Gamma^{-1},\mathcal{B}}(\|A_{\ell}\|_{2}+(\Phi(\ell))^{2})$ for some positive constant $C_{\Gamma^{-1},\mathcal{B}}$.
	Then, from Proposition~\ref{prop:assump} we conclude
	\begin{align*}
		\frac{\Psi(\ell-1)}{(\Phi(\ell))^{2}}\|A_{\ell}+B_{\ell}\Gamma^{-1}B_{\ell}^{\top}\|_{2} & \leq C_{\Gamma^{-1},\mathcal{B}}\Bigg(\frac{\Psi(\ell-1)}{(\Phi(\ell))^{2}}\|A_{\ell}\|_{2}+\Psi(\ell-1)\Bigg) \leq C_{\Gamma^{-1},\mathcal{B}}(C+\Psi(0))=:\tilde{C}.
	\end{align*}
\end{proof}
\noindent
Now, we are ready to show the approximation property of the perturbed
precision matrices. This is done in two steps.
\begin{prop}
	\label{prop:approximation-property}Let Assumptions~\ref{assu:Pell}, \ref{assu:WtoH-best-Vell-approx} and \ref{ass:inverse_inequality} hold. Then, we have the approximation property for the perturbed matrix $\widetilde{A}_{\ell}$
	\begin{align*}
		\|\widetilde{A}_{\ell}^{-1}-I_{\ell-1}^{\ell}\widetilde{A}_{\ell-1}^{-1}I_{\ell}^{\ell-1}\|_{2}
		 & \leq\frac{C}{\|\widetilde{A}_{\ell}\|_{2}}.
	\end{align*}
\end{prop}
\noindent\textit{Proof.} See Appendix~\ref{sec:proof_perturbed_matrix}.\\[1ex]
The result above establishes $\|(\widetilde{A}_{\ell}^{-1}-I_{\ell-1}^{\ell}\widetilde{A}_{\ell}^{-1}I_{\ell}^{\ell-1})\|_{2}\leq C\|\widetilde{A}_{\ell}\|_{2}^{-1}$.
To invoke the $V$-cycle convergence result (Proposition~\ref{prop:V-cycle-bound}),
we still need to verify the equivalent of (\ref{eq:V-important-cond}) with $\widetilde{\mathcal{A}}$ instead of $\mathcal{A}$. For this we
will show $\|\widetilde{M}_{\ell}^{\text{(SGS)}}\|_{2}\leq C\|\widetilde{A}_{\ell}\|_{2}$
for $\ell=0,\dots,L$, with which we have (\ref{eq:V-important-cond}):
\begin{align*}
	\|\sqrt{\widetilde{M}_{\ell}^{\text{(SGS)}}}(\widetilde{A}_{\ell}^{-1}- & I_{\ell-1}^{\ell}\widetilde{A}_{\ell}^{-1}I_{\ell}^{\ell-1})\sqrt{\widetilde{M}_{\ell}^{\text{(SGS)}}}\|_{2}                                                                                                                                                                                                                                    \\
	=                                                                       & \sqrt{\lambda_{\mathrm{max}}\left(\sqrt{\widetilde{M}_{\ell}^{\text{(SGS)}}}(\widetilde{A}_{\ell}^{-1}-I_{\ell-1}^{\ell}\widetilde{A}_{\ell}^{-1}I_{\ell}^{\ell-1})\widetilde{M}_{\ell}^{\text{(SGS)}}(\widetilde{A}_{\ell}^{-1}-I_{\ell-1}^{\ell}\widetilde{A}_{\ell}^{-1}I_{\ell}^{\ell-1})\sqrt{\widetilde{M}_{\ell}^{\text{(SGS)}}}\right)} \\
	=                                                                       & \sqrt{\lambda_{\mathrm{max}}\left(\widetilde{M}_{\ell}^{\text{(SGS)}}(\widetilde{A}_{\ell}^{-1}-I_{\ell-1}^{\ell}\widetilde{A}_{\ell}^{-1}I_{\ell}^{\ell-1})\widetilde{M}_{\ell}^{\text{(SGS)}}(\widetilde{A}_{\ell}^{-1}-I_{\ell-1}^{\ell}\widetilde{A}_{\ell}^{-1}I_{\ell}^{\ell-1})\right)}                                                  \\
	\leq                                                                    & \|\widetilde{M}_{\ell}^{\text{(SGS)}}(\widetilde{A}_{\ell}^{-1}-I_{\ell-1}^{\ell}\widetilde{A}_{\ell}^{-1}I_{\ell}^{\ell-1})\|_{2}\leq C.
\end{align*}
To show $\|\widetilde{M}_{\ell}^{\text{(SGS)}}\|_{2}\leq C\|\widetilde{A}_{\ell}\|_{2}$
we use the following,  the proof of which we defer to Appendix~\ref{appendix:general}.
\begin{lem}
	\label{lem:Aell-lb}Let Assumption~\ref{assu:Pell} hold. Then, $A_{\ell}$
	as in (\ref{eq:def-A_ell}) satisfies
	\[C\Phi^{2}(\ell)\leq\|A_{\ell}\|_{2}.\]
\end{lem}
\begin{proof}
	Let $x\in\mathbb{R}^{n_{\ell}}$. By the definition $\eqref{eq:def-A_ell}$
	of $A_{\ell}$ we have $x^{\top}A_{\ell}x=\langle\mathcal{A}^{1/2}P_{\ell}x,\mathcal{A}^{1/2}P_{\ell}x\rangle_{H}=\|P_{\ell}x\|_{V}^{2}$.
	Thus, Assumption~\ref{assu:Pell} implies
	\[
		C\|x\|_{2}^{2}\Phi^{2}(\ell)\leq x^{\top}A_{\ell}x,
	\]
	where we used $c\|P_{\ell}x\|_{H}\leq\|P_{\ell}x\|_{V}$.
	Hence, noting that $A_{\ell}$ is symmetric and taking the supremum over
	$\|x\|_{2}=1$ yields the result.
      \end{proof}
Finally, we obtain the following, which gives us (\ref{eq:V-important-cond})
and thus allows us to use Proposition~\ref{prop:V-cycle-bound} for
the perturbed precision matrix $\widetilde{A}_{\ell}$.
\begin{lem}
	Let Assumption~\ref{assu:Pell} hold. Suppose $\|M_{\ell}^{\text{(SGS)}}\|_{2}\leq C_{A}\|A_{\ell}\|_{2}$
	for all $\ell=0,\dots,L$, with $C_{A}>0$ independent of $\ell$.
	Then, the symmetric Gauss--Seidel iteration matrix $\widetilde{M}_{\ell}^{\text{(SGS)}}$
	as in (\ref{eq:tildeM-SGS-formula}) satisfies
	\[
		\|\widetilde{M}_{\ell}^{\text{(SGS)}}\|_{2}\leq C\|\widetilde{A}_{\ell}\|_{2}\qquad\text{for }\ell=0,\dots,L,
	\]
	where $C>0$ is independent of $\ell$.
\end{lem}

\begin{proof}
	For $x\in\mathbb{R}^{n_{\ell}}$, noting that $B_{\ell}\Gamma^{-1}B_{\ell}^{\top}$
	is symmetric positive definite, we have
	\begin{align*}
		x^{\top}\widetilde{M}_{\ell}^{\text{(SGS)}}x & =x^{\top}(A_{\ell}+B_{\ell}\Gamma^{-1}B_{\ell}^{\top})x+x^{\top}L_{\ell}\left(D_{\ell}+B_{\ell}\Gamma^{-1}B_{\ell}^{\top}\right)^{-1}L_{\ell}^{\top}x \\
		                                                               & \leq x^{\top}(A_{\ell}+B_{\ell}\Gamma^{-1}B_{\ell}^{\top})x+x^{\top}L_{\ell}D_{\ell}^{-1}L_{\ell}^{\top}x                                              \\
		                                                               & =x^{\top}B_{\ell}\Gamma^{-1}B_{\ell}^{\top}x+x^{\top}(A_{\ell}+L_{\ell}D_{\ell}^{-1}L_{\ell}^{\top})x                                                 \\
		                                                               & =x^{\top}B_{\ell}\Gamma^{-1}B_{\ell}^{\top}x+x^{\top}M_{\ell}^{\text{(SGS)}}x.
	\end{align*}
	Noting that Assumption~\ref{assu:Pell} implies $\|B_{\ell}\Gamma^{-1}B_{\ell}^{\top}\|_{2}\leq\|\Gamma^{-1}\|_{2}\|\mathcal{B}\|_{H\to\mathbb{R}^{\beta}}^{2}(\Phi(\ell))^{2}$,
	we use the assumption $\|M_{\ell}^{\text{(SGS)}}\|_{2}\leq C_{A}\|A_{\ell}\|_{2}$,
	and Lemma~\ref{lem:Aell-lb} to obtain
	\begin{align*}
		\|\widetilde{M}_{\ell}^{\text{(SGS)}}\|_{2} & \leq\|B_{\ell}\Gamma^{-1}B_{\ell}^{\top}\|_{2}+C_{A}\|A_{\ell}\|_{2}             \leq(\|B_{\ell}\Gamma^{-1}B_{\ell}^{\top}\|_{2}/\|A\|_{2}+C_{A})\|A_{\ell}\|_{2} \\
		                                            & \leq(C+C_{A})\|A_{\ell}\|_{2}\leq(C+C_{A})\|\widetilde{A}_{\ell}\|_{2}.
	\end{align*}
\end{proof}
In the lemma above we assumed $\|M_{\ell}^{\text{(SGS)}}\|_{2}\leq C_{A}\|A_{\ell}\|_{2}$
for the unpurterbed precision matrix; this is a standard assumption (see e.g. \cite[Theorem 11.30]{Hackbusch.W_2016_IterativeSolutionLarge}).

From the discussion in this section, we obtain the following $L$-independent bounds for the IACT and the root-mean-square error for the perturbed problem.
\begin{thm}\label{thm:IACT_posterior_robustness}
  Let Assumptions~\ref{assu:Pell}, \ref{assu:WtoH-best-Vell-approx} and \ref{ass:inverse_inequality} hold and suppose %
  $y\in\mathbb{R}^{\beta}$. Consider the MGMC update
  for $\mathcal{N}(\widetilde{\mu}_{L},\widetilde{A}_{L}^{-1})$
	with mean and covariance given by $\widetilde{\mu}_{L}=A_{L}^{-1}B_{L}(\Gamma+B_{L}^{\top}A_{L}^{-1}B_{L})^{-1}y$
	and $\widetilde{A}_{L}^{-1}=(A_{L}+B_{L}\Gamma^{-1}B_{L}^{\top})^{-1}$.
        Assume further that the initial sample $\theta_{L}^{(0)}$ is drawn from
	a distribution with moments $\mathbb{E}[\theta_{L}^{(0)}]$ and $\mathrm{Cov}(\theta_{L}^{(0)})$
	that are bounded such that
        \[  \max\{||\widetilde{A}_{L}^{1/2}\mathbb{E}[\theta_{L}^{(0)}]||_{2},||\widetilde{A}_{L}^{1/2}\mathrm{Cov}(\theta_{L}^{(0)})\widetilde{A}_{L}^{1/2}||_{2}\}<C_{0}
          \]
          for some constant $C_{0}$ that is independent of $L$.
          
	Suppose that the sampler on the coarsest level $\ell=0$ is an exact
	sampler. Assume that the same symmetric splitting $\widetilde{M}_{\ell}^{\text{(SGS)}}$
	as in \eqref{eqn:SGS_splitting} is used for pre- and post- smoothing
	on level $\ell>0$. 

	Then, bounds on the IACT, analogous to \eqref{eqn:tau_int_bounds} in Theorem~\ref{thm:IACT}, hold for this MGMC update with implied constant again independent of $L$. Moreover,
	the root-mean-square error bound analogous to Theorem~\ref{thm:integration-error}
	holds with a constant $C>0$ that depends on $\|y\|_{2}$
	but is independent of $L$.
\end{thm}
\begin{proof}From the discussions in this section, Proposition~\ref{prop:V-cycle-bound}
	holds for the MGMC for $\mathcal{N}(\widetilde{\mu}_{L},\widetilde{A}_{L}^{-1})$
	and thus $\|\tilde{X}_{L}\|_{\tilde{A}_{L}}\leq q<1$ holds with an
	$L$-independent constant $q$. To prove the statement we follow the
	argument in the proofs of Theorems~\ref{thm:IACT} and~\ref{thm:integration-error}.
	Since the moments of the initial condition is assumed to be bounded,
	it suffices to bound $\|F_{L}^\top\widetilde{A}_{L}^{-1/2}\|_{2}$ and
	$\|\widetilde{A}_{L}^{-1/2}f_{L}\|_{2}$.

	To bound $\|F_{L}^\top \widetilde{A}_{L}^{-1/2}\|_{2}^{2}=F_{L}^\top\widetilde{A}_{L}^{-1}F_{L}=\mathcal{F}(P_{L}\widetilde{A}_{L}^{-1}F_{L})$,
	we notice that $\psi:=P_{L}\widetilde{A}_{L}^{-1}F_{L}$ satisfies
	\[
		a(\psi,v_{L})+b(\psi,v_{L})=\mathcal{F}(v_{L})\quad\text{for all }v_{L}\in V_{L},
	\]
	and thus $\|\psi\|_{V}\leq\|\mathcal{F}\|_{H\to\mathbb{R}}/\sqrt{\lambda_{\mathrm{min}}(\mathcal{A})}$.
	Hence, $\|F_{L}^\top\widetilde{A}_{L}^{-1/2}\|_{2}\leq\|\mathcal{F}\|_{H\to\mathbb{R}}/\sqrt{\lambda_{\mathrm{min}}(\mathcal{A})}$.
	To bound $\|\widetilde{A}_{L}^{-1/2}\widetilde{f}_{L}\|_{2}$,
	we consider
	\begin{align*}
		|\widetilde{f}_{L}^{\top}\widetilde{A}_{L}^{-1}\widetilde{f}_{L}| & =v^{\top}B_{L}^{\top}A_{L}^{-1}B_{L}v+v^{\top}B_{L}^{\top}A_{L}^{-1}B_{L}\Gamma^{-1}B_{L}^{\top}A_{L}^{-1}B_{L}v \\
		                                                                                    & =\|A_{L}^{-1/2}B_{L}v\|_{2}^{2}+\|\Gamma^{-1/2}B_{L}^{\top}A_{L}^{-1}B_{L}v\|_{2}^{2}
	\end{align*}
	with $v:=(\Gamma+B_{L}^{\top}A_{L}^{-1}B_{L})^{-1}y_{L}$.
	Assumption~\ref{assu:Pell} on $P_{L}$ implies
	\[
		\frac{\|A_{L}^{-1/2}(B_{L}v)\|_{2}^{2}}{\|B_{L}v\|_{2}^{2}}\leq\max_{y\neq0}\frac{\|y\|_{2}^{2}}{y^{\top}A_{L}y}=\max_{y\neq0}\frac{\|y\|_{2}^{2}}{\|P_{L}y\|_{V}^{2}}\leq\max_{y\neq0}\frac{\|y\|_{2}^{2}}{\lambda_{\mathrm{min}}(\mathcal{A})\|P_{L}y\|_{H}^{2}}\leq\frac{c_{2}^{2}}{\lambda_{\mathrm{min}}(\mathcal{A})(\Phi(L))^{2}},
	\]
	and thus
	$\|A_{L}^{-1/2}B_{L}v\|_{2}^{2}
		\leq
		\frac{c_{2}^{2}\|B_{L}(\Gamma+B_{L}^{\top}A_{L}^{-1}B_{L})^{-1}y_{L}\|_{2}^{2}}{\lambda_{\mathrm{min}}(\mathcal{A})(\Phi(L))^{2}}
		\leq
		\frac{c_{2}^{2}\|\mathcal{B}\|_{H\to\mathbb{R}^{\beta}}^{2}\|\Gamma^{-1}\|_{2}^{2}}{c_{1}^{2}\lambda_{\mathrm{min}}(\mathcal{A})}\|y_{L}\|_{2}^{2}$.
	Similarly, we have
	\[
		\|\Gamma^{-1/2}B_{L}^{\top}A_{L}^{-1}B_{L}v\|_{2}\leq\|\Gamma^{-1/2}\|_{2}\|\mathcal{B}\|_{H\to\mathbb{R}^{\beta}}^{2}\frac{(\Phi(L))^{2}}{c_{1}^{2}}\frac{c_{2}^{2}}{\lambda_{\mathrm{min}}(\mathcal{A})(\Phi(L))^{2}}\|\Gamma^{-1}\|_{2}\|y_{L}\|_{2}.
	\]

	Hence, we conclude
	\begin{align*}
		\|\widetilde{A}_{L}^{-1/2}\widetilde{f}_{L}\|_{2} & \leq\sqrt{\frac{c_{2}^{2}\|\mathcal{B}\|_{H\to\mathbb{R}^{\beta}}^{2}\|\Gamma^{-1}\|_{2}^{2}}{c_{1}^{2}\lambda_{\mathrm{min}}(\mathcal{A})}\|y_{L}\|_{2}^{2}+\|\Gamma^{-1/2}\|_{2}^{2}\|\mathcal{B}\|_{H\to\mathbb{R}^{\beta}}^{4}\frac{c_{2}^{4}}{c_{1}^{4}}\frac{1}{\lambda_{\mathrm{min}}(\mathcal{A})^{2}}\|\Gamma^{-1}\|_{2}^{2}\|y_{L}\|_{2}^{2}} \\
		                                                           & \leq C\|y_{L}\|_{2},
	\end{align*}
	where $C>0$ depends on the norm $\|\mathcal{B}\|_{H\to\mathbb{R}^{\beta}}$,
	the constants $c_{1}$ and $c_{2}$ in Assumption~\ref{assu:Pell},
	and the smallest eigenvalues of $\mathcal{A}$ and $\Gamma^{-1}$
	but independent of $L$.
\end{proof}

\subsection{Optimality of MGMC}\label{sec:optimality}
The MGMC update in Alg. \ref{alg:mgmc} is efficient in the sense that is specified in Corollary~\ref{cor:mgmc_optimality} below. For this, we consider the Bayesian setting in Section~\ref{subsec:Bayesian} with a symmetric positive definite operator $\widetilde{\mathcal{A}}$ that has a compact inverse (which is not necessarily of trace class) and vectors $f_L$ and $F_L$ that are defined as in \eqref{eqn:f_ell_definition} and \eqref{eqn:F_matrix} respectively; we remind the reader that these two vectors are constructed with the help of the bounded linear operators $\mathcal{B}$ and $\mathcal{F}$ introduced in Section~\ref{subsec:Linear-Bayesian-inverse}.
\begin{cor}\label{cor:mgmc_optimality}
  Assume that the assumptions of Proposition~\ref{prop:V-cycle-bound} are satisfied and that the $n_L\times n_L$ matrix $\widetilde{A}_L$ and vectors $f_L, F_L\in \mathbb{R}^{n_L}$ arise from a Bayesian problem as discussed in Sections~\ref{subsec:Linear-Bayesian-inverse} and \ref{sec:discretisation}, where $n_L$ is the number of unknowns. Assume further that the initial sample $\theta_L^{(0)}$ is drawn from a multivariate normal distribution with $\max\{\|\widetilde{A}_L^{-1/2}\mathbb{E}[\theta_L^{(0)}]\|_2,\|\widetilde{A}_L^{1/2} \mathrm{Cov}(\theta_L^{(0)})\widetilde{A}_L^{1/2}\|_2\} \le C_0$ for a constant $C_0$ that is independent of $n_L$.

  Then, the Multigrid Monte Carlo algorithm is algorithmically optimal in the limit $n_L \rightarrow \infty$ in the sense that generating an (approximately) independent sample incurs a cost which grows no more than linearly in the number of unknowns. The generated samples are drawn from a distribution which is exponentially close to the target distribution, independent of the problem size. More specifically:
	\begin{enumerate}
        \item The cost of generating a new sample in the Markov chain with Alg.~\ref{alg:mgmc} is proportional to the number of unknowns $n_L$, for $n_L \to \infty$.
		\item Subsequent samples in the Markov chain are approximately independent in the sense that the integrated autocorrelation time of $F_L^\top \theta_L^{(m)}$ is bounded by a constant that is independent of $n_L$.
		\item The mean and variance of the quantity $F_L^\top\theta_L^{(m)}\in\mathbb{R}$ differ from the mean and variance of the target distribution by a constant that is independent of $n_L$. Moreover, the convergence is exponential, i.e. for the $m$-th state in the Markov chain we have that
			\begin{align*}
				\left|\mathbb{E}[F_L^\top\theta_{L}^{(m)}]-F_L^\top\widetilde{A}_{L}^{-1}f_{L}\right| & \leq C_{1,1} \exp[-C_2 m];             \\
				\left|\mathrm{Var}(F_L^\top\theta_{L}^{(m)})-F_L^\top\widetilde{A}_{L}^{-1}F_L\right|       & \leq C_{1,2} \exp[-2C_2 m],
			\end{align*}
			where $C_{1,1}$, $C_{1,2}$ and $C_2=\log(C_A+\nu)-\log(C_A)$ are positive constants independent of $n_L$. 
                      \end{enumerate}
            Under Assumption~\ref{assu:WtoH-best-Vell-approx}, the sequence $(F_{L}^{\top}\theta_{L}^{*})_{L\in\mathbb{N}}$
            with $\theta_{L}^{*}\sim\mathcal{N}(\widetilde{A}_{L}^{-1}f_{L},\widetilde{A}_{L}^{-1})$
            converges in distribution to the $\mathbb{R}$-valued Gaussian random variable $(\widetilde{v},\mathcal{F})$ with mean and covariance given in Eqs.~\eqref{eq:generalised-mean-post}--\eqref{eq:generalised-cov-post}
            (but with $\chi=\psi=\mathcal{F}$).

\end{cor}
\noindent\textit{Proof.} See Appendix~\ref{sec:proof_optimality}.\\[1ex]
Note that optimality of the computational cost can not be guaranteed for other samplers. For example a (sparse) Cholesky sampler will incur a cost that typically grows faster than linearly in the number of unknowns, in particular in higher dimensions $d$. The numerical experiments in the next section will demonstrate the superior performance of MGMC for large problems. 

Perhaps not very surprisingly, the notion of optimality of MGMC in Corollary~\ref{cor:mgmc_optimality} is closely related to that of multigrid solvers. For this, consider for example \cite[Theorem 11.16]{Hackbusch.W_2016_IterativeSolutionLarge} and \cite[Remark 11.17]{Hackbusch.W_2016_IterativeSolutionLarge} which shows that under certain conditions on the cycle parameter $\gamma$ the cost per multigrid cycle is bounded by a constant times the problem size $n_L$. Further, in \cite[Section 11.5.5]{Hackbusch.W_2016_IterativeSolutionLarge} the application of a nested multigrid iteration to PDE based problems is considered. It is shown there that this approach can reduce the error to the order of the discretisation error at a cost that can be bounded by a constant times $n_L$.

We conclude by demonstrating that the assumptions of Corollary \ref{cor:mgmc_optimality} can be easily satisfied. For this consider the shifted Laplace operator $\mathcal{A}=-\Delta+\kappa^2\id$ as in Example~\ref{example-Laplace} and pick the bounded functional in \eqref{eqn:qoi_linear} for some square-integrable function $\chi$. The operator $\mathcal{A}$ has a compact inverse (but is not trace class for $d>1$). If we pick $\theta_L^{(0)}=0$ (this is what we do in our numerical experiments in Section~\ref{sec:results_convergence}) then $C_0=0$.

\section{Numerical results}\label{sec:results}
We now present numerical results which confirm the theory in Section \ref{sec:theory}. The focus is demonstrating the grid-independent convergence of the Multigrid MC algorithm. We also compare its performance to two other widely used samplers.
\subsection{Setup}\label{sec:setup}
In all cases we sample from the posterior distribution which arises from conditioning a multivariate normal prior by a number of measurements as discussed in Sections \ref{subsec:Linear-Bayesian-inverse} and \ref{sec:discretisation}. More explicitly, this target distribution is
\begin{equation}
    \mathcal{N}(\widetilde{A}_L^{-1}f_L,\widetilde{A}_L^{-1}) \propto \exp\big[-\frac{1}{2}\theta_L^\top \widetilde{A}_L\theta_L + f_L^\top \theta_L \big]
    \qquad\text{with $\widetilde{A}_L = A_L + B_L\Gamma^{-1} B_L^\top$}\label{eqn:target_distribution}
\end{equation}
where $\widetilde{A}_L$ and $B_L$ are specified in the next two sections.
\subsubsection{Prior distribution}\label{sec:prior}
As the prior, we consider $d$-dimensional random Gaussian processes on the unit cube \rev{$D=[0,1]^d$} with mean zero and the following two covariance operators as concrete examples of the operator $\mathcal{A}$ in Section~\ref{sec:sampling_hilbert_space}:
\begin{equation}
    \mathcal{A}^{\mathrm{(SL)}} = -\Delta + \kappa^2\id\qquad \text{(Shifted Laplace)}
    \label{eqn:shifted_laplace}
\end{equation}
with the homogeneous Dirichlet boundary condition $u(x)=0$ for $x\in\partial \rev{D}$ and
\begin{equation}
    \mathcal{A}^{\mathrm{(SSL)}} = (-\Delta + \kappa^2\id)^2=\Delta^2-2\kappa^2\Delta + \kappa^4\id\qquad \text{(Squared Shifted Laplace)}
    \label{eqn:squared_shifted_laplace}
\end{equation}
with $u(x)=\partial{u}/\partial n(x)=0$ for $x\in\partial \rev{D}$. In the numerical experiments we fix the correlation length $\kappa^{-1}=0.1$ in two dimensions and set $\kappa^{-1}=1.0$ in three dimensions unless stated otherwise. Observe that if the operators were defined in the domain $D=\mathbb{R}^d$ instead of the unit cube, then the covariance function of the shifted Laplace operator $\mathcal{A}^{\mathrm{(SL)}}$ in \eqref{eqn:shifted_laplace} would be of Mat\'ern class $\nu=-\frac{1}{2}$ in $d=3$ dimensions and $\nu=0$ in $d=2$ dimensions; for the squared shifted Laplace operator $\mathcal{A}^{(\mathrm{SSL})}$ in \eqref{eqn:squared_shifted_laplace} the Mat\'ern class would be $\nu=1$. For short distances $||x-y||_2\ll \kappa^{-1}$ the covariance function of the problem on $\mathbb{R}^d$ satisfies
\begin{equation}
    \mathrm{Cov}[\phi(x),\phi(y)] \propto \begin{cases}
        z^\nu K_\nu(\kappa r) \\
        K_0(z)                \\
        z^{-\frac{1}{2}}K_{\frac{1}{2}}(z)
    \end{cases}
    =
    \begin{cases}
        \mathcal{O}(1)        & \text{for $\nu>0$}            \\
        \mathcal{O}(-\log(z)) & \text{for $\nu=0$}            \\
        \mathcal{O}(z^{-1})   & \text{for $\nu=-\frac{1}{2}$}
    \end{cases}
    \qquad\text{with $z:=\kappa||x-y||_2\ll 1$.}
\end{equation}
Here $K_\alpha(z)$ is the modified Bessel function of the second kind. It should be stressed that for $\nu=0$ and $\nu=-\frac{1}{2}$ the fields are very rough and sampling them is considered to be numerically challenging.

Since sampling from an infinite dimensional prior is impossible, by dividing the domain  into $n_L^d$ quadrilateral cells of size $h_L^d$ a grid is constructed and the second order operator in \eqref{eqn:shifted_laplace} is discretised with both a simple finite difference discretisation (FD) and with a lowest order conforming piecewise (multi-)linear FE discretisation on this grid in $d=2,3$ dimensions. \rev{While the theory in Section \ref{sec:theory} was derived in the FE framework, we include the FD discretisation since this is often used in practice and to demonstrate that our results also apply in this setting. For the fourth order operator $\mathcal{A}^{\mathrm{(SSL)}}$ in \eqref{eqn:squared_shifted_laplace} we only consider the case $d=2$ and use the 13-point FD discretisation that is written down for example in \cite{bjorstad1983fast}. Note that the FE discretisation of the biharmonic operator is significantly more involved and requires bespoke multigrid patch-smoothers (see e.g. \cite{witte2024tensor}), which is why we do not consider it here.} The homogeneous Dirichlet boundary condition $u(x)=0$ is enforced by implicitly setting the solution on the boundary to zero and only storing the $(n_L-1)^d$ unknowns associated with the interior vertices. For the operator $\mathcal{A}^{\mathrm{(SSL)}}$ we enforce the second boundary condition $\partial u/\partial n(x)=0$ as described in \cite{bjorstad1983fast}.

\rev{The convergence theory in Section~\ref{sec:theory} can in fact be easily adapted to include general boundary conditions, provided the corresponding multigrid solver converges for these boundary conditions.}
\subsubsection{Posterior}\label{sec:posterior}
To construct a posterior, we assume that the observations are obtained by averaging the field $\phi$ over small balls $B_R(\mathring{x}_j)$ of radius $R$ centred at the locations $\mathring{x}_j$ for $j=1,2,\dots,\beta$; we set $\beta=8$ in $d=2$ dimensions and $\beta=32$ in $d=3$ dimensions. The observation operator $\mathcal{B}$ in \eqref{eq:finitedim-obs} can then be written as
\begin{equation}
    \left(\mathcal{B}(\phi)\right)_j = \frac{1}{|B_R(\mathring{x}_j)|} \int_{B_R(\mathring{x}_j)} \phi(x)\;d^dx.
    \label{eqn:observations}
\end{equation}
The individual observations are assumed to be uncorrelated with variance \rev{$\widehat{\sigma}\le \sigma_j^2\le 2\widehat{\sigma}$ with $\widehat{\sigma}=10^{-6}$}, i.e. $\Gamma=\text{diag}(\sigma_1^2,\sigma_2^2,\dots,\sigma_\beta^2)$ is a diagonal matrix. \rev{Results for other values of $\widehat{\sigma}$ are shown in Appendix~\ref{subsec:additional_results}}. In our numerical experiments we condition the Gaussian process on some fixed observed values $1\le \mathring{y}_j\le 4$ of $\left(\mathcal{B}(\phi)\right)_j$  for $j=1,2,\dots,\beta$.
Note that for observations of the form written down in \eqref{eqn:observations} the number of non-zero entries in the corresponding low rank update matrix $B_L\Gamma^{-1} B_L^\top$ is much smaller than the total number of matrix entries $n_L^2$. As a result, the posterior precision matrix $\widetilde{A}_L$ will be sparse.
\subsubsection{Samplers}\label{sec:samplers}
We consider the following three samplers for drawing from the target distribution in \eqref{eqn:target_distribution}.
\paragraph{MGMC sampler.}
To generate a new state $\theta_L^{(m+1)}$ from the current state $\theta_L^{(m)}$ we use the multigrid Monte Carlo update in Alg. \ref{alg:mgmc} with $\nu_1=1$ forward Gibbs-sweeps and $\nu_2=1$ backward Gibbs sweeps on each level, where each sweep consists of an iteration over the entire lattice. When sampling from the posterior we always use the Gibbs-sampler with low-rank updates as written down in Alg. \ref{alg:low_rank_gibbs}. The number of levels is chosen such that the coarsest lattice consists of $n_0^d$ cells where $n_0$ is either an odd number (for example $n_0=3$ if $n_L=48=2^4\cdot 3$) or $n_0=2$ and hence there is only a very small number of interior grid points on this level. On the very coarsest level a few iterations of the symmetric Gibbs sampler with low-rank update are applied.
\jg{In FE methods, the nesting of the function spaces $V_{\ell-1}\subset V_{\ell}$ induces the prolongation, which is equivalent to a multilinear interpolation between the nodal degrees of freedom. We use the same multilinear intergrid operator for the finite difference discretisation.}
For the prior covariance in \eqref{eqn:shifted_laplace} we employ a V-cycle but a W-cycle is used for the prior defined by \eqref{eqn:squared_shifted_laplace}.
\paragraph{Gibbs sampler.}
For comparison, we also consider a standard Gibbs-sampler. Given the current state $\theta_L^{(m)}$ in the Markov chain, a new state $\theta_L^{(m+1)}$ is obtained by $\nu_G=1$ symmetric Gibbs-sweeps, where each symmetric sweep consists of a forward iteration over the entire lattice followed by an analogous backward iteration as defined in Alg. \ref{alg:SGS}. The number of symmetric sweeps $\nu_G=\frac{1}{2}(\nu_1+\nu_2)$ is identical to half the total number of Gibbs sweeps on the finest level of the MGMC sampler. In other words, if the cost for residual calculation and prolongation/restriction are ignored, the MGMC sampler spends approximately the same time on the finest level as the standalone Gibbs sampler.
\paragraph{Cholesky sampler.}
Since this is a widely used method to create i.i.d. samples, we also compare to a sampler based on the Cholesky factorisation. Having computed the factorisation
\begin{equation}
    \widetilde{A}_L = \mathfrak{P}_L^\top U^\top_L U_L \mathfrak{P}_L
\end{equation}
where $U_L$ is an upper triangular matrix and $\mathfrak{P}_L$ is a suitable permutation, we can draw i.i.d. samples $\theta_L^{(m)}$ from the posterior by drawing
an $n_L$ dimensional sample $\xi \sim N(0,\text{Id})$ and solving
\begin{equation}
    U_L \mathfrak{P}_L \theta_L^{(m)} = \xi + g_L\label{eqn:Cholesky_solve}
\end{equation}
for $\theta_L^{(m)}$ where $g_L$ is the solution of the triangular system $U_L^\top g_L = \mathfrak{P}_L f_L$. If $\widetilde{A}_L$ is sparse, the permutation $\mathfrak{P}_L$ can be chosen such as to minimise the number of non-zero entries in $U_L$ which is crucial to make the triangular solve in \eqref{eqn:Cholesky_solve} efficient. In our implemenation we use the Simplicial Cholesky factorisation from the widely used Eigen library \cite{eigenweb}. We find that for the problems considered here this gives slightly better performance than the implementation in the CholMod package \cite{chen2008algorithm}.
\subsubsection{Hardware}
All numerical results were obtained with a sequential C++ implementation developed by the authors, which is freely available at \url{https://github.com/eikehmueller/MultigridMC}. The runs were carried out on an Intel Xeon Platinum 8168 (Skylake) CPU with a clock-speed of 2.70GHz.
\subsection{Performance}\label{sec:results_performance}
We start by empirically confirming the cost analysis in Section \ref{subsec:Cost-analysis} and investigating the grid-independence of the IACT stated in Theorem~\ref{thm:IACT_posterior_robustness}. For this, we measure the time for producing a single Monte Carlo update (with the Gibbs- and Multigrid MC algorithm) and for drawing an independent sample with the Cholesky sampler. Tab.~\ref{tab:performance} shows these results in $d=2$ and $d=3$ dimensions for different priors, we consider both the shifted Laplace operator $\mathcal{A}^{(\mathrm{SL})} = -\Delta+\kappa^2\id$ in \eqref{eqn:shifted_laplace} and its square $\mathcal{A}^{(\mathrm{SSL})} = (-\Delta+\kappa^2\id)^2$ in \eqref{eqn:squared_shifted_laplace}. For the Gibbs and MGMC samplers we also list the IACT (which is 1 for Cholesky). As expected from \eqref{eqn:smoother_cost} and Theorem \ref{thm:mgmc_cost}, the time per sample grows approximately in proportion to the problem size for Multigrid MC and the Gibbs sampler. For the Cholesky sampler the growth in runtime is more rapid, in particular in $d=3$ dimensions: going from the $48\times48\times48$ to the $64\times64\times64$ lattice, the number of unknowns increases by a factor $(63/47)^3\approx 2.4$ but the cost of the Cholesky sampler is $3.8\times$ larger. To account for the fact that the samples in the Markov chain are correlated while the Cholesky sampler produces independent samples, we multiply the time per sample with the IACT for the Gibbs- and Multigrid MC samplers. The resulting \textit{time per independent sample} is shown in the final two columns of Tab.~\ref{tab:performance}, and we use this number for a fair comparison with the Cholesky sampler. One application of the Gibbs sampler is slightly cheaper than a Multigrid MC update for the shifted Laplace operator $\mathcal{A}^{(\mathrm{SL})}$, whereas for $\mathcal{A}^{(\mathrm{SSL})}$ the difference if more pronounced and MGMC is more than twice as expensive as Gibbs. As predicted by Theorem~\ref{thm:IACT_posterior_robustness}, the IACT for MGMC is roughly independent of the resolution. For the shifted Laplace operator $\mathcal{A}^{(\mathrm{SL})}$ it lies 1.1 and 1.4, whereas it is between 2.2 and 4.0 for $\mathcal{A}^{(\mathrm{SSL})}$. This is in stark contrast to the Gibbs sampler, where the IACT grows rapidly and can in fact not be reliably estimated on the finer lattices. Despite being cheaper if a single update is considered, the large IACT means that the Gibbs sampler is not competitive overall: on the finest lattices we considered, producing an independent sample is several orders of magnitude more expensive than with the other two methods. Comparing Cholesky and MGMC, the latter is roughly a factor two slower in $d=2$ dimensions for the shifted Laplace operator $\mathcal{A}^{(\mathrm{SL})}$, for $\mathcal{A}^{(\mathrm{SSL})}$ the difference is even more pronounced with Cholesky being about five times as fast. However, MGMC is significantly faster for the shifted Laplace operator $\mathcal{A}^{(\mathrm{SL})}$ in $d=3$ dimensions, in particular for finer resolutions. This can be attributed to the fact that -- in contrast to Cholesky -- one Multigrid MC update incurs a cost that can be bounded linearly in the number of unknowns and to the grid-independent IACT of Multigrid MC (see also Fig. \ref{fig:autocorrelation}, right).
\renewcommand{\topfraction}{.8}
\renewcommand{\floatpagefraction}{.8}
\begin{table}
    \begin{center}
    \begin{tabular}{|c|rr|rrrrr|}\hline
        \multicolumn{8}{|c|}{\Gape[1ex][1ex]{shifted Laplace $\mathcal{A}^{(\mathrm{SL})}=-\Delta+\kappa^2 I$ in $d=2$ dimensions, FEM discretisation}}\\\hline
         \multirow{2}{*}{grid size}  & \multicolumn{2}{c|}{IACT} & \multicolumn{2}{c}{time / sample } & & \multicolumn{2}{c|}{time / indep. sample}\\
         & Gibbs & MGMC & Gibbs & MGMC & Chol.  & Gibbs & MGMC \\\hline\hline
         $ 32\times 32$  & $3.1 \pm 0.4$ & $1.12 \pm 0.12$ & $0.15$ & $0.22$ & $0.08$ & $0.46$ & $0.25$\\
         $ 64\times 64$  & $10.5 \pm 2.1$ & $1.13 \pm 0.12$ & $0.62$ & $0.91$ & $0.39$ & $6.47$ & $1.04$\\
         $128\times128$  & $47.3 \pm 16.4$ & $1.15 \pm 0.13$ & $2.53$ & $3.71$ & $1.79$ & $119.93$ & $4.28$\\
         $256\times256$  & $95.1 \pm 42.2$ & $1.18 \pm 0.14$ & $10.18$ & $15.17$ & $9.14$ & $968.49$ & $17.96$\\
         $512\times512$  & $264.2 \pm 169.1$ & $1.21 \pm 0.15$ & $43.03$ & $66.69$ & $47.23$ & $11370.84$ & $80.53$\\
        \hline\multicolumn{8}{c}{}\\[-1.5ex]\hline
        \multicolumn{8}{|c|}{\Gape[1ex][1ex]{shifted Laplace $\mathcal{A}^{(\mathrm{SL})}=-\Delta+\kappa^2 I$ in $d=3$ dimensions, FD discretisation}}\\\hline
         \multirow{2}{*}{grid size}  & \multicolumn{2}{c|}{IACT} & \multicolumn{2}{c}{time / sample } & & \multicolumn{2}{c|}{time / indep. sample}\\
         & Gibbs & MGMC & Gibbs & MGMC & Chol.  & Gibbs & MGMC \\\hline\hline
         $ 16\times 16\times 16$  & $2.6 \pm 0.3$ & $1.32 \pm 0.19$ & $0.57$ & $0.74$ & $0.48$ & $1.47$ & $0.98$\\
         $ 32\times 32\times 32$  & $4.7 \pm 0.7$ & $1.20 \pm 0.14$ & $5.20$ & $7.62$ & $12.58$ & $24.60$ & $9.17$\\
         $ 48\times 48\times 48$  & $10.3 \pm 2.0$ & $1.26 \pm 0.17$ & $24.87$ & $30.97$ & $82.81$ & $257.20$ & $39.17$\\
         $ 64\times 64\times 64$  & $20.6 \pm 5.3$ & $1.28 \pm 0.17$ & $50.11$ & $67.89$ & $317.23$ & $1030.82$ & $87.00$\\
        \hline\multicolumn{8}{c}{}\\[-1.5ex]\hline
        \multicolumn{8}{|c|}{\Gape[1ex][1ex]{squared shifted Laplace $\mathcal{A}^{(\mathrm{SSL})}=(-\Delta+\kappa^2 I)^2$ in $d=2$ dimensions, FD discretisation}}\\\hline
         \multirow{2}{*}{grid size}  & \multicolumn{2}{c|}{IACT} & \multicolumn{2}{c}{time / sample } & & \multicolumn{2}{c|}{time / indep. sample}\\
         & Gibbs & MGMC & Gibbs & MGMC & Chol.  & Gibbs & MGMC \\\hline\hline
         $ 32\times 32$  & $22.4 \pm 5.9$ & $2.22 \pm 0.26$ & $0.15$ & $0.36$ & $0.10$ & $3.45$ & $0.81$\\
         $ 64\times 64$  & $3401.7 \pm 4757.4$ & $3.35 \pm 0.43$ & $0.64$ & $1.50$ & $0.52$ & $2167.88$ & $5.02$\\
         $128\times128$  & $1976.3 \pm 2448.2$ & $2.69 \pm 0.35$ & $2.61$ & $6.15$ & $2.49$ & $5148.96$ & $16.55$\\
         $256\times256$  & $2573.9 \pm 3415.3$ & $3.23 \pm 0.40$ & $10.59$ & $25.43$ & $16.13$ & $27265.49$ & $82.01$\\
         $512\times512$  & $1682.9 \pm 1991.1$ & $3.94 \pm 0.57$ & $45.09$ & $110.39$ & $81.68$ & $75871.91$ & $435.17$\\
        \hline
        \end{tabular}

   \caption{IACT, time per sample and time per independent sample for different problem sizes. Results are shown for the posterior with sparse measurements and different priors in $d=2,3$ dimensions. All times are reported in milliseconds.}
        \label{tab:performance}
\end{center}
\end{table}
\subsection{Grid independent convergence}
As shown in Section \ref{sec:theory}, one of the key advantages of the MGMC algorithm is that it shows grid-independent convergence as $h_L\rightarrow 0$. This is in contrast to a naive Gibbs sampler, for which the integrated autocorrelation time and the root mean squared error at fixed sample size grow if the resolution increases. We now demonstrate this by computing several performance indicators of these two samplers. For all results in this section the prior is the FE discretisation of the shifted Laplace operator $\mathcal{A}^{\mathrm{(SL)}}=-\Delta+\kappa^2\id$ defined in \eqref{eqn:shifted_laplace} in $d=2$ dimensions.
\subsubsection{Convergence to the target distribution}\label{sec:results_convergence}
To demonstrate the convergence of the distribution of samples in the Markov chain to the target distribution as described by Theorem \ref{thm:moments-conv}, we consider a Markov chain $\theta_L^{(0)},\theta_L^{(1)},\theta_L^{(2)},\dots$ where $\theta_L^{(0)}\sim \pi_0$ is drawn from some given initial distribution $\pi_0$ which we assume to be multivariate normal. For each $\theta_L^{(m)}$ we define the measurement
\begin{equation}
    z_L^{(m)} = F(\theta_L^{(m)}) := \frac{1}{|B_R(x_{\text{centre}})|} \int_{B_R(x_{\text{centre}})} (P_L\theta_L^{(m)})(x)\;d^dx\label{eqn:measurement_definition}
\end{equation}
where the vector-space isomorphism $P_L:\mathbb{R}^{n_L}\rightarrow V_L$ is defined in Section \ref{sec:discretisation} and $x_{\text{centre}}$ is the centre of the domain $\rev{D}$; the radius $R=0.025$ is the same as for the observations that define the posterior distribution (see \eqref{eqn:observations}). Note that $F$ is a linear operator of the form given in \eqref{eqn:F_matrix} and as a consequence for each step $k$ in the Markov chain $z_L^{(m)}$ is a normal random variable with some mean $\mu_L^{(m)}$ and variance $(\sigma_L^{(m)})^2$. Under suitable conditions we have that $\mu_L^{(m)}\rightarrow \mu_{L}$ and $\sigma_L^{(m)}\rightarrow \sigma_L$ as $k\rightarrow\infty$ (see Theorem \ref{thm:equivalence}). To quantify the rate of convergence we consider the ratios
\begin{xalignat}{2}
    R^{(m)}_L &:= \left|\frac{\mu_L^{(m)}-\mu_L}{\mu_L^{(0)}-\mu_L}\right|, &
    Z^{(m)}_L &:= \left|\frac{(\sigma_L^{(m)})^2-\sigma_L^2}{(\sigma_L^{(0)})^2-\sigma_L^2}\right|.\label{eqn:convergence_rate_definition}
\end{xalignat}
In these expressions $\mu_L^{(m)}$ and $(\sigma_L^{(m)})^2$ are estimated by creating $n_{\text{samples}}$ independent Markov chains $\{\theta_{L;j}^{(m)}\}_{j=1}^{n_{\text{samples}}}$ all starting from $\theta_{L;j}^{(0)}=0$ and computing the sample mean and variance at step $k$ in the chain:
\begin{equation}
    \begin{aligned}
        \widehat{\mu}_L^{(m)}        & := \frac{1}{n_{\text{samples}}}\sum_{j=1}^{n_{\text{samples}}} z_{L;j}^{(m)}\approx \mathbb{E}[z_L^{(m)}] = \mu_L^{(m)}
        \qquad\text{with $z_{L;j}^{(m)}=F(\theta_{L;j}^{(m)})$}                                                                                                                                      \\
        (\widehat{\sigma}_L^{(m)})^2 & := \frac{1}{n_{\text{samples}}-1}\sum_{j=1}^{n_{\text{samples}}} (z_{L;j}^{(m)}-\widehat{\mu}_L^{(m)})^2 \approx \mathrm{Var}[z_L^{(m)}] = (\sigma_L^{(m)})^2 \\
    \end{aligned}
\end{equation}
Replacing $\mu_L^{(m)}\mapsto\widehat{\mu}_L^{(m)}$, $\sigma_L^{(m)}\mapsto\widehat{\sigma}_L^{(m)}$ in \eqref{eqn:convergence_rate_definition}, we can compute the estimators $\widehat{R}_L^{(m)}$, $\widehat{Z}_L^{(m)}$ with associated statistical errors that arise from the finite sample size. Fig. \ref{fig:convergence_mean_variance} shows a plot of these estimators for the first 16 steps in the Markov chain. Note that for both samplers the variance converges faster than the mean, which is consistent with Theorem \ref{thm:moments-conv}. For the Gibbs sampler, convergence is extremely slow. In contrast, for MGMC the estimated mean $\mu_L^{(m)}$ and covariance $(\sigma_L^{(m)})^2$ can not be distinguished from the mean $\mu_L$ and variance $\sigma_L^2$ of the target distribution within statistical errors after a small number of steps.
\begin{figure}
    \begin{center}
        \includegraphics[width=1.0\linewidth]{\figdir/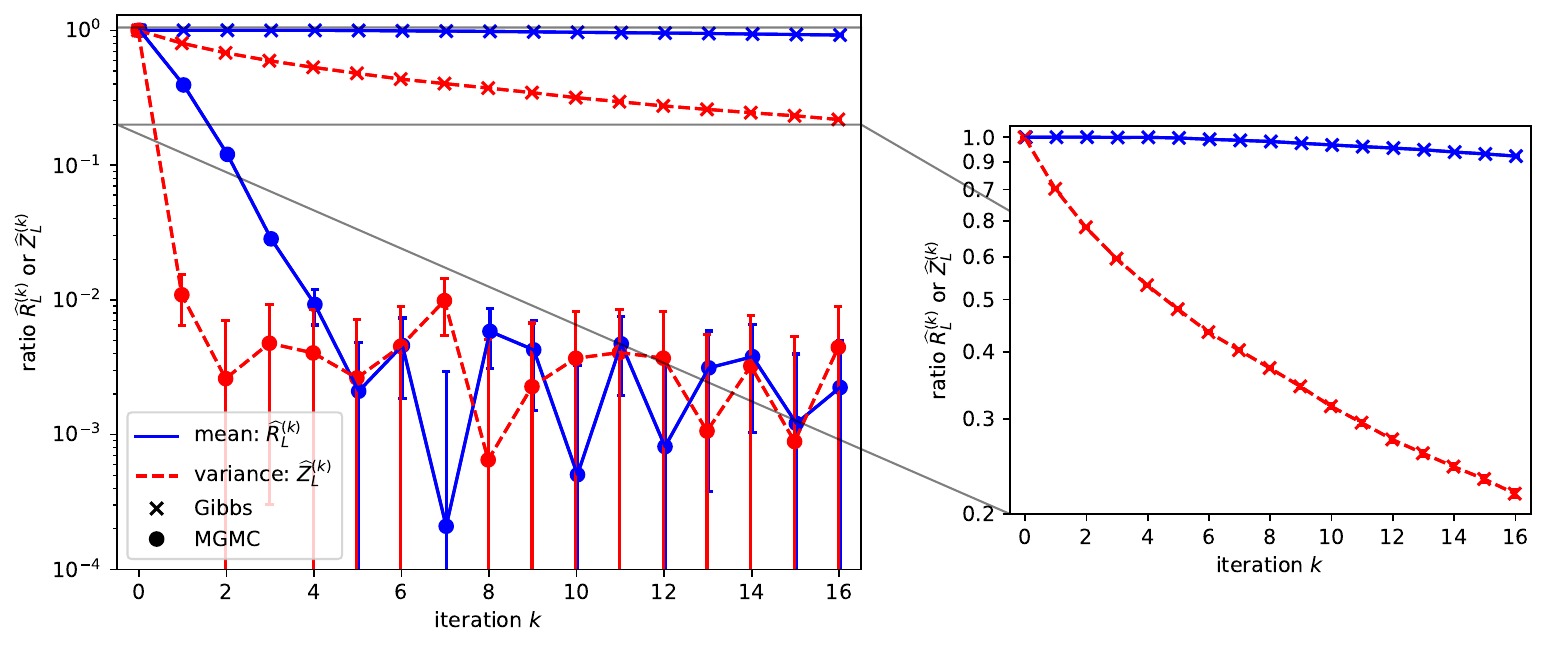}
    \end{center}
    \caption{Convergence of the mean $\mu_L^{(m)}$ and variance $(\sigma_L^{(m)})^2$ for the Gibbs sampler \rev{(crosses)} and for the MGMC sampler \rev{(filled circles)}. The plot shows the estimators $\widehat{R}_L^{(m)}$ \rev{(blue solid lines)} and $\widehat{Z}_L^{(m)}$ \rev{(red dashed lines)} for the quantities defined in \eqref{eqn:convergence_rate_definition}. The grid is of size $128\times 128$ and the number of independent Markov chains is $n_{\text{samples}}=100,000$.}
    \label{fig:convergence_mean_variance}
\end{figure}
To demonstrate grid-independent convergence we define the convergence rates
\begin{xalignat}{2}
    \widehat{\rho}_L &:= \left(\widehat{R}_L^{(m_*)}\right)^{1/m_*}, &
    \widehat{\zeta}_L &:= \left(\widehat{Z}_L^{(m_*)}\right)^{1/m_*}\label{eqn:convergence_rate_definitions}
\end{xalignat}
for some step $m_*$ which we set to be as large as possible but such that the statistical error on $\widehat{R}_L^{(m_*)}$ and $\widehat{Z}_L^{(m_*)}$ does not exceed $10\%$. If this is not possible (for example, because the convergence is extremely rapid) we set $m_*=1$.
\begin{figure}
    \begin{center}
        \includegraphics[width=0.5\linewidth]{\figdir/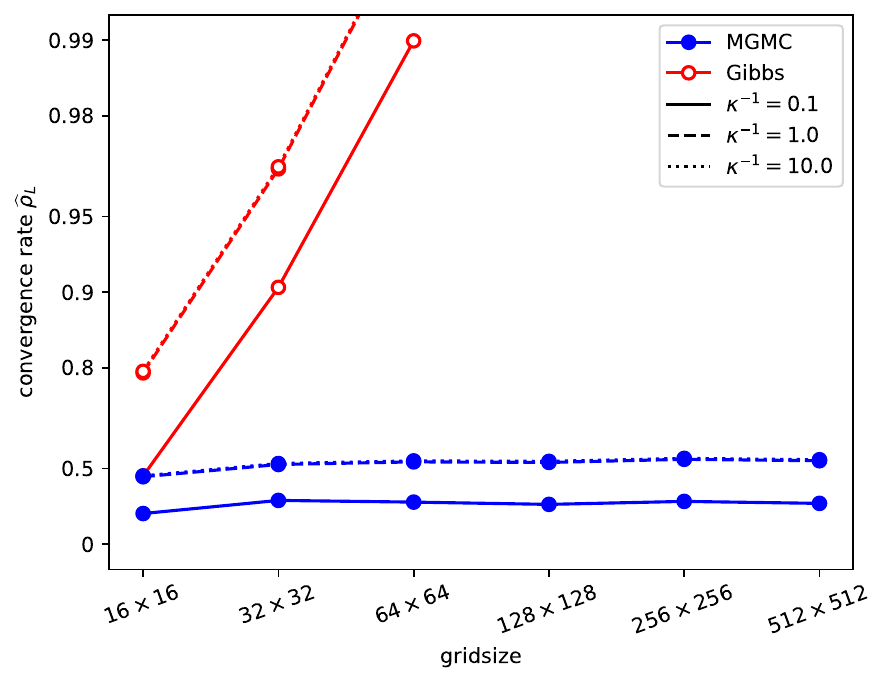}\hfill
        \includegraphics[width=0.5\linewidth]{\figdir/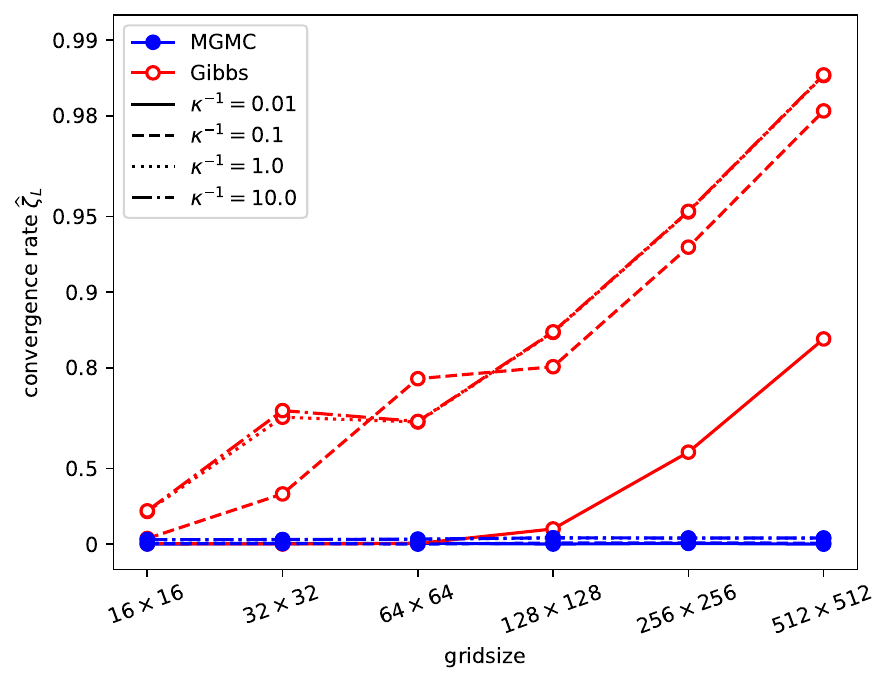}
    \end{center}
    \caption{Dependence of the convergence rates $\widehat{\rho}_L$ (left) and $\widehat{\zeta}_L$ (right) defined in \eqref{eqn:convergence_rate_definitions} on the resolution for different correlation lengths $\kappa^{-1}$. For $\kappa^{-1}=0.01$ the mean $\mu_L^{(m)}$ converged so rapidly for both samplers that the convergence rate $\widehat{\rho}_L$ could not be measured reliably with the given statistics. \rev{While one would expect to see $2\times 3$ different curves for $\widehat{\rho}_L$ and $2\times 4$ curves for $\widehat{\zeta}_L$, the results for $\kappa=1.0$ and $\kappa=10.0$ lie on top of each other in the left figure and they almost overlap in the right figure, in particular for the MGMC sampler.}}
    \label{fig:robustness_convergence_gridindependence}
\end{figure}
Fig. \ref{fig:robustness_convergence_gridindependence} shows how the convergence rates $\widehat{\rho}_L$ and $\widehat{\zeta}_L$ depend on the resolution. Results are shown for different correlation lengths $\kappa^{-1}$. The plot confirms that MGMC shows grid-independent convergence and is robust as the correlation length increases. This should be compared to the Gibbs sampler, for which the convergence rate approaches 1 as the resolution increases. As expected, the Gibbs sampler is also performing worse for larger correlation lengths, which can be explained by the fact that it only carries out local updates. Looking at the results for $\kappa^{-1}=0.01$, the convergence of the Gibbs deteriorates as soon as the correlation length exceeds the grid spacing, i.e. for $\kappa^{-1}\gtrsim h_L$.

\rev{We conclude this section by observing that although in exact arithmetic the Cholesky sampler produces states from the target distribution, in finite precision floating point arithmetic the ill-conditioning of the precision matrix leads to the accumulation of rounding errors. This implies that the generated samples are drawn from the wrong distribution for very large problems. However, for the problem sizes considered in our numerical experiments we did not find this to be an issue.}
\subsubsection{Autocorrelations}\label{sec:results_autocorrelation}
While the Cholesky sampler produces inherently independent samples, the states $\theta_L^{(m)}$ in the Markov chains generated by the Gibbs sampler and the MGMC update in Alg. \ref{alg:mgmc} are inherently correlated. To explore this, we discard the first $n_{\text{\jg{burn-in}}}=1000$ samples from the Markov chain to account for burn-in of the chain so that to a good approximation $\theta_L^{(0)}\sim \mathcal{N}(\widetilde{A}_L^{-1}f_L,\widetilde{A}_L^{-1})$. This will result in a time series $z_L^{(0)},z_L^{(1)},z_L^{(2)},\dots$ with the observable $z_L^{(m)}$ defined as in \eqref{eqn:measurement_definition}. We compute the lagged autocorrelation function $\widehat{\Gamma}_z(t)/\widehat{\Gamma}_z(0)$ defined by
\begin{equation}
    \begin{aligned}
        \quad \widehat{\Gamma}_z(t) & := \frac{1}{n_{\text{steps}}-t}\sum_{m=0}^{n_{\text{steps}-1-t}} (z_L^{(m)}-\widehat{z}_L) (z_L^{(m+t)}-\widehat{z}_L)                           \\
                                    & \approx \mathbb{E}[(z_L^{(m^*)}-\mathbb{E}[z_L]) (z_L^{(m^*+t)}-\mathbb{E}[z_L])]=:\Gamma_z(t)\qquad\text{for some arbitrary $k^*\in\mathbb{N}$}
    \end{aligned}
\end{equation}
where we used the empirical sample mean
\begin{equation}
    \widehat{z}_L := \frac{1}{n_{\text{steps}}}\sum_{m=0}^{n_{\text{steps}}-1} z_L^{(m)} \approx \mathbb{E}[z_L].\label{eqn:sample_mean}
\end{equation}
In all numerical experiments in this section we used $n_{\text{steps}}=10,000$.  Fig. \ref{fig:autocorrelation} (left) shows the lagged autocorrelation function $\widehat{\Gamma}_z(t)/\widehat{\Gamma}_z(0)$ for different lattice sizes. Visually, it is already evident from this figure that subsequent samples generated by the Gibbs samplers are highly correlated, and this correlation grows as the grid resolution increases.
\begin{figure}
    \begin{center}
        \includegraphics[width=0.48\linewidth]{\figdir/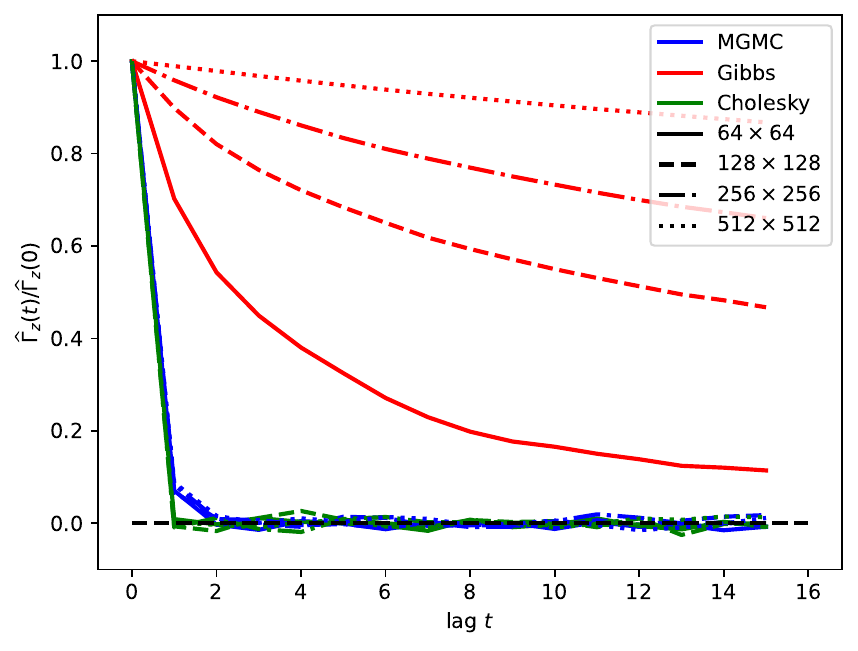}
        \includegraphics[width=0.48\linewidth]{\figdir/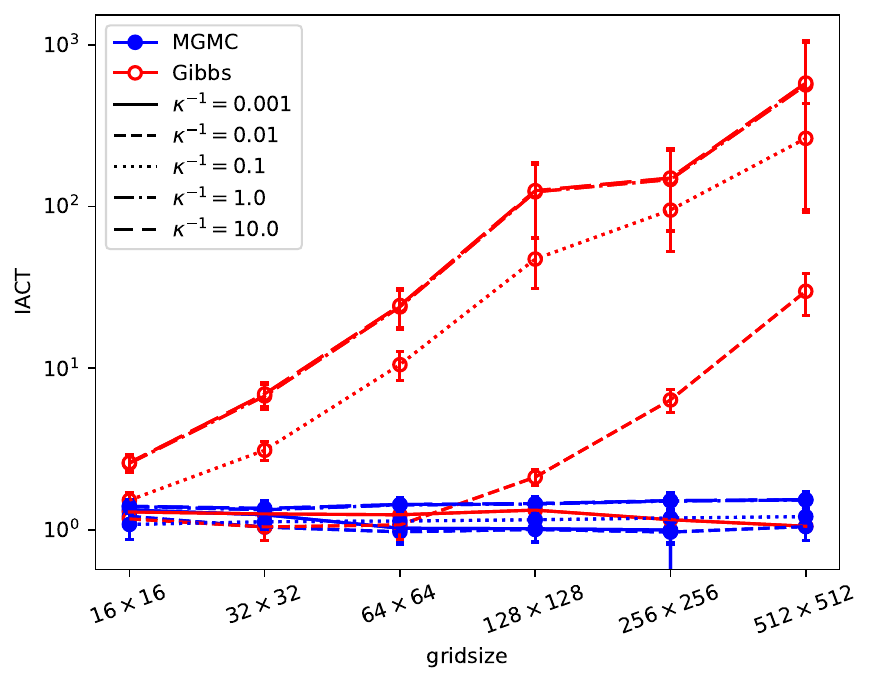}
        \caption{Lagged autocorrelation function $\widehat{\Gamma}_z(t)/\widehat{\Gamma}_z(0)$ for fixed correlation length $\kappa^{-1}=0.1$ (left) and IACT for different $\kappa^{-1}$ (right).}
        \label{fig:autocorrelation}
    \end{center}
\end{figure}
To further quantify autocorrelations in the Markov chain we estimate the integrated autocorrelation time (IACT) defined in \eqref{eqn:tau_int}. Computing the IACT is notoriously difficult if the autocorrelations are strong. Here we use the method in \cite{wolff2004monte} to estimate the IACT, i.e. we compute the estimator
\begin{equation}
    \widehat{\tau}_{\text{int},F_L} := \tau_{\text{int},F_L}(W) := 1+2\sum_{t=1}^{W} \frac{\widehat{\Gamma}_z(t)}{\widehat{\Gamma}_z(0)}
\end{equation}
where the optimal choice of window size $W$ is described in \cite[Section 3.3]{wolff2004monte}.
As stated in Theorem \ref{thm:IACT}, the IACT for MGMC is grid-independent. As in Fig. \ref{fig:robustness_convergence_gridindependence} we also explore the dependence of IACT on the correlation length $\kappa^{-1}$. Fig. \ref{fig:autocorrelation} (right) shows the IACT for both the MGMC sampler and the standard Gibbs sampler as the resolution increases and for a range of different $\kappa^{-1}$.

The plot in Fig. \ref{fig:autocorrelation} (right) demonstrates that the MGMC sampler is robust with respect to both the resolution and the correlation length, as predicted by the theory. This is not the case for the Gibbs sampler, for which the IACT grows strongly as the resolution increases and as the correlation length grows. For the Gibbs sampler the IACT is only small as long as the correlation length does not exceed the grid spacing, i.e. $\kappa^{-1}\lesssim h_L$; consider the curves for $\kappa^{-1}=0.001$ and $\kappa^{-1}=0.01$ in Fig. \ref{fig:autocorrelation} (right).
\subsubsection{Root mean squared error}
To numerically verify the bound in Theorem \ref{thm:integration-error} we compute the root mean squared error (RMSE) of the estimator
\begin{equation}
    \widehat{z}_{L}^{(M)} := \frac{1}{M}\sum_{m=0}^{M-1} z_{L}^{(m)}         \label{eqn:z_L_estimator}
\end{equation}
for different values of $M$ where $z_L^{(m)}$ is obtained by evaluating the function in \eqref{eqn:measurement_definition} on the states of the Markov chain $\theta_L^{(0)}, \theta_L^{(1)},\theta_L^{(2)},\dots$. We proceed as in Section \ref{sec:results_convergence} and create $n_{\mathrm{samples}}=100$ independent Markov chains $\{\theta_{L;j}^{(m)}\}_{j=1}^{n_{\text{samples}}}$ all starting from $\theta_{L;j}^{(0)}=0$ to obtain the following estimator for the quantity on the left-hand side of \eqref{eqn:integration_error}:
\begin{equation}
    \widehat{\Delta}_L^{(M)}          :=\sqrt{\frac{1}{n_{\mathrm{samples}}} \sum_{j=1}^{n_{\mathrm{samples}}} \left(\mu_L - \widehat{z}_{L;j}^{(M)}\right)^2}
    \approx \sqrt{\mathbb{E}\left(\mu_L- \widehat{z}_L^{(M)}\right)^2}
    .\label{eqn:RMSE_estimator}
\end{equation}
Here $\widehat{z}_{L;j}^{(M)}$ is the realisation of \eqref{eqn:z_L_estimator} for the $j$-th Markov chain $\theta_{L;j}^{(0)},\theta_{L;j}^{(1)},\theta_{L;j}^{(2)},\dots$.

Fig. \ref{fig:RMSE} shows $\widehat{\Delta}_L^{(M)}$ for a range of grid spacings. For all samplers the expected asymptotic bound $\widehat{\Delta}_L^{(M)} < C_L M^{-1/2}$ (compare to \eqref{eqn:integration_error}) can be observed empirically. However, the constant $C_L$ is only grid independent for the Cholesky- sampler and the Multigrid MC algorithm, both of which have a comparable RMSE. Furthermore, for all considered resolutions the RMSE is much larger for the Gibbs sampler, and this effect becomes more pronounced on the larger grids.
\begin{figure}[t]
    \begin{center}
        \includegraphics[width=0.78\linewidth]{\figdir/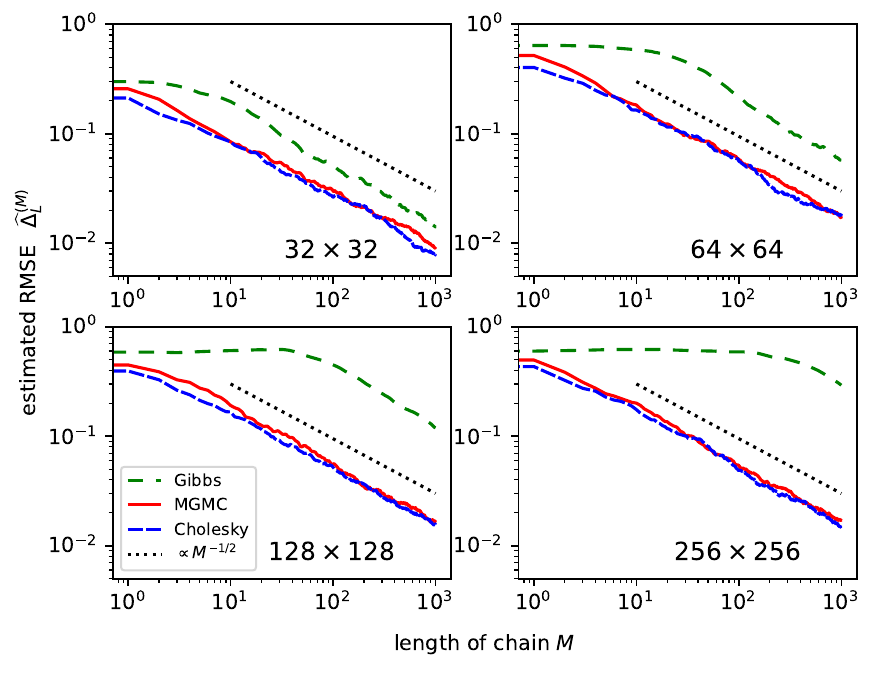}
    \end{center}
    \caption{Estimated root mean squared error $\widehat{\Delta}_L^{(M)}$ as defined in \eqref{eqn:RMSE_estimator} as a function of the length of the Markov chain $M$ for different grid sizes.}
    \label{fig:RMSE}
\end{figure}

\section{Conclusion}\label{sec:conclusion}
In this paper we have presented a rigorous analysis of the Multigrid Monte Carlo approach originally proposed by Goodman and Sokal in \cite{goodman1989multigrid}. We have extended the analysis to the important case of conditioned Gaussian random fields in a linear Bayesian setting. We have shown that the rate of convergence to the targest distribution and the integrated autocorrelation time is grid-independent. To achieve this, we relied on the inherent connection between solvers for sparse linear systems and samplers for multivariate normal distributions discussed in \cite{Fox.C_Parker_2017_AcceleratedGibbsSampling}. This allowed us to bound convergence rates with tools from standard multigrid theory which, however, had to be extended to the Bayesian setup considered here. Our cost-analysis confirms that the cost for one MGMG update grows linearly in the number of unknowns. As a consequence, the algorithm is optimal in the sense that the cost for generating a single independent sample is also proportional to the number of unknowns.

The theoretical results are confirmed by a set of numerical experiments, which demonstrate that MGMC is robust and efficient as the resolution increases. On finer grids, MGMC is always significantly faster than a standard Gibbs sampler. In three dimensions and on larger lattices, it beats even the Cholesky method, which demonstrates that our approach is particularly promising in higher dimensions. We are able to sample very rough fields with precision operators $-\Delta+\kappa^2\id$ and $(-\Delta+\kappa^2\id)^2$ without problems, but find that MGMC is most efficient in the former case.

\rev{As stated above, the MGMC method is applicable in more general setups such as irregular domains or non-stationary distributions where other, potentially more efficient samplers based on spectral methods are not available.}
\paragraph{Future work.}
There are several ways of extending the work in this paper. While we only considered relatively small problems here, significantly larger problems can be simulated with the parallel implementation which will be described in a subsequent publication. For simplicity, we have also limited ourselves to a linear setting where in particular the quantity of interest and the observation operators are bounded linear functionals. It would be interesting to consider more general cases with likelihoods given by $\propto\exp\left[-\frac{1}{2}(\mathcal{B}_\text{NL}(\theta)-\mu)^\top\Gamma^{-1}(\mathcal{B}_\text{NL}(\theta)-\mu)\right]$ for some non-linear $\mathcal{B}_\text{NL}$. Similarly, the generalisation to non-Gaussian priors could be considered. While the original MGMC algorithm in \cite{goodman1989multigrid} is written down for the general non-Gaussian case and these cases should be covered in principle, in practice it will likely have to be adapted to be efficient.

\rev{The numerical results presented here were obtained for covariances of the form $(-\Delta+\kappa^2 I)^{-\alpha}$ with $\alpha=1,2$. It would be interesting to also consider the case of fractional $\alpha$ since this corresponds to a more general class of Mat\'{e}rn covariance functions. However, it should be stressed here that further work is required to understand how MGMC can be applied to this setup.}

\rev{While in this work we assumed that the number of observations $\beta$ is fixed and much smaller than the problem size $n_L$, it will be interesting to also consider the case where $\beta$ is large and grows with some power $n_L$. In this case the performance of the different samplers will likely depend on the distribution of the observations over the domain, and we expect MGMC to be particularly efficient if parts of the domain are not informed by the data. Note that if $\beta\propto n_L^p$ with $p>1$ the implementation needs to be adjusted to achieve optimal performance: for example, in this setup it might be advantageous to pre-compute the $n_\ell\times n_\ell$ matrix product $G_\ell B_\ell^\top$ in \eqref{eqn:low_rank_update} instead of applying the $n_\ell\times \beta$ and $\beta\times n_\ell$ matrices $B_\ell^\top$ and $G_\ell$ in sequence. This will change the bounds on the computational cost and storage requirements in Section \ref{subsec:Cost-analysis}.}

\rev{While here we considered operators that arise from the discretisation of a differential operator on a grid, statistical distributions can also be formulated on graphs. In this case, the precision operator has to be treated as a purely algebraic object with no underlying mesh hierarchy. However, the MGMC philosophy is still applicable in this case: as in algebraic multigrid (AMG) methods \cite{ruge1987algebraic,stuben2001review}, a hierarchy of coarse level spaces can be constructed in a purely algebraic way. The resulting sampler is obtained by replacing the AMG smoother by a suitable sampler. By following the same steps as in Section~\ref{sec:theory}, the convergence theory of the sampler can be reduced to the well-established convergence theory of AMG. This also allows the easy application of MGMC on more complex domains and unstructured grids. A preliminary exploration of this idea, including some numerical experiments and parallel scalability studies, can be found in \cite{friess2024}.}
\subsection*{Acknowledgements}
This work was funded as Exploratory Project 4.6 by the Deutsche Forschungsgemeinschaft (German Research Foundation) under Germany's Excellence Strategy EXC 2181/1 - 390900948 (the Heidelberg STRUCTURES Excellence Cluster). We are grateful to Colin Fox for useful discussions, and thank Nils Friess for his careful reading of the manuscript and detailed comments. \jg{We would like to thank Jonathan Goodman for sending us many extremely helpful comments on the original submission, which have helped to improve this paper.} This work made use of the nimbus cloud computer, and the authors gratefully acknowledge the University of Bath's Research Computing Group (doi.org/10.15125/b6cd-s854) for their support in this work.
\printbibliography
\pagebreak
\appendix

\section{Proofs of results for MGMC invariance and convergence\label{subsec:Proofs}}
This section presents proofs for the main theoretical results given in Sections~\ref{sec:theory_invariance} and \ref{sec:theory_convergence}.
\subsection{Invariance of the coarse level update}\label{sec:proof_MGMC_coarse_level_invariance}
\begin{proof}[Proof of Proposition \ref{prop:coarse_level_invariance}]
	We have
	\begin{align*}
		\mu_{\theta_\ell^{*}}(A) & =\int_{\mathbb{R}^{n_{\ell}}}\int_{\mathbb{R}^{n_{\ell-1}}}1_{A}(\theta_\ell^{*}(\theta_\ell,\Psi_{\ell-1}))p(\theta_\ell,\psi_{\ell-1})\mathrm{d}\psi_{\ell-1}\mathrm{d}\theta_\ell                                                                                                              \\
		                         & =\int_{\mathbb{R}^{n_{\ell}}}\int_{\mathbb{R}^{n_{\ell-1}}}1_{A}(\theta_\ell^{*}(\theta_\ell,\psi_{\ell-1}))p_{\ell-1}(\psi_{\ell-1}|\theta_\ell)p_{\ell}(\theta_\ell)\mathrm{d}\psi_{\ell-1}\mathrm{d}\theta_\ell                                                                                \\
		                         & =\int_{\mathbb{R}^{n_{\ell-1}}}\Biggl(\int_{\mathbb{R}^{n_{\ell}}}1_{A}(\theta_\ell^{*}(\theta_\ell,\psi_{\ell-1}))\frac{h_{\ell}(\theta_\ell+I_{\ell-1}^\ell\psi_{\ell-1})}{Z^*_{\ell-1}(\theta_\ell)}\frac{h_{\ell}(\theta_\ell)}{Z_{\ell}}\mathrm{d}\theta_\ell\Biggr)\mathrm{d}\psi_{\ell-1}.
	\end{align*}
	Since the Lebesgue measure on $\mathbb{R}^{n_{\ell}}$ is invariant
	under the translation $\mathbb{R}^{n_{\ell}}\ni\theta_\ell\mapsto\theta_\ell-I_{\ell-1}^\ell\psi_{\ell-1}$ with $\psi_{\ell-1}\in\mathbb{R}^{n_{\ell-1}}$,
	the inner integral can be rewritten as
	\[
		\int_{\mathbb{R}^{n_{\ell}}}1_{A}(\theta_\ell)\frac{h_{\ell}(\theta_\ell)}{Z^*_{\ell-1}(\theta_\ell-I_{\ell-1}^\ell\psi_{\ell-1})}\frac{h_{\ell}(\theta_\ell-I_{\ell-1}^\ell\psi_{\ell-1})}{Z_\ell}\mathrm{d}\theta_\ell,
	\]
	and thus swapping the order of the integrals again leads to
	\begin{equation}
		\mu_{\theta_\ell^{*}}(A)=\int_{\mathbb{R}^{n_{\ell}}}1_{A}(\theta_\ell)\frac{h_{\ell}(\theta_\ell)}{Z_\ell}\biggl(\int_{\mathbb{R}^{n_{\ell-1}}}\frac{h_{\ell}(\theta_\ell-I_{\ell-1}^\ell\psi_{\ell-1})}{Z^*_{\ell-1}(\theta_\ell-I_{\ell-1}^\ell\psi_{\ell-1})}\mathrm{d}\psi\biggr)\mathrm{d}\theta_\ell.\label{eqn:integral_invariance}
	\end{equation}
	Now, with $g_{\ell-1}$ where $g_{\ell-1}(y_{\ell-1}):=h_{\ell}(\theta_\ell-I_{\ell-1}^\ell\psi_{\ell-1}+I_{\ell-1}^\ell y_{\ell-1})$,
	the positive homogeneity and the translation invariance of the Lebesgue
	integral on $\mathbb{R}^{n_{\ell-1}}$ implies
	\begin{equation}
		\begin{aligned}
			Z^*_{\ell-1} (\theta_\ell-I_{\ell-1}^\ell\psi_{\ell-1}) & =\int_{\mathbb{R}^{n_{\ell-1}}}h_{\ell}(\theta_\ell-I_{\ell-1}^\ell\psi_{\ell-1}+I_{\ell-1}^\ell y_{\ell-1})\,\mathrm{d}y_{\ell-1},                                                               \\
			                                                        & =\int_{\mathbb{R}^{n_{\ell-1}}}g_{\ell-1}(y_{\ell-1})\,\mathrm{d}y=\int_{\mathbb{R}^{n_{\ell-1}}}g_{\ell-1}(-y_{\ell-1})\,\mathrm{d}y_{\ell-1}                                                    \\
			                                                        & =\int_{\mathbb{R}^{n_{\ell-1}}}g_{\ell-1}(\psi_{\ell-1}-y_{\ell-1})\,\mathrm{d}y_{\ell-1}=\int_{\mathbb{R}^{n_{\ell-1}}}h_{\ell}(\theta_{\ell}-I_{\ell-1}^\ell y_{\ell-1})\,\mathrm{d}y_{\ell-1}.
			\label{eqn:Z_ell}
		\end{aligned}
	\end{equation}
	In the last line the linearity of $I_{\ell-1}^{\ell}$ is used. Inserting \eqref{eqn:Z_ell} into \eqref{eqn:integral_invariance} gives the desired result.
\end{proof}
\subsection{Representation of the MGMC iteration}\label{sec:proof_MGMC_representation}
To prove Lemma \ref{lem:MG-iteration-rep} and derive the explicit expression for the MGMC iteration in \eqref{eq:MG-iteration} we make repeated use of the following two propositions:
\begin{prop}\label{prop:preconditioned_iteration}
	Let $A$ and $B$ be $n\times n$ matrices with $A$ positive definite. Given some initial $\theta^{(0)}\in\mathbb{R}^n$, fixed right-hand side $f\in\mathbb{R}^n$ and $w^{(j)}\in\mathbb{R}^n$ for $j=0,1,\dots,\nu-1$, define the iteration
	\begin{equation}
		\theta^{(j+1)} = \theta^{(j)} + B (f-A\theta^{(j)}) + w^{(j)}\quad\text{for $j=0,1,\dots,\nu-1$}.
	\end{equation}
	Then
	\begin{equation}
		\theta^{(\nu)} = S^\nu \theta^{(0)} + (\id-S^\nu)A^{-1}f + W\label{eqn:theta_nu}
	\end{equation}
	where
	\begin{xalignat}{2}
		S &:=\id-BA, &
		W & := \sum_{j=1}^{\nu} S^{j-1} w^{(\nu-j)}\label{eqn:def_XYWS}
	\end{xalignat}
\end{prop}
\begin{proof}
	Using induction and the definition of $S$ in \eqref{eqn:def_XYWS}, it is easy to write down $\theta^{(\nu)}$ in closed form as
	\begin{equation}
		\theta^{(\nu)} = S^\nu\theta^{(0)}+ \sum_{j=0}^{\nu-1} S^j B f + \sum_{j=1}^{\nu} S^{j-1} w^{(\nu-j)}\label{eqn:theta_nu_alt}
	\end{equation}
	Multiplying the first sum in \eqref{eqn:theta_nu_alt} from the left by $A$, using $BA=\id-S$ and evaluating the telescoping sum we find
	\begin{equation}
		\sum_{j=0}^{\nu-1} S^j BA
		= \sum_{j=0}^{\nu-1} S^j (\id-S)= \sum_{j=0}^{\nu-1} \left(S^j -S^{j+1}\right) = \id - S^\nu.
	\end{equation}
	Together with the definitions in \eqref{eqn:def_XYWS} this gives the result in \eqref{eqn:theta_nu}.
\end{proof}
\begin{prop}
	\label{lem:matrix-identity}Let $A\in\mathbb{R}^{n\times n}$ be symmetric
	and invertible. Let the matrices $M,N\in\mathbb{R}^{n\times n}$ be
	such that $A=M-N$ and that $M$ is invertible. Then, with $S:=M^{-1}N$,
	\[
		M^{-1}(M^{\top}+N)M^{-\top}=A^{-1}-SA^{-1}S^{\top}
	\]
	holds.
\end{prop}
\begin{proof}
	As $A$ is symmetric, we have $\id=A^{-1}(M^{\top}-N^{\top})$,
	and thus together with $\id+NA^{-1}=MA^{-1}$,
	\begin{align*}
		M^{\top}+N & =M^{\top}+NA^{-1}(M^{\top}-N^{\top})=(\id+NA^{-1})M^{\top}-NA^{-1}N^{\top} \\
		           & =MA^{-1}M^{\top}-NA^{-1}N^{\top}       =M(A^{-1}-M^{-1}NA^{-1}N^{\top}M^{-\top})M^{\top},
	\end{align*}
	holds, which leads to the statement.
\end{proof}
\begin{proof}[Proof of Lemma \ref{lem:MG-iteration-rep}]
	We will show \eqref{eq:MG-iteration} by mathematical induction. For this, we introduce the following collections of independent multivariate normal random variables for $1\le\ell\le L$
	\begin{equation}
		\begin{aligned}
			 & \{w_{\ell}^{\mathrm{pre}(n)}(m_{\ell},\dots,m_{L-1},m_{L})\mid0\leq n\leq\nu_{1}-1,(m_{\ell},\dots,m_{L-1})\in\mathscr{M}_{\ell}, m_{L}\in \mathbb{N}\}             \\
			 & \quad{\cup}\{w_{\ell}^{\mathrm{post}(n)}(m_{\ell},\dots,m_{L-1},m_{L})\mid0\leq n\leq\nu_{2}-1,(m_{\ell},\dots,m_{L-1})\in\mathscr{M}_{\ell}, m_{L}\in \mathbb{N}\} \\
			 & \qquad{\cup}\{w_{0}^{\mathrm{coarse}(n)}(m_0,\dots,m_{L-1},m_L)\mid0\leq n\leq\nu_{0}-1,(m_0,\dots,m_{L-1})\in\mathscr{M}_0, m_{L}\in \mathbb{N}\}
		\end{aligned}
		\label{eqn:w_samples}
	\end{equation}
	where
	\begin{equation}
		\mathscr{M}_\ell := \{ (m_{\ell},m_{\ell+1},\dots,m_{L-1})\mid 0\leq m_{\ell'}\leq\gamma_{\ell'+1}-1\;\text{for all $\ell'=\ell,\dots,L$}\}
	\end{equation}
	such that for each $n$ and each multiindex $(m_{\ell},\dots,m_{L-1},m_L)\in\mathcal{M}_{\ell+1}\times \mathbb{N}$ the variables are distributed as
	\begin{equation}
		\begin{aligned}
			w_{\ell}^{\mathrm{pre}(n)}(m_{\ell},\dots,m_{L})  & \sim\mathcal{N}(0,(M^{\mathrm{pre}}_{\ell})^{\top}+N^{\mathrm{pre}}_{\ell})   \\
			w_{\ell}^{\mathrm{post}(n)}(m_{\ell},\dots,m_{L}) & \sim\mathcal{N}(0,(M^{\mathrm{post}}_{\ell})^{\top}+N^{\mathrm{post}}_{\ell}) \\
			w_{0}^{\mathrm{coarse}(n)}(m_0,\dots,m_{L})       & \sim\mathcal{N}(0,(M^{\mathrm{coarse}}_{0})^{\top}+N^{\mathrm{coarse}}_{0}).
		\end{aligned}\label{eqn:w_sample_distributions}
	\end{equation} Let
	\begin{equation}
		W_{0}(m_{1},\dots,m_{L}):=\sum_{j=1}^{\nu_{0}}(S_{0}^{\mathrm{coarse}})^{j-1}(M^{\mathrm{coarse}}_{0})^{-1}w_{0}^{\mathrm{coarse}(\nu_0-j)}(m_{1},\dots,m_{L});\label{eq:def-W0}
	\end{equation}
	\begin{equation}
		\begin{aligned}
			W_{\ell}(m_{\ell},\dots,m_{L}) & :=(S^{\mathrm{post}}_{\ell})^{\nu_{2}}\Bigl(I_{\ell-1}^{\ell}\sum_{m=1}^{\gamma_{\ell}}X_{\ell-1}^{m-1}W_{\ell-1}(\gamma_\ell-m,m_{\ell},\dots,m_{L})                                                            \\
			                               & \qquad\qquad+Q_{\ell}\sum_{j=1}^{\nu_{1}}(S^{\mathrm{pre}}_{\ell})^{j-1}(M^{\mathrm{pre}}_{\ell})^{-1}w_{\ell}^{\mathrm{pre}(\nu_1-j)}(m_{\ell},\dots,m_{L})\Bigr)                                               \\
			                               & \qquad\qquad+\sum_{j=1}^{\nu_{2}}(S_{\ell}^{\mathrm{post}})^{j-1}(M_{\ell}^{\mathrm{post}})^{-1}w_{\ell}^{\mathrm{post}(\nu_2-j)}(m_{\ell},\dots,m_{L}),\qquad\text{\ensuremath{\ell\geq1}}.\label{eq:def-W_ell}
		\end{aligned}
	\end{equation}

	We
	start with the coarsest level, i.e. $\ell=0$. For a given set of multiindices $(m_0,\dots,m_L)\in\mathscr{M}_0\times \mathbb{N}$, we can use Proposition \ref{prop:preconditioned_iteration} with $A=A_0$, $B=(M_0^{\mathrm{coarse}})^{-1}$, $f=f_0$, $w^{(j)}=w_0^{\mathrm{coarse}(j)}(m_0,\dots,m_L)$ and $\nu=\nu_0$ to show that the update $\theta_0(m_0,\dots,m_L)=:\theta_0^{\mathrm{init}}\mapsto \theta_0^{\mathrm{new}}:=\theta_0(m_0+1,\dots,m_L)$ can be written as
	\begin{equation}
		\begin{aligned}
			\theta_0^{\mathrm{new}} & = (S_0^{\mathrm{coarse}})^{\nu_0}\theta_0^{\mathrm{init}} + (\id-(S_0^{\mathrm{coarse}})^{\nu_0})A_0^{-1} f_0 + W_0(m_0,\dots,m_L) \\
			                        & = X_0\theta_0^{\mathrm{init}} + Y_0 f_0 + W_0(m_0,\dots,m_L).\label{eqn:theta_new_0}
		\end{aligned}
	\end{equation}
	where we used the definitions of $X_0$, $Y_0$ and $W_0(m_0,\dots,m_L)$ in \eqref{eq:def-X0}, \eqref{eq:def-Y-ell} and \eqref{eq:def-W0}. We conclude that \eqref{eq:MG-iteration} holds on level $\ell=0$.

	Note that, as stated in to Remark \ref{rem:exact_coarse_sampler}, the coarse sampler can be exact. In this case $M_0^{\mathrm{coarse}} = A_0$,	$N_0^{\mathrm{coarse}}=0$ (which implies $X_0=0$, $Y_0=A_0^{-1}$) and \eqref{eqn:theta_new_0} reduces to
	\begin{equation}
		\theta_0^{\mathrm{new}} = A_0^{-1} f_0 +
		A_{0}^{-1}w_{0}^{\mathrm{coarse}}(m_{0},\dots,m_{L}),
	\end{equation}
	with $w_{0}^{\mathrm{coarse}}(m_0,\dots,m_{L}) \sim\mathcal{N}(0,A_0)$. Hence, up to the law 	$\theta_0^{\mathrm{new}}$ is equal to the exact sampler
	\begin{equation}
		\theta_0^{\mathrm{new}} = A_0^{-1} f_0 +
		w_0^{\mathrm{coarse}(0)}(m_0,\dots,m_L).
	\end{equation}
	Next, assume that $\ell\ge 1$ and that the statement in \eqref{eq:MG-iteration} is true on level $\ell-1$. Proposition \ref{prop:preconditioned_iteration} with $A=A_\ell$, $B=(M_\ell^{\mathrm{pre}})^{-1}$, $f=f_\ell$, $w^{(j)}=(M_\ell^{\mathrm{pre}})^{-1}w_\ell^{\mathrm{pre}(j)}(m_{\ell},\dots,m_L)$ and $\nu=\nu_1$ shows that pre-smoothing leads to the update $\theta_{\ell}(m_\ell,\dots,m_L)=: \theta_{\ell}^{\mathrm{init}}\mapsto \theta_{\ell,\nu_1}$ with
	\begin{equation}
		\theta_{\ell,\nu_1} = (S_\ell^{\mathrm{pre}})^{\nu_1} \theta_\ell^{\mathrm{init}} + (\id-(S_{\ell}^{\mathrm{pre}})^{\nu_1}) A_{\ell}^{-1}f_{\ell} + \sum_{j=0}^{\nu_1-1} (S_\ell^{\mathrm{pre}})^j (M_\ell^{\mathrm{pre}})^{-1} w_\ell^{\mathrm{pre}(j)}(m_{\ell},\dots,m_L).
		\label{eqn:pre_update}
	\end{equation}
	To compute the coarse grid correction $\theta_{\ell,\nu_1} \mapsto \theta_{\ell,\nu_1+1}$ observe that on level $\ell-1$ we compute $\psi^{(\gamma_\ell)}_{\ell-1}$ recursively in lines 10--13 of Alg.~\ref{alg:mgmc} as
	\begin{equation}
		\psi^{(m+1)}_{\ell-1} = \mathrm{MGMC}_{\ell-1}(A_{\ell-1},f_{\ell-1},\psi^{(m)}_{\ell-1}),\qquad \psi^{(0)}_{\ell-1}=0\qquad\text{for $m=0,1,\dots,\gamma_\ell-1$}.
	\end{equation}
	According to the inductive assumption the update $\psi_{\ell-1}^{(m)}\mapsto \psi_{\ell-1}^{(m+1)}$ can be written in the form
	\begin{equation}
		\begin{aligned}
			\psi_{\ell-1}^{(m+1)} & =X_{\ell-1} \psi_{\ell-1}^{(m)}+(\id-X_{\ell-1})A_{\ell-1}^{-1} f_{\ell-1} + W_{\ell-1}(m,m_{\ell},\dots,m_L)                          \\
			                      & = \psi_{\ell-1}^{(m)} + (\id-X_{\ell-1})A_{\ell-1}^{-1}(f_{\ell-1}-A_{\ell-1}\psi_{\ell-1}^{(m-1)}) + W_{\ell-1}(m,m_{\ell},\dots,m_L)
		\end{aligned}.
	\end{equation}
	This allows us to apply Proposition \ref{prop:preconditioned_iteration} with $A=A_{\ell-1}$, $B=(\id-X_{\ell-1})A_{\ell-1}^{-1}$, $f=f_{\ell-1}$, $w^{(j)}=W_{\ell-1}(j,m_{\ell},\dots,m_L)$ and $\nu=\gamma_\ell$ to obtain
	\begin{equation}
		\begin{aligned}
			\psi^{(\gamma_\ell)}_{\ell-1} & = (\id-X_\ell^{\gamma_\ell}) A_{\ell-1}^{-1}f_{\ell-1}		+\sum_{m=1}^{\gamma_\ell} X_{\ell-1}^{m-1} W_{\ell-1}(\gamma_\ell-m,m_{\ell},\dots,m_L).
		\end{aligned}
	\end{equation}
	Using the definitions of $f_{\ell-1} = I_\ell^{\ell-1} (f_\ell-A_\ell \theta_{\ell,\nu_1})$ and $\theta_{\ell,\nu_1+1}=\theta_{\ell,\nu_1} + I_{\ell-1}^{\ell} \psi_{\ell-1}^{(\gamma_\ell)}$ in lines 9 and 14 of Alg.~\ref{alg:mgmc}, some straightforward algebra shows that this leads to the update $\theta_{\ell,\nu_1}\mapsto \theta_{\ell,\nu_1+1}$ with
	\begin{equation}
		\theta_{\ell,\nu_1+1} = Q_\ell \theta_{\ell,\nu_1} + I_{\ell-1}^{\ell}(\id - X_{\ell-1}^{\gamma_\ell})A_{\ell-1}^{-1} I_{\ell}^{\ell-1} f_\ell
		+ I_{\ell-1}^{\ell} \sum_{m=1}^{\gamma_\ell} X_{\ell-1}^{m-1} W_{\ell-1}(\gamma_\ell-m,m_{\ell},\dots,m_L).
		\label{eqn:cgc_update}
	\end{equation}
	Finally, another application of Proposition \ref{prop:preconditioned_iteration} with $A=A_\ell$, $B=(M_\ell^{\mathrm{post}})^{-1}$, $f=f_\ell$, $w^{(j)}=(M_\ell^{\mathrm{post}})^{-1}w_\ell^{\mathrm{post}(j)}(m_{\ell},\dots,m_L)$ and $\nu=\nu_2$ shows that post-smoothing results in the update $\theta_{\ell,\nu_1+1}\mapsto \theta_{\ell}^{\mathrm{new}}:=\theta_{\ell}(m_\ell+1,\dots,m_L)$ with
	\begin{equation}
		\theta_{\ell}^{\mathrm{new}} = (S_\ell^{\mathrm{post}})^{\nu_2} \theta_{\ell,\nu_1+1} + (\id-(S_{\ell}^{\mathrm{post}})^{\nu_2}) A_{\ell}^{-1}f_{\ell} + \sum_{j=1}^{\nu_2} (S_\ell^{\mathrm{post}})^{j-1} (M_\ell^{\mathrm{post}})^{-1} w_\ell^{\mathrm{post}(\nu_2-j)}(m_{\ell},\dots,m_L)
		\label{eqn:post_update}
	\end{equation}
	Combining \eqref{eqn:pre_update}, \eqref{eqn:cgc_update} and \eqref{eqn:post_update} and using the definitions of $X_\ell$ in \eqref{eq:def-Xell-Qell} and of $W_\ell(m_{\ell},\dots,m_L)$ in \eqref{eq:def-W_ell} results in
	\begin{equation}
		\theta_\ell^{\mathrm{new}} = X_\ell \theta_\ell^{\mathrm{init}}+ \overline{Y}_\ell f_\ell + W_\ell(m_{\ell},\dots,m_L)
	\end{equation}
	where
	\begin{equation}
		\overline{Y}_\ell := (\id-(S_\ell^{\mathrm{post}})^{\nu_2})A_\ell^{-1} + (S_\ell^{\mathrm{post}})^{\nu_2}\left(I_{\ell-1}^{\ell}(\id-X_{\ell-1}^{\gamma_\ell})A_{\ell-1}^{-1}I_{\ell}^{\ell-1} + Q_\ell (\id-(S_\ell^{\mathrm{pre}})^{\nu_1})A_{\ell}^{-1}\right).
	\end{equation}
	It is easy to see that $\overline{Y}_\ell A_\ell = \id -X_\ell$ and thus $\overline{Y}_\ell=Y_\ell$, which shows \eqref{eq:MG-iteration}.

	To derive an expression for the covariance $\WCov_{\ell}=\mathbb{E}[W_\ell W_\ell^\top]$ of $W_\ell$, the statement in \eqref{eq:noise-cov-id} is shown by induction over the levels $\ell$. First consider the case $\ell=0$. Using
	the independence of $\{w_{0}^{\mathrm{coarse}(\nu_{0}-j)}(m_{0},\dots,m_{L})\}_{1\leq j\leq\nu_{0}}$
	Proposition\ref{lem:matrix-identity}, we have
	\begin{align}
		\WCov_{0} & =\sum_{j=1}^{\nu_{0}}S_{0}^{j-1}M_{0}^{-1}(M_{0}^{\top}+N_{0})M_{0}^{-\top}(S_{0}^{\top})^{j-1}=\sum_{j=1}^{\nu_{0}}S_{0}^{j-1}\bigl(A_{0}^{-1}-S_{0}A_{0}^{-1}S_{0}^{\top}\bigr)(S_{0}^{\top})^{j-1} \\
		          & =S_{0}^{1-1}A_{0}^{-1}(S_{0}^{\top})^{1-1}-S_{0}^{\nu_{0}}A_{0}^{-1}(S_{0}^{\top})^{\nu_{0}}=A_{0}^{-1}-X_{0}A_{0}^{-1}X_{0}.
		\label{eqn:Gamma_0}
	\end{align}
	Note that for the exact sampler we have $X_{0}=0$ and \eqref{eqn:Gamma_0} reduces to $\WCov_{0}=A_{0}^{-1}$.

	For $\ell\geq1$, suppose that \eqref{eq:noise-cov-id} holds
	on level $\ell-1$. Noting the independence of
	\[
		\{w_{0}^{\mathrm{coarse}(n_0)}(m_0,\dots,m_L),w_{\ell}^{\mathrm{pre}(n_1)}(m_\ell,\dots,m_L),w_{\ell}^{\mathrm{post}(n_2)}(m_\ell,\dots,m_L)\},
	\]
	from Proposition\ref{lem:matrix-identity}, calculations analogous to above
	lead to
	\begin{align*}
		\WCov_{\ell} & =\hat{S}_{1}^{\nu_{2}}\biggl(I_{\ell-1}^{\ell}\sum_{k=1}^{\gamma_{\ell}}X_{\ell-1}^{k-1}\WCov_{\ell-1}(X_{\ell-1}^{\top})^{k-1}I_{\ell}^{\ell-1}+Q_{\ell}(A_{\ell}^{-1}-S_{\ell}^{\nu_{1}}A_{\ell}^{-1}(S_{\ell}^{\top})^{\nu_{1}})Q_{\ell}^{\top}\biggr)(\hat{S}_{1}^{\top})^{\nu_{2}} \\
		             & \qquad\qquad+A_{\ell}^{-1}-\hat{S}_{\ell}^{\nu_{2}}A_{\ell}^{-1}(\hat{S}_{\ell}^{\top})^{\nu_{2}}
	\end{align*}
	Now, the identity in \eqref{eq:noise-cov-id} for $\ell-1$ implies
	\[
		\sum_{m=1}^{\gamma_{\ell}}X_{\ell-1}^{m-1}\WCov_{\ell-1}(X_{\ell-1}^{\top})^{m-1}=A_{\ell-1}^{-1}-X_{\ell-1}^{\gamma_{\ell}}A_{\ell-1}^{-1}(X_{\ell-1}^{\top})^{\gamma_{\ell}}.
	\]
	Moreover, the symmetry of $A_{\ell}$ and $I_{\ell}^{\ell-1}A_{\ell}I_{\ell-1}^{\ell}=A_{\ell-1}$
	imply
	\begin{align*}
		Q_{\ell}A_{\ell}^{-1}Q_{\ell}^{\top}
		 & =A_{\ell}^{-1}-I_{\ell-1}^{\ell}\Bigl(A_{\ell-1}^{-1}-X_{\ell-1}^{\gamma_{\ell}}A_{\ell-1}^{-1}(X_{\ell-1}^{\top})^{\gamma_{\ell}}\Bigr)I_{\ell}^{\ell-1}.
	\end{align*}
	Hence, $\WCov_{\ell}$ above can be rewritten as
	\begin{align*}
		\WCov_{\ell} & =\hat{S}_{1}^{\nu_{2}}\Bigl(A_{\ell}^{-1}-Q_{\ell}S_{\ell}^{\nu_{1}}A_{\ell}^{-1}(S_{\ell}^{\top})^{\nu_{1}}Q_{\ell}^{\top}\Bigr)(\hat{S}_{1}^{\top})^{\nu_{2}}+A_{\ell}^{-1}-\hat{S}_{\ell}^{\nu_{2}}A_{\ell}^{-1}(\hat{S}_{\ell}^{\top})^{\nu_{2}} =A_{\ell}^{-1}-X_{\ell}A_{\ell}^{-1}X_{\ell}^{\top},
	\end{align*}
	which concludes the inductive proof of \eqref{eq:noise-cov-id} for all
	$\ell=0,\dots,L$.
\end{proof}
\subsection{Recursion formulae for mean and covariance}\label{sec:proof_recursion}
\begin{proof}[Proof of Lemma~\ref{lem:MGidentity-mean-cov}]
	Each new state $\theta_L^{(m+1)}$
	is obtained from $\theta_L^{(m)}$ according to the update rule in \eqref{eq:MG-iteration_finest_level} with some $W_L^{(m)}$ as constructed in Lemma \ref{lem:MG-iteration-rep}.
	Applying \eqref{eq:MG-iteration_finest_level} repeatedly we find that the state $\theta_L^{(m)}$ can be expressed as a linear combination of random variables
	\begin{equation}\label{eq:theta-itr-simple}
		\theta_L^{(m)} = b_0 + \sum_{m'=0}^{m-1} B_m W_L^{(m')}
	\end{equation}
	for some random variable $b_0\in\mathbb{R}^{n_L}$ independent of the random variables in \eqref{eqn:w_samples}, and some matrices $B_m\in\mathbb{R}^{n_L\times n_L}$.
	From the definitions in \eqref{eqn:w_samples}, \eqref{eqn:w_sample_distributions}, \eqref{eq:def-W0} and \eqref{eq:def-W_ell},
	each $W_L^{(m)}$ is a zero-mean multivariate normal random variable because it is a linear combination of zero-mean multivariate normal random variables from \eqref{eqn:w_samples}.
	Taking the expectation value of \eqref{eq:MG-iteration_finest_level} implies that
	\begin{equation}
		\mathbb{E}[\theta_{L}^{(m+1)}]=X_{L}\mathbb{E}[\theta_{L}^{(m)}]+Y_{L}f_{L}.
		\label{eqn:theta_L_m_mean_update}
	\end{equation}
	Let $u_{L}\in\mathbb{R}^{n_{L}}$ be the solution of $A_{L}u_{L}=f_{L}.$
	Then, from $Y_Lf_{L}=(\id -X_L)A_L^{-1}f_{L}$, we see that
	\begin{equation}
		u_{L}=X_{L}u_{L}+Y_{L}f_{L}\label{eqn:u_L_update}
	\end{equation}
	holds. Taking the difference of \eqref{eqn:theta_L_m_mean_update} and \eqref{eqn:u_L_update} yields \eqref{eqn:error_iteration}.

	To prove the identity for the evolution of the covariance in \eqref{eqn:convergence_iteration},
	we note that
	$\theta_L^{(m)}$ and $W_L^{(m)}$ are independent; indeed these are two linear combinations of disjoint subsets of the collection  \eqref{eqn:w_samples} of independent random variables.
	Using the independence of $\theta_L^{(m)}$ and $W_L^{(m)}$,
	and the fact that $f_L$ is deterministic,
	with \eqref{eq:MG-iteration_finest_level} we can compute the covariance of $\theta_{L}^{(m+1)}$ as
	\[
		\mathrm{Cov}(\theta_{L}^{(m+1)})=X_{L}\mathrm{Cov}(\theta_{L}^{(m)})X_{L}^{\top}+Y_{L}\mathrm{Cov}(f_{L})Y_{L}^{\top}+\WCov_{L}=X_{L}\mathrm{Cov}(\theta_{L}^{(m)})X_{L}^{\top}+\WCov_{L}.
	\]
	In view of \eqref{eq:noise-cov-id} in Lemma~\ref{lem:MG-iteration-rep}, subtracting $A_{L}^{-1}$
	from both sides yields
	\[
		\mathrm{Cov}(\theta_{L}^{(m+1)})-A_{L}^{-1}=X_{L}\bigl(\mathrm{Cov}(\theta_{L}^{(m)})-A_{L}^{-1}\bigr)X_{L}^{\top}.
	\]
	Using the same arguments as above, the update rule in \eqref{eq:MG-iteration_finest_level} and the definition of $W_L^{(m)}$ imply that $\mathbb{E}[\theta_{L}^{(m)}W_{L}^{(m+s-1)}]=\mathbb{E}[W_{L}^{(m+s-1)}]=0$ for $s\ge1$. Since $f_L$ is fixed we get the recursion
	\begin{align*}
		\mathrm{Cov}(\theta_{L}^{(m+s)},\theta_{L}^{(m)}) & =\mathrm{Cov}(X_{L}\theta_{L}^{(m+s-1)},\theta_{L}^{(m)})+\mathrm{Cov}(Y_{L}f_{L},\theta_{L}^{(m)})+\mathrm{Cov}(W_{L}^{(m+s-1)},\theta_{L}^{(m)}) \\
		                                                  & =X_{L}\mathrm{Cov}(\theta_{L}^{(m+s-1)},\theta_{L}^{(m)}).
	\end{align*}
	From this the result in \eqref{eqn:cross_covariance_iteration} follows by induction over $s$.
\end{proof}
With the iteration formula in Lemma~\ref{lem:MG-iteration-rep} and the recursion relation for the mean and covariance in \eqref{eqn:error_iteration} and \eqref{eqn:convergence_iteration} in Lemma \ref{lem:MGidentity-mean-cov}, we are now ready to show the central equivalence result.
\subsection{Equivalence of Multigrid and MGMC}\label{sec:proof_equivalence}
\begin{proof}[Proof of Theorem~\ref{thm:equivalence}]
	From Lemma~\ref{lem:MG-iteration-rep} and its proof, we see that  $u_{L}^{(m)}=u_{L}^{(m)}(u_{L}^{(0)})$
	and $\theta_{L}^{(m)}=\theta_{L}^{(m)}(\theta_{L}^{(0)})$ are given
	by the following iterations:
	\begin{equation}
		u_{L}^{(m+1)}=X_{L}u_{L}^{(m)}+Y_{L}f_{L};\label{eq:det-iter}
	\end{equation}
	\begin{equation}
		\theta_{L}^{(m+1)}=X_{L}\theta_{L}^{(m)}+Y_{L}f_{L}+W_{L}^{(m)}.\label{eq:sto-iter}
	\end{equation}

	Suppose (i) holds. Then, from \cite[Theorem 3.5.1]{Young.D.M_book_1971},
	the spectral radius of $X_{L}$ is less than $1$, so that $\lim_{m\to\infty}X_{L}^{m}=0$.
	But from
	\eqref{eqn:error_iteration} and \eqref{eqn:convergence_iteration} in Lemma~\ref{lem:MGidentity-mean-cov}, we have
	\[
		\lim_{m\to\infty}\mathbb{E}[\theta_{L}^{(m)}]=A_{L}^{-1}f_{L}\quad\text{and}\quad\lim_{m\to\infty}\mathrm{Cov}(\theta_{L}^{(m)})=A_{L}^{-1},
	\]
	whatever $\theta_{L}^{(0)}$ is. Hence, the characteristic function
	of $\theta_{L}^{(m)}$ converges to that of $\theta\sim\mathcal{N}(A_{L}^{-1}f_{L},A_{L}^{-1})$.
	This shows (ii).

	To show ((ii)$\implies$(i)) we show the contraposition. 
	To show (ii) does not hold,
	expecting a contradiction suppose it does.
	Then, $\theta_{L}^{(m)}$
	converges in distribution to
	$\mathcal{N}(A_L^{-1}f_L,A_L^{-1})$
	whatever
	the initial state is.
	Then, we must have
	$\lim_{m\to\infty}\mathbb{E}[\theta_{L}^{(m)}]=A_L^{-1}f_L$,
	as we will show below. But since (i) does not hold, 
	we can choose $\theta_{L}^{(0)}$ such that
	$\lim_{m\to\infty}\mathbb{E}[\theta_{L}^{(m)}]\neq A_L^{-1}f_L$
	because upon taking the expectation the iteration \eqref{eq:sto-iter} is identical
	to \eqref{eq:det-iter}, a contradiction.

	It remains to show $\lim_{m\to\infty}\mathbb{E}[\theta_{L}^{(m)}]=A_L^{-1}f_L$.
	For this, we prove the following: if $(\theta_{L}^{(m)})_{m\in\mathbb{N}}$ converges to a random variable $\theta^{*}$
	in distribution, then $\theta^{*}$ is integrable
	and $\lim_{m\to\infty}\mathbb{E}[\theta_{L}^{(m)}]=\mathbb{E}[\theta^{*}]$.
	To see this, first we use the pointwise convergence of the characteristic
	function
	\[
		\lim_{m\to\infty}\left|\exp\left(\mathrm{i}t^{\top}\mathbb{E}[\theta_{L}^{(m)}]-\frac{1}{2}t^{\top}\mathrm{Cov}(\theta_{L}^{(m)})t\right)\right|=\lim_{m\to\infty}\exp\left(-\frac{1}{2}t^{\top}\mathrm{Cov}(\theta_{L}^{(m)})t\right)=|\varphi(t)|\quad\text{for all }t\in\mathbb{R}^{n_{L}},
	\]
	where $\varphi$ is the characteristic function of the limiting law. From this we deduce that $\mathrm{Cov}(\theta_{L}^{(m)})$ is convergent.
	Next, we deduce that $\mathbb{E}[\theta_{L}^{(m)}]$ is convergent from the following convergence:
	\begin{equation}
		\lim_{m\to\infty}\exp\bigl(\mathrm{i}t^{\top}\mathbb{E}[\theta_{L}^{(m)}]\bigr)=\varphi(t)\lim_{m\to\infty}\exp\left(\frac{1}{2}t^{\top}\mathrm{Cov}(\theta_{L}^{(m)})t\right)=\varphi(t)\ln|\varphi(t)|^{-2}\quad\text{for all }t\in B_{\epsilon}(0),
	\end{equation}
	where 	$B_{\epsilon}(0)\subset\mathbb{R}^{n_{L}}$ is an $\epsilon$-ball on which $\varphi\neq 0$, which exists because $\varphi$ is continuous at $\boldsymbol{0}$
	and $\varphi(\boldsymbol{0})=1$.
	Hence, we conclude that $\theta_{L}^{(m)}$
	converging in distribution implies  $\mathbb{E}[\theta_{L}^{(m)}]$
	is convergent. Since $\mathrm{Cov}(\theta_{L}^{(m)})$ is also convergent,
	the limiting distribution needs to be Gaussian with mean $\lim_{m\to\infty}\mathbb{E}[\theta_{L}^{(m)}]$
	and covariance $\lim_{m\to\infty}\mathrm{Cov}(\theta_{L}^{(m)})$.
	Hence, $\lim_{m\to\infty}\mathbb{E}[\theta_{L}^{(m)}]=\mathbb{E}[\theta^{*}]$.

\end{proof}
\subsection{Convergence of distributions}\label{sec:proof_distribution_convergence}
\begin{proof}[Proof of Corollary~\ref{cor:conv-dist}]
	Let $p_{m}$ be the density function of $\theta_{L}^{(m)}$ and $p$
	be that of $\mathcal{N}(A_{L}^{-1}f_{L},A_{L}^{-1})$. We
	have
	\begin{align*}
		2{\displaystyle D_{\text{KL}}\bigl(({\mu_{m}},\ensuremath{\Sigma_{m}}})\parallel(A_{L}^{-1}f_{L},A_{L}^{-1})\bigr) & :=2\int_{\mathbb{R}^{n_{L}}}(\ln p_{m}(x)-\ln p(x))p_{m}(x)\mathrm{d}x                                                                                                                     \\
		=\operatorname{tr}\left(A_{L}\Sigma_{m}\right)                                                                              & -n_{L}+\left(A_{L}^{-1}f_{L}-\ensuremath{\mu_{m}}\right)^{\top}A_{L}(A_{L}^{-1}f_{L}-\ensuremath{\mu_{m}})+\ln\frac{\mathrm{det}(A_{L}^{-1})}{\mathrm{det}(\Sigma_{m})};
	\end{align*}
	see \cite[(6.32)]{murphy2022probabilistic}. Here, we note that for
	sufficiently large $m$, the continuity of the determinant together
	with Theorem~\ref{thm:moments-conv} implies $\mathrm{det}(\Sigma_{m})\neq0$.
	The continuity of the trace function and determinant implies $\lim_{m\to\infty}\operatorname{tr}(A_{L}\Sigma_{m})=n_{L}$
	and $\lim_{m\to\infty}\frac{\mathrm{det}(A_{L}^{-1})}{\mathrm{det}(\Sigma_{m})}=1$.
	Hence,
	\[
		\lim_{m\to\infty}{\displaystyle D_{\text{KL}}\bigl((\text{\ensuremath{\mu_{m}}, \ensuremath{\Sigma_{m}}})\parallel(A_{L}^{-1}f_{L},A_{L}^{-1})\bigr)=0.}
	\]
	Similarly, $\lim_{m\to\infty}{ D_{\text{KL}}\bigl((A_{L}^{-1}f_{L},A_{L}^{-1})\parallel(\text{\ensuremath{\mu_{m}}, \ensuremath{\Sigma_{m}}})\bigr)=0}$
	holds, where we again note $\mathrm{det}(\Sigma_{m})\neq0$ for sufficiently
	large $m$. In view of the Pinsker's inequality, the total variation
	distance converges to $0$ as well, which implies the weak convergence.
	The weak convergence can be checked directly with the convergence
	of characteristic function.
\end{proof}
\subsection{Bounds on the IACT}\label{sec:proof_IACT_bounds}
\begin{proof}[Proof of Theorem~\ref{thm:IACT}]
	\rev{To show the last statement,}  
	note that if $\theta_L^{(0)}\sim \mathcal{N}(\mu_L,A_L^{-1})$ then according to Theorem~\ref{thm:mgmc_invariance} $\theta_L^{(m)}\sim \mathcal{N}(\mu_L,A_L^{-1})$ for all $m\ge 0$. 
	With \eqref{eq:MG-corr} in Lemma~\ref{lem:MGidentity-mean-cov} this implies that
	\[
		\begin{aligned}
			\text{Cov}(F_{L}^{\top}\theta_L^{(m+s)},F_{L}^{\top}\theta_L^{(m)}) & =
			F_{L}^{\top} \text{Cov}(\theta_L^{(m+s)},\theta_L^{(m)})F_{L}                           =
			F_{L}^{\top} X_L^s \text{Cov}(\theta_L^{(m)},\theta_L^{(m)})F_{L}                    = F_{L}^{\top} X_L^s A_L^{-1}F_{L}.
		\end{aligned}
	\]
	Inserting this into \eqref{eqn:tau_int_definition} gives the desired result.
	To derive a bound on $\rev{\tau^{(\infty)}_{\mathrm{int},F_L}}$, i.e. to show the first inequality in \eqref{eqn:tau_int_bounds}, observe that the individual terms in the sum in \eqref{eqn:tau_int} can be bound as follows if we set $w:=A_{L}^{-1/2}F_{L}$
	\begin{equation}
		\begin{aligned}\frac{F_{L}^{\top}X_{L}^{s}A_{L}^{-1}F_{L}}{F_{L}^{\top}A_{L}^{-1}} & \le\left|\frac{F_{L}^{\top}X_{L}^{s}A_{L}^{-1}F_{L}}{F_{L}^{\top}A_{L}^{-1}F_{L}^{\top}}\right|  =\left|\frac{w^{\top}A_{L}^{1/2}X_{L}^{s}A_{L}^{-1/2}w}{w^{\top}w}\right|  \le||A_{L}^{1/2}X_{L}^{s}A_{L}^{-1/2}||_{2}\le||X_{L}||_{A_{L}}^{s} ,
		\end{aligned}
		\label{eqn:X_bound}
	\end{equation}
	where the energy norm $||\cdot||_{A_L}$ is defined in \eqref{eqn:energy_norm}. Using \eqref{eqn:X_bound} and summing the geometric series we get the desired bound on $\rev{\tau^{(\infty)}_{\mathrm{int},F_L}}$:
	\begin{equation}
		\begin{aligned}\rev{\tau^{(\infty)}_{\mathrm{int},F_L}} & =1+2\sum_{s=1}^{\infty}\frac{F_{L}^{\top}X_{L}^{s}A_{L}^{-1}F_{L}}{F_{L}^{\top}A_{L}^{-1}F_{L}}\le1+2\sum_{s=1}^{\infty}||X_{L}||_{A_{L}}^{s}=\frac{1+||X_{L}||_{A_{L}}}{1-||X_{L}||_{A_{L}}}.
		\end{aligned}
	\end{equation}
	To show the second inequality in \eqref{eqn:tau_int_bounds} first observe that according to 
	Lemma~\ref{lem:MGidentity-mean-cov}
	the IACT $\tau_{\mathrm{int},F_{L}}^{(m)}$ defined in \eqref{eqn:tau_int_definition} can be written as
	\begin{equation}
		\label{eqn:tau_m_formula}
		\tau_{\mathrm{int},F_{L}}^{(m)}=1+2\sum_{s=1}^{\infty}\frac{F_{L}^{\top}X_{L}^{s}\text{Cov}(\theta_{L}^{(m)},\theta_{L}^{(m)})F_{L}}{F_{L}^{\top}\text{Cov}(\theta_{L}^{(m)},\theta_{L}^{(m)})F_{L}}
	\end{equation}
	Introducing the shorthand
	\begin{equation}
		C_{L}^{(m)}:=\text{Cov}(\theta_{L}^{(m)},\theta_{L}^{(m)}),
	\end{equation}
	and defining the following (scalar) quantities
	\begin{xalignat*}{2}
		\alpha_{s} & :=F_{L}^{\top}X_{L}^{s}A_{L}^{-1}F_{L}, & \alpha'_{m,s} & :=F_{L}^{\top}X_{L}^{s}C_{L}^{(m)}F_{L},\\
		\beta & :=F_{L}^{\top}A_{L}^{-1}F_{L}, & \beta'_{m} & :=F_{L}^{\top}C_{L}^{(m)}F_{L},
	\end{xalignat*}
	we can write the sums in \eqref{eqn:tau_int_definition} and \eqref{eqn:tau_int} in compact form as
	\begin{xalignat*}{2}
		\rev{\tau^{(\infty)}_{\mathrm{int},F_L}}       & =1+2\sum_{s=1}^{\infty}\frac{\alpha_{s}}{\beta},                                                                                                                                                         &
		\tau_{\mathrm{int},F_{L}^{\top}}^{(m)} & =1+2\sum_{s=1}^{\infty}\frac{\alpha'_{m,s}}{\beta'_{m}}                            =\rev{\tau^{(\infty)}_{\mathrm{int},F_L}}+2\sum_{s=1}^{\infty}\left(\frac{\alpha'_{m,s}}{\beta'_{m}}-\frac{\alpha_{s}}{\beta}\right).
	\end{xalignat*}
	We now bound the terms in the sum. Each individual term can be re-written
	as
	\begin{equation}
		\begin{aligned}\frac{\alpha'_{m,s}}{\beta'_{m}}-\frac{\alpha_{s}}{\beta} & =\frac{\beta}{\beta_m'}\left(\frac{\alpha'_{m,s}-\alpha_{s}}{\beta}-\frac{\alpha_{s}}{\beta}\cdot\frac{\beta'_{m}-\beta}{\beta}\right)\end{aligned}
		\label{eqn:difference_bound}
	\end{equation}
	Hence, we need to look at $(\beta'_{m}-\beta)/\beta$ and $(\alpha'_{m,s}-\alpha_{s})/\beta$.
	We know from Lemma~\ref{lem:MGidentity-mean-cov} that
	\begin{equation}
		C_{L}^{(m)}-A_{L}^{-1}=X_{L}^{m}(C_{L}^{(0)}-A_{L}^{-1})(X_{L}^\top)^{m}
		\label{eqn:C_Lm}
	\end{equation}
	holds. Hence, with $\ensuremath{w:=A_{L}^{-1/2}F_{L}}$ and since $\left|\left|A_L^{1/2}X_L^m A_L^{-1/2}\right|\right|_2=\left|\left|A_L^{-1/2}(X_L^\top)^m A_L^{1/2}\right|\right|_2=||X_L||^m_{A_L}$
	\begin{equation}
		\begin{aligned}\left|\frac{\beta'_{m}-\beta}{\beta}\right| & =\left|\frac{F_{L}^{\top}(C_{L}^{(m)}-A_{L}^{-1})F_{L}}{F_{L}^{\top}A_{L}^{-1}F_{L}}\right|                                             =\left|\frac{w^{\top}A_{L}^{1/2}(C_{L}^{(m)}-A_{L}^{-1})A_{L}^{1/2}w}{w^{\top}w}\right|  \le\left|\left|A_{L}^{1/2}(C_{L}^{(m)}-A_{L}^{-1})A_{L}^{1/2}\right|\right|_{2}                                                                                                                                         \\ &= \left|\left|
               \left(A_L^{1/2} X_L^m A_L^{-1/2}\right) \left(A_L^{1/2} (C_L^{(0)}-A_L^{-1})A_L^{1/2}\right) \left(A_L^{-1/2}(X_L^\top)^m A_L^{1/2}\right)
               \right|\right|_2
               \\                                             &  \le \left|\left|\id - A_L^{1/2}C_{L}^{(0)}A_L^{1/2}\right|\right|_2 ||X_{L}||_{A_{L}}^{2m}\le (1+C_0) ||X_{L}||_{A_{L}}^{2m}
		\end{aligned}
		\label{eqn:beta_bound}
	\end{equation}
	\begin{equation}
		\begin{aligned}\left|\frac{\alpha'_{m,s}-\alpha_{s}}{\beta}\right| & =\left|\frac{F_{L}^{\top} X_{L}^{s}(C_{L}^{(m)}-A_{L}^{-1})F_{L}}{F_{L}^{\top}A_{L}^{-1}F_{L}}\right|   =\left|\frac{w^{\top}A_{L}^{1/2}X_{L}^{s}(C_{L}^{(m)}-A_{L}^{-1})A_{L}^{-1/2}A_{L}w^{\top}}{w^{\top}w}\right|                                   \\
                                                                   & \le\left|\left|A_{L}^{1/2}X_{L}^{s}A_{L}^{-1/2}A_{L}^{1/2}(C_{L}^{(m)}-A_{L}^{-1})A_{L}^{-1/2}A_{L}\right|\right|_{2}                                                                                                                  \\
                                                                   & =        \left|\left|
               \left(A_L^{1/2} X_L^s A_L^{-1/2}\right)\left(A_L^{1/2} X_L^m A_L^{-1/2}\right) \left(A_L^{1/2} (C_L^{(0)}-A_L^{-1})A_L^{1/2}\right) \left(A_L^{-1/2}(X_L^\top)^m A_L^{1/2}\right)
               \right|\right|_2                                                                                                                                                                                                                                                                             \\
                                                                   & \le\left|\left|\id - A_L^{1/2}C_{L}^{(0)}A_L^{1/2}\right|\right|_2||X_{L}||_{A_{L}}^{2m+s}                                                                                                         \le(1+C_0)||X_{L}||_{A_{L}}^{2m+s}.
		\end{aligned}
		\label{eqn:alpha_bound}
	\end{equation}
	Similar arguments show that $|\alpha_s/\beta| \le ||X_L||^s_{A_L}$.

	To bound $\beta/\beta_m'$ from above, consider the inverse which can be bounded from below as follows since $C_L^{(0)}$ is positive semi-definite, and we can rewrite $C_L^{(m)}$ with the help of \eqref{eqn:C_Lm}:
	\begin{equation}
		\begin{aligned}
			\frac{\beta_m'}{\beta} & = \frac{F_{L}^{\top} C_L^{(m)}F_{L}}{F_{L}^{\top}A_L^{-1}F_{L}}
			= \frac{F_{L}^{\top} \left(A_L^{-1}+X_L^m C_L^{(0)}(X_L^\top)^m - X_L^m A_L^{-1}(X_L^\top)^m\right)F_{L}}{F_{L}^{\top}A_L^{-1}F_{L}}                                \\
			                       & = 1 - \frac{F_{L}^{\top}X_L^m A_L^{-1} (X_L^\top)^m F_{L}}{F_{L}^{\top}A_L^{-1}F_{L}} +
			\underbrace{\frac{F_{L}^{\top}X_L^m C_L^{(0)} (X_L^\top)^m F_{L}}{F_{L}^{\top}A_L^{-1}F_{L}}}_{\ge 0}                                                               \\
			                       & \ge 1 - \frac{w^\top A_L^{1/2} X_L^m A_L^{-1}(X_L^\top)^m A_L^{1/2}w}{w^\top w}\qquad\text{with $w:= A_L^{-1/2}F_{L}$}      \\
			                       & = 1 - \frac{w^\top \left(A_L^{1/2} X_L A_L^{-1/2}\right)^m \left((A_L^{1/2}X_L A_L^{-1/2})^\top\right)^m A_L^{1/2}w}{w^\top w}  \ge 1-||X_L||_{A_L}^{2m}.
		\end{aligned}\label{eqn:beta_m_beta_lower_bound}
	\end{equation}
	We can now use the bounds in \eqref{eqn:X_bound}, \eqref{eqn:beta_bound}, \eqref{eqn:alpha_bound} and \eqref{eqn:beta_m_beta_lower_bound} to bound the left-hand side of
	(\ref{eqn:difference_bound}) from above:
	\begin{equation}
		\begin{aligned}\left|\frac{\alpha'_{m,s}}{\beta'_{m}}-\frac{\alpha_{s}}{\beta}\right| & \le \frac{\beta}{\beta'_{m}} \left(\left|\frac{\alpha'_{m,s}-\alpha_{s}}{\beta}\right|+\left|\frac{\alpha_{s}}{\beta}\right|\cdot\left|\frac{\beta'_{m}-\beta}{\beta}\right|\right)  \le2(1+C_0) \frac{||X_{L}||_{A_{L}}^{2m+s}}{1-||X_L||_{A_L}^{2m}}\qquad\text{for $m\ge 1$}.
		\end{aligned}
	\end{equation}
	With this we get the desired result:
	\begin{equation}
		\begin{aligned}|\tau_{\mathrm{int},F_{L}}^{(m)}-\rev{\tau^{(\infty)}_{\mathrm{int},F_L}}| & \le2\sum_{s=1}^{\infty}\left|\frac{\alpha'_{m,s}}{\beta'_{m}}-\frac{\alpha_{s}}{\beta}\right|                                           \le4(1+C_0)\frac{||X_{L}||_{A_{L}}^{2m}}{1-||X_L||_{A_L}^{2m}}\sum_{s=1}^{\infty}||X_{L}||_{A_{L}}^{s} \\
                                                                       & =C\frac{||X_{L}||_{A_{L}}^{2m+1}}{(1-||X_{L}||_{A_{L}})(1-||X_L||_{A_L}^{2m})}  \qquad\text{for $m\ge 1$},
		\end{aligned}\label{eqn:tau_difference_bound}
	\end{equation}
	with the grid-independent constant $C:= 4(1+C_0)$. We conclude that $\tau_{\mathrm{int},F_{L}}^{(m)}$ converges to $\rev{\tau^{(\infty)}_{\mathrm{int},F_L}}$ for $m\rightarrow\infty$.
\end{proof}
\subsection{Bounds on the Root-mean-squared error}\label{sec:proof_RMS_error}
To show Theorem~\ref{thm:integration-error}, we note that the mean-squared-error
is decomposed into
\begin{equation}
	\mathbb{E}[|I(F_{L})-\tilde{I}_{M}(F_{L})|^{2}]=[\mathrm{Bias}(\tilde{I}_{M}(F_{L}))]^{2}+\mathrm{Var}(\tilde{I}_{M}(F_{L})),\label{eq:MSE-decomp}
\end{equation}
with
$\mathrm{Bias}(\tilde{I}_{M}(F_{L}))  =I(F_{L})-\mathbb{E}\bigl[\tilde{I}_{M}(F_{L})\bigr]$ and $\mathrm{Var}(\tilde{I}_{M}(F_{L})) =\mathbb{E}\bigl[\bigl(\tilde{I}_{M}(F_{L})-\mathbb{E}[\tilde{I}_{M}(F_{L})]\bigr)^{2}\bigr]$.

We bound each term. To bound the bias we first show the following
strong convergence result.
\begin{lem}
	\label{lem:L2-limit}Suppose $\|X_{L}\|_{A_{L}}<1$. Then, the sequence
	$(\theta_{L}^{(m)})_{m\in\mathbb{N}}$ admits an  $L^{2}(\Omega;\mathbb{R}^{n_{L}})$-limit
	$\theta_{L}^{\infty}$. This limit is a Gaussian random variable with
	mean $A_{L}^{-1}f_{L}$ and covariance $A_{L}^{-1}$.
\end{lem}

\begin{proof}
	From $\theta_{L}^{(m)}=X_{L}^{m}\theta_{L}^{(0)}+\sum_{j=0}^{m-1}X_{L}^{j}(Y_{L}f_{L}+W_{L}^{(m-1-j)})$,
	for all $m'>m$ we have
	\begin{align*}
		\| & A_{L}^{1/2}(\theta_{L}^{(m')}-\theta_{L}^{(m)})\|_{L^{2}(\Omega;\mathbb{R}^{n_{L}})}                                                                                                                                                                     \\
		   & \leq\|X_{L}^{m'}-X_{L}^{m}\|_{A_{L}}\|A_{L}^{1/2}\theta_{L}^{(0)}\|_{L^{2}(\Omega;\mathbb{R}^{n_{L}})}+\sum_{j=m}^{m'-1}\|X_{L}\|_{A_{L}}^{j}(\|A_{L}^{1/2}Y_{L}f_{L}\|_{2}+\|A_{L}^{1/2}W_{L}^{(m'-1-j)}\|_{L^{2}(\Omega;\mathbb{R}^{n_{L}})}) \\
		   & =\|X_{L}^{m'}-X_{L}^{m}\|_{A_{L}}\|A_{L}^{1/2}\theta_{L}^{(0)}\|_{L^{2}(\Omega;\mathbb{R}^{n_{L}})}+(\|A_{L}^{1/2}Y_{L}f_{L}\|_{2}+\sqrt{\mathrm{tr}(A_{L}^{1/2}\WCov_{L}A_{L}^{1/2})})\sum_{j=m}^{m'-1}\|X_{L}\|_{A_{L}}^{j}.
	\end{align*}
	Provided $\|X_{L}\|_{A_{L}}<1$, the sequences $(X_{L}^{m})_{m\in\mathbb{N}}$
	and $\sum_{j=m}^{m'-1}\|X_{L}\|_{A_{L}}^{j}$ are convergent,
	and thus
	$(\theta_{L}^{(m)})_{m\in\mathbb{N}}$ is a Cauchy sequence in $L^{2}(\Omega;\mathbb{R}^{n_{L}})$.
	Thus, there is n $L^{2}(\Omega;\mathbb{R}^{n_{L}})$-limit $\theta_{L}^{\infty}$.
	The sequence $(\theta_{L}^{(m)})_{m\in\mathbb{N}}$ must be convergent
	to $\theta_{L}^{\infty}$ in distribution as well, but since $(\theta_{L}^{(m)})_{m\in\mathbb{N}}$
	is also convergent to $\mathcal{N}(A_{L}^{-1}f_{L},A_{L}^{-1})$
	as in Corollary~\ref{cor:conv-dist}, we conclude $\theta_{L}^{\infty}\sim\mathcal{N}(A_{L}^{-1}f_{L},A_{L}^{-1})$.
\end{proof}
Now we are ready to show Theorem~\ref{thm:integration-error}.
\begin{proof}[Proof of Theorem~\ref{thm:integration-error}]
	First, we bound the bias in \eqref{eq:MSE-decomp}. With $\theta_{L}^{\infty}$
	obtained in Lemma~\ref{lem:L2-limit}, note that $I(F_{L})=\int_{\Omega}F_{L}^{\top}\theta_{L}^{\infty}(\omega)\mathrm{d}\mathbb{P}(\omega)$
	holds. Hence, noting $I(F_{L})=F_{L}^{\top}A_{L}^{-1}f_{L}$ and
	$\mathbb{E}[\tilde{I}_{M}(F_{L})]=\frac{1}{M}\sum_{m=1}^{M}F_{L}^{\top}\mathbb{E}[\theta_{L}^{(m)}]$,
	Lemma~\ref{lem:MGidentity-mean-cov} yields the bias bound
	\begin{align*}
		|I(F_{L})-\mathbb{E}[\tilde{I}_{M}(F_{L})]| & \leq\frac{1}{M}\sum_{m=1}^{M}\|F_{L}^{\top}A_{L}^{-1/2}\|_{2}\|A_{L}^{1/2}(A_{L}^{-1}f_{L}-\mathbb{E}[\theta_{L}^{(m)}])\|_{2}                       \\
		                                            & \leq\frac{1}{M}\|F_{L}^{\top}A_{L}^{-1/2}\|_{2}\|A_{L}^{1/2}(A_{L}^{-1}f_{L}-\mathbb{E}[\theta_{L}^{(0)}])\|_{2}\sum_{m=1}^{M}\|X_{L}\|_{A_{L}}^{m}.
	\end{align*}

	To bound $\|F_{L}^{\top}A_{L}^{-1/2}\|_{2}$, we note that from \eqref{eqn:F_matrix}
	we have $F_{L}^{\top}A_{L}^{-1}F_{L}=\mathcal{F}(P_{L}A_{L}^{-1}F_{L})$, which implies $\|F_{L}^{\top}A_{L}^{-1/2}\|_{2}^{2}=|F_{L}^{\top}A_{L}^{-1}F_{L}|\leq\|\mathcal{F}\|_{H\to\mathbb{R}}\|\psi\|_{H}$
	with $\psi:=P_{L}A_{L}^{-1}F_{L}$. But $\psi$ is the solution to
	the following problem: find $\psi\in V_{L}$ such that
	\[
		a(\psi,v_{L})=\mathcal{F}(v_{L})\text{ for any }v_{L}\in V_{L},
	\]
	and thus $\left\Vert \psi\right\Vert _{V}\leq\frac{\|\mathcal{F}\|_{H\rightarrow\mathbb{R}}}{\sqrt{\lambda_{\min}(\mathcal{A})}}$.
	Hence, $\|F_{L}^{\top}A_{L}^{-1/2}\|_{2}^{2}\leq\frac{\|\mathcal{F}\|_{H\rightarrow\mathbb{R}}^{2}}{\lambda_{\min}(\mathcal{A})}$.
	An analogous argument can be used to bound $\|A_{L}^{1/2}(A_{L}^{-1}f_{L}-\mathbb{E}[\theta_{L}^{(0)}])\|_{2}$;
	indeed, noting $f_{L}^{\top}A_{L}^{-1}f_{L}=\langle f,P_{L}A_{L}^{-1}f_{L}\rangle_{H}$,
	we have $\|A_{L}^{-1/2}f_{L}\|_{2}^{2}\leq\frac{\|f\|_{H}^{2}}{\lambda_{\min}(\mathcal{A})}$,
	so that
	$
		\|A_{L}^{1/2}(A_{L}^{-1}f_{L}-\mathbb{E}[\theta_{L}^{(0)}])\|_{2}\leq\|A_{L}^{-1/2}f_{L}\|_{2}+C_{0}\leq\frac{\|f\|_{H}}{\sqrt{\lambda_{\min}(\mathcal{A})}}+C_{0}
	$.
	Together with $\|X_{L}\|_{A_{L}}\leq q<1$ with an $L$-independent
	constant $q$, we have
	\[
		|I(F_{L})-\mathbb{E}[\tilde{I}_{M}(F_{L})]|\leq\frac{C}{M},
	\]
	with $C:=\frac{\|\mathcal{F}\|_{H\rightarrow\mathbb{R}}}{\sqrt{\lambda_{\min}(\mathcal{A})}}\bigg(\frac{\|f\|_{H}}{\sqrt{\lambda_{\min}(\mathcal{A})}}+C_{0}\bigg)\sum_{m=1}^{\infty}q^{m}$.

	For the variance, since $F_{L}$ is a vector we have
	\begin{align}
		\mathrm{Var}(\tilde{I}_{M}(F_{L})) &                                                                                                                                                                                        
		=\frac{1}{M^{2}}\sum_{m=1}^{M}F_{L}^{\top}\mathrm{Cov}(\theta^{(m)})F_{L}+\frac{2}{M^{2}}\sum_{m=1}^{M}\sum_{k>m}^{M}F_{L}^{\top}\mathrm{Cov}(\theta^{(k)},\theta^{(m)})F_{L}\nonumber                                      \\
		                                   & =\frac{F_{L}^{\top}C_{L}^{(m)}F_{L}}{M^{2}}\sum_{m=1}^{M}\Bigl(\tau_{\mathrm{int},F_{L}}^{(m)}-2\sum_{s=M+1}^{\infty}\frac{\alpha_{m,s}^{\prime}}{\beta'_{m}}\Bigr),\label{eq:var-mid}
	\end{align}
	where we used \eqref{eqn:tau_int_definition} and notations $\alpha_{m,s}^{\prime}=F_{L}^{\top}X_{L}^{s}\mathrm{Cov}(\theta^{(m)})F_{L}$
	and $\beta'_{m}=F_{L}^{\top}\mathrm{Cov}(\theta^{(m)})F_{L}$. Here,
	we note $\beta'_{m}>0$ for all $m\in\mathbb{N}$ under $\left\Vert X_{L}\right\Vert _{A_{L}}<1$,
	which can be checked using $C_{L}^{(m)}=A_{L}^{-1}+X_{L}^{m}C_{L}^{(0)}(X_{L}^{\top})^{m}-X_{L}^{m}A_{L}^{-1}(X_{L}^{\top})^{m}$;
	see the proof of Lemma~\ref{lem:MGidentity-mean-cov}.

	Now, following the proof of Theorem~\ref{thm:IACT},
	using $|F_{L}^{\top}A_{L}^{-1}F_{L}|\leq\frac{\|\mathcal{F}\|_{H\rightarrow\mathbb{R}}^{2}}{\lambda_{\min}(\mathcal{A})}$ we see
	\[
		|F_{L}^{\top}C_{L}^{(m)}F_{L}|\leq\biggl(\left(1+C_{0}\right)\left\Vert X_{L}\right\Vert _{A_{L}}^{2m}+1\biggr)\frac{\|\mathcal{F}\|_{H\rightarrow\mathbb{R}}^{2}}{\lambda_{\min}(\mathcal{A})}.
	\]
	For $\tau_{\mathrm{int},F_{L}}^{(m)}-2\sum_{s=M+1}^{\infty}\frac{\alpha_{m,s}^{\prime}}{\beta'_{m}}$
	in (\ref{eq:var-mid}), from Theorem~\ref{thm:IACT} we have
	\begin{align*}
		\tau_{\mathrm{int},F_{L}}^{(m)} & \leq\frac{1+\|X_{L}\|_{A_{L}}}{1-\|X_{L}\|_{A_{L}}}+C\frac{\|X_{L}\|_{A_{L}}^{2m+1}}{(1-\|X_{L}\|_{A_{L}})(1-\|X_{L}\|_{A_{L}}^{2m})}  \leq\frac{1+q}{1-q}+C\frac{q^{3}}{(1-q)(1-q^{2})}
	\end{align*}
	and $\left|\frac{\alpha_{m,s}^{\prime}}{\beta'_{m}}\right|\leq q^{s}+2(1+C_{0})\frac{q^{2+s}}{(1-q)(1-q^{2})}.$
	Altogether, $\mathrm{Var}(\tilde{I}_{M}(F_{L}))$ can be bounded as
	$\mathrm{Var}(\tilde{I}_{M}(F_{L}))\leq\frac{C\|\mathcal{F}\|_{H\rightarrow\mathbb{R}}^{2}}{M}$,
	where $C$ depends on $q<1$, $\lambda_{\min}(\mathcal{A})$, and
	$C_{0}$ but is independent of $L$.

	From \eqref{eq:MSE-decomp} the statement follows.
\end{proof}
\subsection{Approximation property for perturbed matrix}\label{sec:proof_perturbed_matrix}
\begin{proof}[Proof of Proposition~\ref{prop:approximation-property}]
	Let $f_{\ell}\in\mathbb{R}^{n_{\ell}}$ be given. For
	$\widetilde{A}_{\ell}=A_{\ell}+B_{\ell}\Gamma^{-1}B_{\ell}^{\top}$ and
	$\widetilde{A}_{\ell-1}=A_{\ell-1}+B_{\ell-1}\Gamma^{-1}B_{\ell-1}^{\top}$, let
	\begin{align*}
		u_{\ell} & :=\widetilde{A}_{\ell}^{-1}f_{\ell}\in\mathbb{R}^{n_{\ell}}
		\quad\text{and}\quad
		u_{\ell-1}:=\widetilde{A}_{\ell-1}^{-1}I_{\ell}^{\ell-1}f_{\ell}\in\mathbb{R}^{n_{\ell-1}}.
	\end{align*}
	With $u_{\ell}$ and $u_{\ell-1}$, let $\tilde{u}_{\ell}:=P_{\ell}u_{\ell}$
	and $\tilde{u}_{\ell-1}:=P_{\ell-1}u_{\ell-1}$, so that
	\begin{align}
		\|(\widetilde{A}_{\ell}^{-1}- & I_{\ell-1}^{\ell}\widetilde{A}_{\ell-1}^{-1}I_{\ell}^{\ell-1})f_{\ell}\|_{2}  =\|u_{\ell}-I_{\ell-1}^{\ell}u_{\ell-1}\|_{2} =\|P_{\ell}^{-1}\tilde{u}_{\ell}-I_{\ell-1}^{\ell}P_{\ell-1}^{-1}\tilde{u}_{\ell-1}\|_{2} \\
		                              & =\|P_{\ell}^{-1}(\tilde{u}_{\ell}-\tilde{u}_{\ell-1})\|_{2}
		\leq c(\|\tilde{u}_{\ell}-u\|_{H}+\|u-\tilde{u}_{\ell-1}\|_{H})/\Phi(\ell),\label{eq:diff-in-uell-uellminus1}
	\end{align}
	where we used Assumption~\ref{assu:Pell}.

	To use (\ref{eq:perturbed-H-W-bound}) we rewrite the equation $\widetilde{A}_{\ell}u_{\ell}=f_{\ell}$
	in $\mathbb{R}^{n_{\ell}}$ and $\widetilde{A}_{\ell-1}u_{\ell-1}=I_{\ell}^{\ell-1}f_{\ell}$
	in $\mathbb{R}^{n_{\ell-1}}$ in the variational problems in function
	spaces $V_{\ell}$ and $V_{\ell-1}$.
	Let $(P_{\ell}^{-1})^{*}$ be
	the adjoint of $P_{\ell}^{-1}\colon(V_{\ell},\langle\cdot,\cdot\rangle_{H})\to\mathbb{R}^{n_{\ell}}$,
	i.e., $\langle(P_{\ell}^{-1})^{*}x_{\ell},y_{\ell}\rangle_{H}=x_{\ell}^{\top}P_{\ell}^{-1}y_{\ell}$
	for all $x_{\ell}\in\mathbb{R}^{n_{\ell}}$, $y_{\ell}\in V_{\ell}$,
	and define $F_{\ell}(\varphi):=\langle(P_{\ell}^{-1})^{*}f_{\ell},\varphi\rangle_{H}$
	for $\varphi\in V$. Then, the function $\tilde{u}_{\ell}=P_{\ell}u_{\ell}$
	satisfies
	\begin{align*}
		a(z_{\ell},\tilde{u}_{\ell})+b(z_{\ell},\tilde{u}_{\ell}) & =f_{\ell}^{\top}P_{\ell}^{-1}z_{\ell}=\langle(P_{\ell}^{-1})^{*}f_{\ell},z_{\ell}\rangle_{H}  =F_{\ell}(z_{\ell})\qquad\text{ for all }z_{\ell}\in V_{\ell},
	\end{align*}
	and $\tilde{u}_{\ell-1}=P_{\ell-1}u_{\ell-1}$ satisfies
	\begin{align*}
		a(z_{\ell-1},\tilde{u}_{\ell-1})+b(z_{\ell-1},\tilde{u}_{\ell-1}) & =(I_{\ell}^{\ell-1}f_{\ell})^{\top}P_{\ell-1}^{-1}z_{\ell-1}=(f_{\ell})^{\top}I_{\ell-1}^{\ell}P_{\ell-1}^{-1}z_{\ell-1} \\
		                                                  & =\langle(P_{\ell}^{-1})^{*}f_{\ell},z_{\ell-1}\rangle_{H}                                                                          =F_{\ell}(z_{\ell-1})\text{ for all }z_{\ell-1}\in V_{\ell-1}.
	\end{align*}
	Hence, $\tilde{u}_{\ell}$ and $\tilde{u}_{\ell-1}$ are approximations of the solution
	$u$ to the problem
	\[
		a(u,\varphi)+b(u,\varphi)=F_{\ell}(\varphi)\quad\text{for all}\quad\varphi\in V.
	\]
	We now use \eqref{eq:perturbedu-Wnorm-bound} and \eqref{eq:perturbed-H-W-bound} in Lemma~\ref{lem:approximation-perturbed} to
	\eqref{eq:diff-in-uell-uellminus1} to obtain
	\begin{align*}
		\|(\widetilde{A}_{\ell}^{-1}-I_{\ell-1}^{\ell}\widetilde{A}_{\ell-1}^{-1}I_{\ell}^{\ell-1})f_{\ell}\|_{2} & \leq c(\Psi(\ell)+\Psi(\ell-1))\|u\|_{W}/\Phi(\ell)                        \leq c\Psi(\ell-1)\|(P_{\ell}^{-1})^{*}f_{\ell}\|_{H}/\Phi(\ell).
	\end{align*}
	Finally, Assumption~\ref{assu:Pell} implies $\|P_{\ell}^{-1}\varphi_{\ell}\|_{2}\leq\frac{c_{2}}{\Phi(\ell)}\|\varphi_{\ell}\|_{H}$
	for $\varphi_{\ell}\in V_{\ell}$ and thus
	\[
		\|(P_{\ell}^{-1})^{*}\|_{\mathbb{R}^{n_{\ell}}\to V_{\ell}}=\|P_{\ell}^{-1}\|_{V_{\ell}\to\mathbb{R}^{n_{\ell}}}\leq\frac{c_{2}}{\Phi(\ell)},
	\]
	where $\|P_{\ell}^{-1}\|_{V_{\ell}\to\mathbb{R}^{n_{\ell}}}$ denotes
	the operator norm from $(V_{\ell},\langle\cdot,\cdot\rangle_{H})$
	to $\mathbb{R}^{n_{\ell}}$. We now invoke Proposition~\ref{prop:1/tildeAll-lb}
	to conclude
	\[
		\|(\widetilde{A}_{\ell}^{-1}-I_{\ell-1}^{\ell}\widetilde{A}_{\ell-1}^{-1}I_{\ell}^{\ell-1})\|_{2}\leq c\Psi(\ell-1)/(\Phi(\ell))^{2}\leq\frac{C}{\|\widetilde{A_{\ell}}\|_{2}}.
	\]
\end{proof}
\subsection{Optimality of MGMC}\label{sec:proof_optimality}
\begin{proof}[Proof of Corollary~\ref{cor:mgmc_optimality}]
	Statement 1 follows immediately from the cost analysis in Theorem \ref{thm:mgmc_cost} and Statement 2 is a consequence of Theorem \ref{thm:IACT}. Theorem \ref{thm:moments-conv} implies Statement 3. To see this, first observe that $\|F_L^\top \widetilde{A}_L^{-1/2}\|_2$ can be bounded by a constant that is independent of $L$. This can be shown with the same techniques as in the proof of Theorem~\ref{thm:integration-error}, which gives
\begin{equation}
	\|F_L^\top\widetilde{A}_L^{-1/2}\|_2 \le (\lambda_{\min}(\widetilde{\mathcal{A}}))^{-1/2}\|\mathcal{F}\|_{H\rightarrow \mathbb{R}}=:C_F
	\label{eqn:target_variance_bound}
\end{equation}
with finite $C_F$ independent of $n_L$ since $\mathcal{F}$ is bounded. Similarly, it can be shown that $\|\widetilde{A}_L^{-1/2}f_L\|_2$ is bounded independent of $n_L$. To see this, define the functional $\mathcal{G}:H\rightarrow \mathbb{R}$ with $\mathcal{G}(\phi) = \langle f,\phi\rangle_H = y^\top \Gamma^{-1}\mathcal{B}\phi$ for all $\phi\in H$. We can then derive the following bound (again using the techniques from the proof of Theorem~\ref{thm:integration-error}):
\begin{equation}
	\|\widetilde{A}_L^{-1/2}f_L\|_2 \le (\lambda_{\min}(\widetilde{\mathcal{A}}))^{-1/2}\|\mathcal{G}\|_{H\rightarrow \mathbb{R}} \le 
	(\lambda_{\min}(\widetilde{\mathcal{A}}))^{-1/2}\left(\lambda_{\min}(\Gamma)\right)^{-1}\|y\|_2\cdot \|\mathcal{B}\|_{H\rightarrow\mathbb{R}}=:C_f
	\label{eqn:target_mean_bound}
\end{equation}
with finite $C_f$ independent of $n_L$ since $\mathcal{B}$ is bounded. To show the exponential convergence of the expectation value observe that
\begin{equation}
	\begin{aligned}
		\left|\mathbb{E}[F_L^\top\theta_L^{(m)}]-F_L^\top\widetilde{A}_L^{-1}f_L\right| &= \left|(F_L^\top\widetilde{A}_L^{-1/2}) \widetilde{A}_L^{1/2}\left(\mathbb{E}[\theta_L^{(m)}]-\widetilde{A}_L^{-1}f_L\right)\right|\\
		&\leq \| F_L^\top\widetilde{A}_L^{-1/2}\|_2\cdot\|\mathbb{E}[\theta_L^{(m)}]-\widetilde{A}_L^{-1}f_L\|_{\widetilde{A}_L}\\
		&\leq C_F \left(\frac{C_A}{C_A+\nu}\right)^{m}\|\mathbb{E}[\theta_L^{(0)}]-\widetilde{A}_L^{-1}f_L\|_{\widetilde{A}_L}\le C_F (C_0+C_f)\left(\frac{C_A}{C_A+\nu}\right)^{m}.\label{eqn:optimality_convergence_mean}
	\end{aligned}
\end{equation}
	The exponential convergence of the variance can be shown similarly:
	\begin{equation}
	\begin{aligned}
		\left|\mathrm{Var}(F_L^\top\theta_{L}^{(m)})-F_L^\top\widetilde{A}_{L}^{-1}F_L\right| &= F_L^\top\widetilde{A}_L^{-1}F_L \frac{\left|F_L^\top\widetilde{A}_L^{-1/2}\left(\widetilde{A}_L^{1/2}\mathrm{Cov}(\theta_L^{(m)})\widetilde{A}_L^{1/2}-\id\right)\widetilde{A}_L^{-1/2}F_L\right|}{F_L^\top\widetilde{A}_L^{-1}F_L}\\
		&\leq \|F_L^\top \widetilde{A}_L^{-1/2}\|_2\cdot \|\widetilde{A}_L^{1/2}\mathrm{Cov}(\theta_L^{(m)})\widetilde{A}_L-\id\|_2\\
		&\leq C_F \left(\frac{C_A}{C_A+\nu}\right)^{2m}\|\widetilde{A}_L^{1/2}\mathrm{Cov}(\theta_L^{(0)})\widetilde{A}_L-\id\|_2\\
		&\le C_F(1+C_0) \left(\frac{C_A}{C_A+\nu}\right)^{2m}.\label{eqn:optimality_convergence_variance}
	\end{aligned}
\end{equation}
	From \eqref{eqn:optimality_convergence_mean} and \eqref{eqn:optimality_convergence_variance} we read off the constants $C_{1,1}:= C_F(C_0+C_f)$, $C_{1,2} = C_F(1+C_0)$ and $C_2 = \log(C_A+\nu)-\log(C_A)>0$. 
	
	Finally, noting that $u:=\widetilde{A}^{-1}f$ satisfies the problem $a(u,v)+b(u,v)=\langle f,v\rangle_{H}$
	for all $v\in H$, Assumption~\ref{assu:WtoH-best-Vell-approx} implies
	that $P_{L}u_{L}:=P_{L}\widetilde{A}_{L}^{-1}f_{L}$ with $f_{L}=(\langle f,\phi_{j}^{L}\rangle_{H})_{j=1,\dots,n_{L}}$
	converges to $u$ in $V$ as $L\to\infty$. Hence, 
	\[
	\lim_{L\to\infty}\mathbb{E}[F_{L}^{\top}\theta_{L}]=\lim_{L\to\infty}\mathcal{F}P_{L}\widetilde{A}_{L}^{-1}f_{L}=\mathcal{F}\widetilde{A}^{-1}f
	\]
	Similarly, $\lim_{L\to\infty}\mathrm{Cov}(F_{L}^{\top}\theta_{L})=\lim_{L\to\infty}\mathcal{F}P_{L}\widetilde{A}_{L}^{-1}F_{L}=\lim_{L\to\infty}\mathcal{F}\widetilde{A}^{-1}\mathcal{F}$, where we used $F_{L}=(\langle\mathcal{F},\phi_{j}^{L}\rangle_{H})_{j=1,\dots,n_{L}}$. 
	Thus, the characteristic function of $(F_{L}^{\top}\theta_{L}^{*})_{L}$
	converges to that of $(\widetilde{v},\mathcal{F})$ pointwise. Now
	the proof is complete.
\end{proof}
\section{Additional proofs and theorems}\label{appendix:general}
In the following we expand on a couple of condensed proofs in the main
text. 
\subsection{Symmetrised Random Smoother}\label{sec:proof_lem:symmetric_random_smoother}
\begin{proof}[Detailed proof of Lemma~\ref{lem:symmetric_random_smoother}]
	The combined update of the two individual smoothers calculates $\theta''$ from $\theta$ as follows:
	\begin{subequations}
		\begin{align}
			\theta'  & =\theta+M^{-1}(f+\xi_1-A \theta)\qquad\text{with $\xi_1\sim\mathcal{N}(0,M+M^{\top}-A)$}       \label{eqn:two_step_update_I}  \\
			\theta'' & =\theta^*+M^{-\top}(f+\xi_2-A \theta')\qquad\text{with $\xi_2\sim\mathcal{N}(0,M+M^{\top}-A)$} \label{eqn:two_step_update_II} \\
			         & = \theta +\left(M^{-1}+M^{-\top}-M^{-\top}AM^{-1}\right) (f-A\theta)
			+ \left(M^{-1}-M^{-\top}AM^{-1}\right)\xi_1 + M^{-\top} \xi_2.\nonumber
		\end{align}
	\end{subequations}
	Some simple algebra shows that $M^{-1}+M^{-\top}-M^{-\top}AM^{-1} = (M^{\mathrm{sym}})^{-1}$ with $M^{\mathrm{sym}}$ given in \eqref{eqn:symmetric_M}. Define the random variable $\xi$ as
	\begin{equation}
		(M^{\mathrm{sym}})^{-1} \xi := \left(M^{-1}-M^{-\top}AM^{-1}\right)\xi_1 + M^{-\top} \xi_2.
	\end{equation}
	$\xi$ depends linearly on $\xi_1,\xi_2\sim \mathcal{N}(0,M+M^{\top}-A)$ and obviously $\mathbb{E}[\xi]=0$. Since $\xi_1$, $\xi_2$ are independently drawn from normal distributions with known covariance matrices, some further tedious but straightforward algebra shows that $\mathbb{E}[\xi \xi^T] = 2M^{\mathrm{sym}}-A$ and therefore $\xi\sim \mathcal{N}(0,2M^{\mathrm{sym}}-A)$. We conclude that the two-step update in \eqref{eqn:two_step_update_I} and \eqref{eqn:two_step_update_II} is equivalent to the one-step update $\theta\mapsto \theta''$ with
	\begin{equation}
		\theta''   =\theta+(M^{\mathrm{sym}})^{-1}(f+\xi-A \theta)\qquad\text{with $\xi\sim\mathcal{N}(0,2M^{\mathrm{sym}}-A)$}.
	\end{equation}
\end{proof}
\subsection{Invariance of normal distribution under random smoothing}\label{sec:proof_prop:smoother_invariance}
\begin{proof}[Detailed proof of Proposition~\ref{prop:smoother_invariance}]
	Define
	\begin{xalignat}{2}
		z&:=\theta-\mu,&
		z'&:=\theta'-\mu=(I-M^{-1}A)z+M^{-1}\xi
	\end{xalignat}
	with $\mathbb{E}[z]=0$, $\mathbb{E}[zz^\top]=A^{-1}$. Then $\mathbb{E}[\theta']=\mu$ and $\mathbb{E}[(\theta'-\mu)(\theta'-\mu)^\top]=A^{-1}$
	is equivalent to $\mathbb{E}[z']=0$ and $\mathbb{E}[z'z'^\top]=A^{-1}$. To show the latter two identities we use the linearity of the expectation value and obtain
	\begin{equation}
		\mathbb{E}[z'] = (I-M^{-1}A)\mathbb{E}[z] + M^{-1}\mathbb{E}[\xi] = 0.
	\end{equation}
	Further, since $z$ and $\xi$ are independent random variables:
	\begin{equation}
		\begin{aligned}
			\mathbb{E}[z'z'^\top] & =(I-M^{-1}A)\mathbb{E}[zz^\top](I-AM^{-\top}) + M^{-1}\mathbb{E}[\xi\xi^\top]M^{-\top} \\
			                      & = (I-M^{-1}A)A^{-1}(I-AM^{-\top}) + M^{-1}(M+M^T-A)M^{-\top}                            = A^{-1}.
		\end{aligned}
	\end{equation}
\end{proof}
\subsection{Invertibility condition}\label{sec:Y-invertible}
\begin{lem}\label{lem:Y-invertible}
	Let $A$ be an SPD matrix and define the $||\cdot||_{A}$ norm $||X||_{A}:=||A^{1/2}XA^{-1/2}||_{2}$. Then for an arbitrary matrix $Y$ the inequality $||\id-YA||_{A}<1$ implies that $Y$ is invertible.
\end{lem}
\begin{proof}
	This can be shown by contradiction. Observe that
	\begin{equation}
		||\id-YA||_{A}=\left|\left|\id-A^{1/2}YA^{1/2}\right|\right|_{2}=\max_{u\ne0}\frac{\left|\left|(\id-A^{1/2}YA^{1/2})u\right|\right|_{2}}{||u||_{2}}\label{eqn:X_norm_bound}
	\end{equation}
	Assume that $Y$ is not invertible, then there is a $v\ne0$ such
	that $Yv=0$. Setting $u:=A^{-1/2}v$ we have that $(\id-A^{1/2}YA^{1/2})u=u$,
	so there exists a $u\ne0$ such that
	\begin{equation}
		\frac{\left|\left|(\id-A^{1/2}YA^{1/2})u\right|\right|_{2}}{||u||_{2}}=1,
	\end{equation}
	which contradicts the assumption $||\id-YA||_{A}<1$.
\end{proof}
\rev{\section{Derivation of additional complexity estimates}}
\rev{In this appendix we present detailed derivations of the MGMC setup costs in Section~\ref{sec:mgmc_setup_costs} and storage requirements in Section~\ref{sec:mgmc_memory_requirements}.}
\rev{\subsection{Derivation of MGMC setup costs}\label{sec:proof_mgmc_setup_costs}}
\rev{\begin{proof}[Proof of Theorem~\ref{thm:mgmc_setup_cost}]
    In the setup phase, the Gibbs-sampler in Alg. \ref{alg:low_rank_gibbs} requires $\mathcal{O}(\beta n_\ell)$ operations to solve the $\beta$ triangular systems in line 7. The computation of the dense $\beta\times \beta$ matrix $\Gamma+B_\ell^\top C_\ell$, its factorisation and the construction of the $n_\ell\times \beta$ matrix $G_\ell$ in line 8 incur costs of $\mathcal{O}(\beta^2n_\ell)$, $\mathcal{O}(\beta^3)$ and $\mathcal{O}(\beta^2n_\ell)$ respectivly. This leads to the following bound of the setup costs for Alg. \ref{alg:low_rank_gibbs}:
\begin{equation}
    \mathrm{Cost}^{\mathrm{(setup)}}_{\mathrm{Gibbs}}(n_\ell,\beta) \le C_{\mathrm{Gibbs}}^{(\mathrm{setup},1)} \beta^{3} + C_{\mathrm{Gibbs}}^{(\mathrm{setup},2)} \beta^2 n_\ell.\label{eqn:smoother_setup_cost}
\end{equation}
As explained in Section~\ref{subsec:Cost-analysis}, when applying the low-rank correction in the multiplication by $\widetilde{A}_\ell$, which is required in the residual calculation, it is advantageous to precompute the $\beta\times n_\ell$ matrix $\Gamma^{-1}B_\ell^\top$ at a cost of $\mathcal{O}(\beta^{p_\Gamma}n_\ell)$, which incurs additional costs of
\begin{equation}
    \mathrm{Cost}^{(\mathrm{setup})}_{\mathrm{other}}(\beta,n_\ell) \le C^{\mathrm{(setup)}}_{\mathrm{other}}\beta^{p_\Gamma}n_\ell\label{eqn:coarse_setup_other}
\end{equation}
on each level. The setup cost on the coarsest level can be bounded as
\begin{equation}
    \mathrm{Cost}^{(\mathrm{setup})}_{\mathrm{coarse}}(\beta,n_0) \le
    \begin{cases}
        C_{\mathrm{Gibbs}}^{(\mathrm{setup},1)} \beta^{3} + C_{\mathrm{Gibbs}}^{(\mathrm{setup},2)} \beta^2 n_0 & \text{(Alg.~\ref{alg:coarse_sampler})} \\
        \mathrm{Cost}^{(\mathrm{setup})}_{\mathrm{Cholesky}}(n_0)                & \text{(Cholesky sampler)}.
    \end{cases}
    \label{eqn:coarse_setup_cost}
\end{equation}
As in the proof of Theorem~\ref{thm:mgmc_cost}, we can bound the setup costs on each level $\ell$ of the hierarchy recursively
\begin{equation}
        \mathrm{Cost}^{(\mathrm{setup})}_{\mathrm{MGMC}}(\ell) %
        \le \begin{cases} C^{(\mathrm{setup})}_{\mathrm{other}} \beta^{p_\Gamma} n_\ell + \mathrm{Cost}^{(\mathrm{setup})}_{\mathrm{Gibbs}}(n_\ell,\beta) +\mathrm{Cost}^{(\mathrm{setup})}_{\mathrm{MGMC}}(\ell-1) & \text{for $\ell>0$} \\
              \mathrm{Cost}^{(\mathrm{setup})}_{\mathrm{coarse}}(\beta,n_0)                                                                                            & \text{for $\ell=0$}\end{cases}
    \end{equation}
Inserting the bounds in \eqref{eqn:smoother_setup_cost}, \eqref{eqn:coarse_setup_other} and \eqref{eqn:coarse_setup_cost} into this expression leads to 
\begin{equation}
    \begin{aligned}
    \mathrm{Cost}^{(\mathrm{setup})}_{\mathrm{MGMC}}(\ell) &\le C_{\mathrm{Gibbs}}^{(\mathrm{setup},1)} \beta^3L + \left(C_{\mathrm{other}}^{(\mathrm{setup})}\beta^{p_\Gamma}+C_{\mathrm{Gibbs}}^{(\mathrm{setup},2)}\beta^2\right) \sum_{\ell = 1}^{L} n_\ell + \mathrm{Cost}_{\mathrm{coarse}}^{(\mathrm{setup})}(\beta,n_0) \\
    &\le C_{\mathrm{MG}}^{(\mathrm{setup},1)} \beta^3 L+ C_{\mathrm{MG}}^{(\mathrm{setup},2)} \beta^2 n_L + \mathrm{Cost}_{\mathrm{coarse}}^{(\mathrm{setup})}(\beta,n_0)
    \end{aligned}
\end{equation}
where, since $p_\Gamma\le 2$ and $\rho_G<1$,
\begin{xalignat}{2}
	C_{\mathrm{MG}}^{(\mathrm{setup},1)} &= C_{\mathrm{Gibbs}}^{(\mathrm{setup},1)}, &
    C_{\mathrm{MG}}^{(\mathrm{setup},2)} &= \frac{C_{\mathrm{oher}}^{(\mathrm{setup})}+C_{\mathrm{Gibbs}}^{(\mathrm{setup},2)}}{1-\rho_G}.
\end{xalignat}
\end{proof}}
\rev{\subsection{Derivation of MGMC memory requirements}\label{sec:proof_mgmc_memory}}
\rev{\begin{proof}[Proof of Theorem~\ref{thm:mgmc_memory}]
On each level $\ell$ we need to store no more than $K$ vectors of size $n_\ell$, such as $\theta_\ell$, $f_\ell$, $\psi_\ell$ and $\xi_\ell^{\mathrm{diag}}$; the exact value of $K$ (which is independent of $\ell$) depends on details of the implementation. To apply Alg.~\ref{alg:low_rank_gibbs} and to compute the residual, we need to store the sparse matrix $A_\ell$, which also has $\mathcal{O}(n_\ell)$ entries. In addition, we require the storage of the vector $\xi_\ell^{\mathrm{LR}}$ of length $\beta$, as well as the $n_\ell\times \beta$ matrices $B_\ell$, $G_\ell$ and the $\beta\times n_\ell$ matrices $B_\ell^\top$, $\Gamma^{-1}B_\ell^\top$. This results in the following bound
\begin{equation}
	\mathrm{Memory}_{\mathrm{Gibbs}}(\ell) \le C_{\mathrm{Gibbs}}^{(\mathrm{mem})} (1+\beta)n_\ell + \beta
\end{equation}
for some constant $C_{\mathrm{Gibbs}}^{(\mathrm{mem})}$. On the coarsest level, we either need to store the quantities that are required for Alg.~\ref{alg:low_rank_gibbs} with $\ell=0$ or the Cholesky factorisation of the $n_0\times n_0$ matrix $\widetilde{A}_0$. This implies that
\begin{equation}
    \mathrm{Memory}_{\mathrm{coarse}}(\beta,n_0) \le
    \begin{cases}
        C_{\mathrm{MG}}^{(\mathrm{mem})} (1+\beta)n_0 + \beta & \text{(Alg.~\ref{alg:coarse_sampler})} \\
        \mathrm{Memory}_{\mathrm{Cholesky}}(n_0)                & \text{(Cholesky sampler)}.
    \end{cases}
\end{equation}
Putting everything together, the total storage cost for Alg.~\ref{alg:mgmc} can be bounded as follows:
\begin{equation}
	\begin{aligned}
		\mathrm{Memory}_{\mathrm{MGMC}}(L) &\le C_{\mathrm{Gibbs}}^{(\mathrm{mem})} (1+\beta) \sum_{\ell=1}^{L}n_\ell + \beta L + \mathrm{Memory}_{\mathrm{coarse}}(\beta,n_0)\\
		&\le  C_{\mathrm{Gibbs}}^{(\mathrm{mem})} (1+\beta) n_L \sum_{\ell=1}^L\rho_G^{L-\ell} + \beta L + \mathrm{Memory}_{\mathrm{coarse}}(\beta,n_0)\\
		&\le C_{\mathrm{MG}}^{(\mathrm{mem})} (1+\beta) n_L + \beta L+ \mathrm{Memory}_{\mathrm{coarse}}(\beta,n_0)
	\end{aligned}
\end{equation}
with
\begin{equation}
C_{\mathrm{MG}}^{(\mathrm{mem})} = \frac{C_{\mathrm{Gibbs}}^{(\mathrm{mem})}}{1-\rho_G}.
\end{equation}
\end{proof}}
\rev{\section{Additional numerical results\label{subsec:additional_results}}
While in Sec.~\ref{sec:results} we assumed that the entries of the covariance matrix $\Gamma$ lie in the interval $[\widehat{\sigma},2\widehat{\sigma}]$ with $\widehat{\sigma}=10^{-6}$, we also measured the integrated autocorrelation time (IACT) of the Gibbs- and MGMC sampler for a wide range of other values of $\widehat{\sigma}$ and for three different correlation lengths $\kappa^{-1}$. The results in Tab.~\ref{tab:vary_sigma_2d} were obtained for the FEM discretisation of the shifted Laplace operator $\mathcal{A}^{(\mathrm{SL})}=-\Delta + \kappa^2 I$ on a two-dimensional grid of size $128\times 128$. All other parameters are the same as described in Sec.~\ref{sec:setup}. As the table shows, the IACT is approximately independent of value of $\widehat{\sigma}$. While for the Gibbs sampler it increases as the inverse correlation length grows (this is consistent with the deteriorating convergence rates in Fig.~\ref{fig:robustness_convergence_gridindependence}), for the MGMC sampler the IACT is smaller than $1.6$ for all values of $\kappa^{-1}$ and $\widehat{\sigma}$ considered here.}
\begin{table}
    \begin{center}
        \rev{
\begin{tabular}{|c|rrr|rrr|}
    \hline
& \multicolumn{3}{c|}{Gibbs} & \multicolumn{3}{c|}{MGMC}\\
$\widehat{\sigma}$
 & $\kappa^{-1}=0.10$ & $\kappa^{-1}=0.25$ & $\kappa^{-1}=0.50$ & $\kappa^{-1}=0.10$ & $\kappa^{-1}=0.25$ & $\kappa^{-1}=0.50$\\\hline\hline
 $10^{-10}$ & $41.4 \pm 10.6$ & $59.6 \pm 17.4$ & $95.8 \pm 33.3$ & $1.20 \pm 0.13$ & $1.41 \pm 0.12$ & $1.42 \pm 0.13$ \\
 $10^{-8}$ & $38.0 \pm 9.4$ & $88.8 \pm 30.0$ & $68.9 \pm 21.2$ & $1.23 \pm 0.14$ & $1.36 \pm 0.11$ & $1.47 \pm 0.14$ \\
 $10^{-6}$ & $34.6 \pm 8.3$ & $58.5 \pm 17.0$ & $118.1 \pm 44.2$ & $1.24 \pm 0.14$ & $1.42 \pm 0.13$ & $1.59 \pm 0.12$ \\
 $10^{-4}$ & $47.0 \pm 12.6$ & $89.1 \pm 30.1$ & $69.9 \pm 21.6$ & $1.18 \pm 0.12$ & $1.39 \pm 0.12$ & $1.55 \pm 0.16$ \\
 $10^{-2}$ & $35.8 \pm 8.7$ & $74.7 \pm 23.6$ & $88.0 \pm 29.6$ & $1.23 \pm 0.14$ & $1.40 \pm 0.12$ & $1.42 \pm 0.13$ \\
 $1$ & $46.5 \pm 12.5$ & $102.8 \pm 36.6$ & $110.9 \pm 40.6$ & $1.13 \pm 0.11$ & $1.14 \pm 0.11$ & $1.15 \pm 0.12$ \\
 $10^{2}$ & $48.2 \pm 13.1$ & $107.8 \pm 39.1$ & $150.2 \pm 61.3$ & $1.13 \pm 0.11$ & $1.13 \pm 0.11$ & $1.15 \pm 0.11$ \\
 $10^{4}$ & $48.2 \pm 13.1$ & $107.8 \pm 39.1$ & $150.2 \pm 61.4$ & $1.13 \pm 0.11$ & $1.13 \pm 0.11$ & $1.15 \pm 0.11$ \\
\hline\end{tabular}
}
\caption{\rev{IACT for different values of the correlation length $\kappa^{-1}$ and magnitude $\widehat{\sigma}$ of the covariance matrix $\Gamma$. All results were obtained with the FEM discretisation of the shifted Laplace operator $\mathcal{A}^{(\mathrm{SL})}=-\Delta + \kappa^2 I$ on a grid of size $128\times 128$.}}
\label{tab:vary_sigma_2d}
\end{center}
\end{table}
\end{document}